\documentclass[11pt]{article}

\usepackage[hidelinks]{hyperref}
\hypersetup{
  colorlinks   = true, 
  urlcolor     = blue, 
  linkcolor    = blue, 
  citecolor   = red 
}

\usepackage{tikz}
\usetikzlibrary{shapes.geometric, arrows}
\tikzstyle{startstop} = [rectangle, rounded corners, minimum width=3cm, minimum height=1cm, text centered, text width= 7cm, draw=black, fill=cyan!30]
\tikzstyle{startstop'} = [rectangle, rounded corners, minimum width=3cm, minimum height=1cm, text centered, text width= 7cm, draw=black, fill=red!30]
\tikzstyle{startstop2} = [rectangle, rounded corners, minimum width=3cm, minimum height=1cm, text centered, text width= 10cm, draw=black, fill=cyan!30]
\tikzstyle{startstop2'} = [rectangle, rounded corners, minimum width=3cm, minimum height=1cm, text centered, text width= 10cm, draw=black, fill=red!30]
\tikzstyle{io} = [trapezium, trapezium left angle=70, trapezium right angle = 110, minimum width=3cm, minimum height=1cm, text centered, draw=black, fill=blue!30]
\tikzstyle{decision} = [diamond, minimum width=3cm, minimum height=1cm, text centered, draw=black, fill=green!30]
\tikzstyle{arrow} = [thick, ->, >=stealth]

\usepackage{amsmath}
\usepackage{amssymb}
\usepackage{amsfonts}
\usepackage{amsthm}
\usepackage{tikz,tikz-cd,float}
\usepackage{latexsym}
\usepackage{esint}
\usepackage{mathtools}
\usepackage{mathrsfs}
\usepackage{graphicx}
\usepackage{bbm}
\usepackage{color}
\usepackage{bm}
\usepackage[colorinlistoftodos]{todonotes}
\usepackage[top=1in, bottom=1.25in, left=1.10in, right=1.10in]{geometry}
\usepackage{enumerate}

\newcommand{\T}{\mathbb{T}^N}
\newcommand{\N}{\mathbb{N}}									

\newcommand{\R}{\mathbb{R}}

\newtheorem{theorem}{Theorem}[section]
\newtheorem{lemma}[theorem]{Lemma}
\newtheorem{proposition}[theorem]{Proposition}
\newtheorem{corollary}[theorem]{Corollary}

\theoremstyle{definition}
\newtheorem{definition}[theorem]{Definition}

\newtheorem{assumption}[theorem]{Assumption}

\theoremstyle{remark}
\newtheorem{remark}[theorem]{Remark}

\numberwithin{equation}{section}

\usepackage{stmaryrd}

\newcommand{\sy}{\boldsymbol{\Psi}}
\newcommand{\py}{\boldsymbol{\Phi}}

\newcommand{\inner}[2]{\langle #1, #2 \rangle}
\DeclarePairedDelimiter\norm{\lVert}{\rVert}
\DeclarePairedDelimiter\abs{\lvert}{\rvert}

\begin{document}

\title{Existence and Uniqueness of Maximal Solutions to SPDEs with Applications to Viscous Fluid Equations}

\author{Daniel Goodair, Dan Crisan, Oana Lang}


\maketitle




\begin{abstract}
We present two criteria to conclude that a stochastic partial differential equation (SPDE) posseses a unique maximal strong solution. This paper provides the full details of the abstract well-posedness results first given in \cite{goodair2022existence}, and partners the paper \cite{goodair2022inprep} which rigorously addresses applications to the 3D SALT (Stochastic Advection by Lie Transport, \cite{holm2015variational}) Navier-Stokes Equation in velocity and vorticity form, on the torus and the bounded domain respectively. Each criterion has its corresponding set of assumptions and can be applied to viscous fluid equations with additive, multiplicative or a general transport type noise.
\end{abstract}

\tableofcontents

\setcounter{page}{1}
\section{Introduction}
The theoretical analysis of fluid models perturbed by transport type noise has been in significant demand since the release of the seminal works \cite{holm2015variational} and \cite{memin2014fluid}. In these papers Holm and Memin establish a new class of stochastic equations driven by transport noise which serve as fluid dynamics models by adding uncertainty in the transport of the fluid parcels to reflect the unresolved scales. The significance of such equations in modelling, numerical schemes and data assimilation continues to be well documented, see (\cite{cotter2018modelling}, \cite{cotter2019numerically}, \cite{holm2020stochastic}, \cite{holm2019stochastic} \cite{street2021semi}, \cite{van2021bayesian}, \cite{crisan2021theoretical}, \cite{dufee2022stochastic}, \cite{lang2022pathwise}, \cite{cotter2020data}, \cite{flandoli20212d}, \cite{flandoli2022additive}, \cite{alonso2020modelling}). In contrast there has been limited progress in proving well-posedness for this class of equations: Crisan, Flandoli and Holm \cite{crisan2019solution} have shown the existence and uniqueness of maximal solutions for the 3D Euler Equation on the torus, whilst Crisan and Lang (\cite{crisan2019well},\cite{crisan2020local},\cite{crisan2021well}) extended the well-posedness theory for the Euler, Rotating Shallow Water and Great Lake Equations on the torus once more. Alonso-Or\'{a}n and Bethencourt de Le\'{o}n \cite{alonso2020well} show the same properties for the Boussinesq Equations again on the torus, whilst Brze\'{z}niak and Slav\'{i}k  \cite{brzezniak2021well} demonstrate these properties on a bounded domain for the Primitive Equations but for a specific choice of transport noise which facilitates their analysis. Indeed the theoretical analysis of fluid equations driven by a specifically chosen transport noise is well developed in the literature, see (\cite{luo2020scaling}, \cite{luo2021convergence}, \cite{flandoli2021scaling}, \cite{attanasio2011renormalized}, \cite{flandoli2021high}, \cite{flandoli2015open}).\\ 

The first step in developing the theoretical analysis of either stochastic or deterministic PDEs is well-posedness. The class of equations in consideration is ever expanding; Figure 2 of \cite{crisan2021theoretical} gives a brief overview of just some of the determinstic fluid models, each of which can be stochastically perturbed through a similarly widening array of variational principles (beyond the seminal works \cite{holm2015variational} and \cite{memin2014fluid}, see more recently \cite{holm2019stochastic} and \cite{street2021semi}). The significance of an abstract approach to the well-posedness question is clear, looking to encapsulate the similarities between these equations whilst working in as much generality as possible to incorporate their technical differences. We state our equation in the form
\begin{equation} \label{thespde}
    \sy_t = \sy_0 + \int_0^t \mathcal{A}(s,\sy_s)ds + \int_0^t\mathcal{G} (s,\sy_s) d\mathcal{W}_s
\end{equation}
for $\mathcal{W}$ a Cylindrical Brownian Motion and operators $\mathcal{A}$ and $\mathcal{G}$ which heuristically allow for nonlinearity with $\mathcal{A}$ a second order differential operator and $\mathcal{G}$ of first order, at the expense of some weak monotonicity and coercivity constraints (Subsection \ref{assumptionschapter}, \ref{subsection for assumptions 2}). We prove the existence, uniqueness and maximality of solutions to (\ref{thespde}) which are strong in both the probabilistic and PDE sense. Looking to work in great generality, we separate our results into two distinct criteria whereby we work under differnt sets of assumptions (Sections \ref{Section v valued} and \ref{h valued}). In Section \ref{Section v valued} we present the first criterion, requiring a triple of embedded Hilbert Spaces. The criterion is extended in Section \ref{h valued}, giving rise to solutions of the SPDE in a larger space. The results are applied to the SALT Navier-Stokes Equation in both its velocity and vorticity form, applying only the result of Section \ref{Section v valued} to the vorticity form but requiring Section \ref{h valued} for the velocity one, making explicit the insufficieny of just Section \ref{Section v valued} for solving the velocity equation in the optimal spaces. In the interests of brevity we simply sketch these applications (Section \ref{section applications}), deferring a complete treatment to the associated paper \cite{goodair2022inprep}.\\ 

The embedded Hilbert Spaces with respect to which we pose (\ref{thespde}) are not assumed to form a Gelfand Triple, a core difference between the current work and existing ones. Variational approaches in a Gelfand Triple for additive and multiplicative noise have long been studied, initially in the works (\cite{pardoux1975equations}, \cite{gyongy1980stochastic}, \cite{krylov2007stochastic}) and more recently (\cite{neelima2020coercivity}, \cite{liu2013well}, \cite{liu2010spde}, \cite{liu2013local}). Perhaps the most relevant papers come from Debussche, Glatt-Holtz and Temam \cite{debussche2011local} as well as R\"{o}ckner, Shang and Zhang \cite{rockner2022well}. In \cite{debussche2011local} the authors show existence, uniqueness and maximality for an abstract fluids model perturbed by a multiplicative noise, though this noise does not extend to the case of a differential operator. The work \cite{rockner2022well} was released just in recent weeks, in which the authors extend the framework of a variational approach in the Gelfand Triple to cover transport type noise. Their assumptions  achieve a global existence result, however, which naturally will not account for PDE-strong solutions of the 3D Navier-Stokes Equation and related models. Indeed the very setting of a Gelfand Triple lends itself to solutions which are weak in the analytic sense, though we appreciate that Liu and R\"{o}ckner \cite{liu2013local} prove the existence of a local strong solution to the incompressible 3D Navier-Stokes Equation with additive noise by working in the spaces $W^{2,2} \xhookrightarrow{}W^{1,2}\xhookrightarrow{} L^2$ with the understanding that the $(W^{2,2})^*$ norm is controlled by the $L^2$ one. To directly work with such a triplet of spaces, Kato and Lai \cite{kato1984nonlinear} prove a PDE existence result which they apply to the Euler Equation in the context of an 'admissable triplet' of spaces, where a bilinear form relation that reduces to the inner product of the middle space exists however there is no assumed duality structure. We shall work in this setting throughout the paper.\\

The quoted works to solve an abstract SPDE have two distinct mechanisms of proof: transformation of the SPDE to a random PDE, and energy estimates leading to relative compactness methods. Our departure from the generalised 'admissable triplet' framework in Section \ref{Section v valued} is representative of a new method, which we summarise here. This is in the context of embedded Hilbert spaces $V \xhookrightarrow{} H \xhookrightarrow{} U$, where:
\begin{itemize}
    \item $V$ is the space in which solutions have square integrability in time. This represents the most regular space and a sequence of finite-dimensional approximating solutions are constructed within it. This is referred to as a Galerkin Approximation.
    \item $H$ is the space in which solutions are continuous in time, and in which the initial condition takes value. Typically this space is where the analysis takes place, through a bilinear form reducing to the $H$ inner product where some weak coercivity assumption is used to generate the control in $V$. To ease the burden of a first order diffusion operator we conduct much of the analysis in $U$, only working in the $H$ inner product to generate the $V$ control for the Galerkin Approximations.
    \item $U$ is the space in which solutions satisfy their identity. The described limitations are circumvented by conducting analysis in the $U$ inner product, which is combined with the work in the $H$ inner product to achieve the required regularity of the solution. In particular a Cauchy property is shown in the $U$ inner product; this is an asset of our method as typically such a property must be shown in $H$, and in fact is necessary in handling the differential operator in the noise term. 
\end{itemize}

This is explained in more detail now. Our method takes inspiration from that of Glatt-Holtz and Ziane in the paper \cite{glatt2009strong}, where they deduce the existence of a local strong solution to the incompressible 3D Navier-Stokes Equation with multiplicative noise by considering a Galerkin Approximation where each finite dimensional solution is treated up to a first hitting time. Provided a Cauchy Property in the norm of the first hitting time is satisfied, alongside a uniform rate of convergence of the processes to their initial conditions, then a limiting process and almost surely positive stopping time can be inferred (restated in the Appendix, Lemma \ref{greenlemma}) which is then shown to be a solution as desired. In this paper the Cauchy Property is demonstrated for the $H$ inner product; to manage the differential operator in the noise we have to appeal to some uniform higher regularity of the solutions, which is not immediately evident from the first hitting times. Moreover our proof comes with the additional step of showing uniform boundedness generated in the $H$ inner product, enabling us to prove the Cauchy Property only for $U$. This summarises our existence approach for Section \ref{Section v valued}, where we consider so called $H-$valued solutions of (\ref{thespde}).\\

A different notion of solution ($U-$valued) is considered in Section \ref{h valued}, where the criterion of Section \ref{Section v valued} is extended. A fourth Hilbert Space $X$ ($U \xhookrightarrow{} X$) is introduced, with the spaces serving a different purpose for the $U-$valued solutions:
\begin{itemize}
    \item $V$ is superfluous in making the definition of $U-$valued solutions, though we again use an approximating sequence of solutions which must exist in this space.
    \item $H$ is the space in which solutions have square integrability in time. The approximating sequence of solutions are no longer in finite dimensions, but rather are solutions of the full equation with $H$ valued initial conditions. 
    \item $U$ is the space in which solutions are continuous in time, and in which the initial condition takes value. The approximating solutions correspond to solutions for $H$ valued initial conditions convergent to the $U$ valued one. A Cauchy property is shown in this space, which directly generates the $H$ control through a weak coercivity assumption. We are afforded the option to show the Cauchy property in this space as the differential operator in the noise term only prevented the Cauchy property in Section \ref{Section v valued} due to the difference in dimensionality of the approximating sequence. As alluded to in the discussion there, an 'admissable triplet' relation is assumed for the spaces $H,U,X$ to facilitate working in this space.
    \item $X$ is the space in which solutions satisfy their identity. The only analysis conducted in this space is for the uniqueness of solutions, and to justify that the limiting process obtained from the above Cauchy property is a solution. 
\end{itemize}
The idea (in line with \cite{glatt2012local}, \cite{glatt2011cauchy} for example) is to use the more regular $H-$valued solutions corresponding to a sequence of $H-$valued initial conditions which are convergent to a $U-$valued initial condition. This involves iterating the procedure of Section \ref{Section v valued} though in a much simplified way, as we do not need any uniform higher regularity to justify the Cauchy property which is now shown directly in the space of existence. This simplification is again owing to the fact that the transport type noise only proves problematic when dealing with the difference of the finite dimensional projection operators. \\

The text is structured as follows. In Section \ref{section prelims} we simply introduce some notation and establish the stochastic framework. Section \ref{Section v valued} is where we prove the existence of $H-$valued local strong solutions as described, in addition to their uniqueness and maximality. The uniqueness is proved by simply showing that the expectation of the norm of the difference of two solutions is null. As for maximality, we apply Zorn's Lemma at an abstract level to show the existence of a maximal solution. The maximal time is then characterised as the blow up by showing that at the minimum between this maximal time and any first hitting time in the required norm, then a local solution exists and can be extended: thus the first hitting time must be smaller than the maximal time, as by definition the maximal time cannot be extended. Section \ref{h valued} follows the same path of existence, uniqueness and maximality for the $U-$valued solutions as outlined. We give the application to the SALT Navier-Stokes Equation in both its velocity and vorticity forms in Section \ref{section applications}, again making explicit the necessity of having both Sections \ref{Section v valued} and \ref{h valued} and their distinction. To reach the main results more efficiently some of the proofs are deferred to the Appendix (Subection \ref{proofs section 3 appendix}), which also contains a few key results from the literature that we apply in the text and concludes this paper (Subsection \ref{subsection useful results}).

\section{Preliminaries} \label{section prelims}

In the following $\mathcal{O}$ represents a subset of $\R^N$. Throughout the paper we consider Banach Spaces as measure spaces equipped with the Borel $\sigma$-algebra, and will on occassion use $\lambda$ to represent the Lebesgue Measure.

\begin{definition} \label{definition of spaces}
Let $(\mathcal{X},\mu)$ denote a general measure space, $(\mathcal{Y},\norm{\cdot}_{\mathcal{Y}})$ and $(\mathcal{Z},\norm{\cdot}_{\mathcal{Z}})$ be Banach Spaces, and $(\mathcal{U},\inner{\cdot}{\cdot}_{\mathcal{U}})$, $(\mathcal{H},\inner{\cdot}{\cdot}_{\mathcal{H}})$ be general Hilbert spaces. $\mathcal{O}$ is equipped with Euclidean norm.
\begin{itemize}
    \item $L^p(\mathcal{X};\mathcal{Y})$ is the usual class of measurable $p-$integrable functions from $\mathcal{X}$ into $\mathcal{Y}$, $1 \leq p < \infty$, which is a Banach space with norm $$\norm{\phi}_{L^p(\mathcal{X};\mathcal{Y})}^p := \int_{\mathcal{X}}\norm{\phi(x)}^p_{\mathcal{Y}}\mu(dx).$$ The space $L^2(\mathcal{X}; \mathcal{Y})$ is a Hilbert Space when $\mathcal{Y}$ itself is Hilbert, with the standard inner product $$\inner{\phi}{\psi}_{L^2(\mathcal{X}; \mathcal{Y})} = \int_{\mathcal{X}}\inner{\phi(x)}{\psi(x)}_\mathcal{Y} \mu(dx).$$ In the case $\mathcal{X} = \mathcal{O}$ and $\mathcal{Y} = \R^N$ note that $$\norm{\phi}_{L^2(\mathcal{O};\R^N)}^2 = \sum_{l=1}^N\norm{\phi^l}^2_{L^2(\mathcal{O};\R)}$$ for the component mappings $\phi^l: \mathcal{O} \rightarrow \R$. 
    
\item $L^{\infty}(\mathcal{X};\mathcal{Y})$ is the usual class of measurable functions from $\mathcal{X}$ into $\mathcal{Y}$ which are essentially bounded, which is a Banach Space when equipped with the norm $$ \norm{\phi}_{L^{\infty}(\mathcal{X};\mathcal{Y})} := \inf\{C \geq 0: \norm{\phi(x)}_Y \leq C \textnormal{ for $\mu$-$a.e.$ }  x \in \mathcal{X}\}.$$
    
    \item $L^{\infty}(\mathcal{O};\R^N)$ is the usual class of measurable functions from $\mathcal{O}$ into $\R^N$ such that $\phi^l \in L^{\infty}(\mathcal{O};\R)$ for $l=1,\dots,N$, which is a Banach Space when equipped with the norm $$ \norm{\phi}_{L^{\infty}}:= \sup_{l \leq N}\norm{\phi^l}_{L^{\infty}(\mathcal{O};\R)}.$$
    
     \item $C(\mathcal{X};\mathcal{Y})$ is the space of continuous functions from $\mathcal{X}$ into $\mathcal{Y}$.
      
    
    
    
    
        \item $W^{m,p}(\mathcal{O}; \R)$ for $1 \leq p < \infty$ is the sub-class of $L^p(\mathcal{O}, \R)$ which has all weak derivatives up to order $m \in \N$ also of class $L^p(\mathcal{O}, \R)$. This is a Banach space with norm $$\norm{\phi}^p_{W^{m,p}(\mathcal{O}, \R)} := \sum_{\abs{\alpha} \leq m}\norm{D^\alpha \phi}_{L^p(\mathcal{O}; \R)}^p$$ where $D^\alpha$ is the corresponding weak derivative operator. In the case $p=2$, $W^{m,2}(U, \R)$ is a Hilbert Space with inner product $$\inner{\phi}{\psi}_{W^{m,2}(\mathcal{O}; \R)} := \sum_{\abs{\alpha} \leq m} \inner{D^\alpha \phi}{D^\alpha \psi}_{L^2(\mathcal{O}; \R)}.$$
    
    \item $W^{m,\infty}(\mathcal{O};\R)$ for $m \in \N$ is the sub-class of $L^\infty(\mathcal{O}, \R)$ which has all weak derivatives up to order $m \in \N$ also of class $L^\infty(\mathcal{O}, \R)$. This is a Banach space with norm $$\norm{\phi}_{W^{m,\infty}(\mathcal{O}, \R)} := \sup_{\abs{\alpha} \leq m}\norm{D^{\alpha}\phi}_{L^{\infty}(\mathcal{O};\R^N)}.$$
    
        \item $W^{m,p}(\mathcal{O}; \R^N)$ for $1 \leq p < \infty$ is the sub-class of $L^p(\mathcal{O}, \R^N)$ which has all weak derivatives up to order $m \in \N$ also of class $L^p(\mathcal{O}, \R^N)$. This is a Banach space with norm $$\norm{\phi}^p_{W^{m,p}(\mathcal{O}, \R^N)} := \sum_{l=1}^N\norm{\phi^l}_{W^{m,p}(\mathcal{O}; \R)}^p.$$ In the case $p=2$ the space $W^{m,2}(\mathcal{O}, \R^N)$ is Hilbertian with inner product $$\inner{\phi}{\psi}_{W^{m,2}(\mathcal{O}; \R^N)} := \sum_{l=1}^N \inner{\phi^l}{\psi^l}_{W^{m,2}(\mathcal{O}; \R)}.$$
    
          \item $W^{m,\infty}(\mathcal{O}; \R^N)$ for is the sub-class of $L^\infty(\mathcal{O}, \R^N)$ which has all weak derivatives up to order $m \in \N$ also of class $L^\infty(\mathcal{O}, \R^N)$. This is a Banach space with norm $$\norm{\phi}_{W^{m,\infty}(\mathcal{O}, \R^N)} := \sup_{l \leq N}\norm{\phi^l}_{W^{m,\infty}(\mathcal{O}; \R)}.$$
          
    
          
    
    \item $\mathscr{L}(\mathcal{Y};\mathcal{Z})$ is the space of bounded linear operators from $\mathcal{Y}$ to $\mathcal{Z}$. This is a Banach Space when equipped with the norm $$\norm{F}_{\mathscr{L}(\mathcal{Y};\mathcal{Z})} = \sup_{\norm{y}_{\mathcal{Y}}=1}\norm{Fy}_{\mathcal{Z}}.$$ It is the dual space $\mathcal{Y}^*$ when $\mathcal{Z}=\R$, with operator norm $\norm{\cdot}_{\mathcal{Y}^*}.$
    
     \item $\mathscr{L}^2(\mathcal{U};\mathcal{H})$ is the space of Hilbert-Schmidt operators from $\mathcal{U}$ to $\mathcal{H}$, defined as the elements $F \in \mathscr{L}(\mathcal{U};\mathcal{H})$ such that for some basis $(e_i)$ of $\mathcal{U}$, $$\sum_{i=1}^\infty \norm{Fe_i}_{\mathcal{H}}^2 < \infty.$$ This is a Hilbert space with inner product $$\inner{F}{G}_{\mathscr{L}^2(\mathcal{U};\mathcal{H})} = \sum_{i=1}^\infty \inner{Fe_i}{Ge_i}_{\mathcal{H}}$$ which is independent of the choice of basis.

\end{itemize}

\end{definition}

\subsection{Stochastic Framework}

We work with a fixed filtered probability space $(\Omega,\mathcal{F},(\mathcal{F}_t), \mathbb{P})$ satisfying the usual conditions of completeness and right continuity. We take $\mathcal{W}$ to be a cylindrical Brownian Motion over some Hilbert Space $\mathfrak{U}$ with orthonormal basis $(e_i)$. Recall (\cite{goodair2022stochastic}, Subsection 1.4) that $\mathcal{W}$ admits the representation $\mathcal{W}_t = \sum_{i=1}^\infty e_iW^i_t$ as a limit in $L^2(\Omega;\mathfrak{U}')$ whereby the $(W^i)$ are a collection of i.i.d. standard real valued Brownian Motions and $\mathfrak{U}'$ is an enlargement of the Hilbert Space $\mathfrak{U}$ such that the embedding $J: \mathfrak{U} \rightarrow \mathfrak{U}'$ is Hilbert-Schmidt and $\mathcal{W}$ is a $JJ^*-$cylindrical Brownian Motion over $\mathfrak{U}'$. Given a process $F:[0,T] \times \Omega \rightarrow \mathscr{L}^2(\mathfrak{U};\mathscr{H})$ progressively measurable and such that $F \in L^2\left(\Omega \times [0,T];\mathscr{L}^2(\mathfrak{U};\mathscr{H})\right)$, for any $0 \leq t \leq T$ we define the stochastic integral $$\int_0^tF_sd\mathcal{W}_s:=\sum_{i=1}^\infty \int_0^tF_s(e_i)dW^i_s$$ where the infinite sum is taken in $L^2(\Omega;\mathscr{H})$. We can extend this notion to processes $F$ which are such that $F(\omega) \in L^2\left( [0,T];\mathscr{L}^2(\mathfrak{U};\mathscr{H})\right)$ for $\mathbb{P}-a.e.$ $\omega$ via the traditional localisation procedure. In this case the stochastic integral is a local martingale in $\mathscr{H}$. \footnote{A complete, direct construction of this integral, a treatment of its properties and the fundamentals of stochastic calculus in infinite dimensions can be found in \cite{goodair2022stochastic} Section 1. Properties that we shall make frequent use of are the Burkholder-Davis-Gundy type inequality Theorem 1.6.8 and the passage of a bounded linear operator through the stochastic integral, Theorem 1.6.9. Section 2 of this work looks at an abstract setting for SPDEs which we follow here, including a survey of useful results in this framework such as the It\^{o} Formula (Subsection 2.5).}


\section{H-Valued Solutions} \label{Section v valued}

In this section we state and prove our existence, uniqueness and maximality results for an SPDE (\ref{thespde}) satisfying the first set of assumptions. Throughout we will use $c$ to be a generic constant which can change from line to line.

\subsection{Assumptions} \label{assumptionschapter}

We state the assumptions for a triplet of embedded Hilbert Spaces $$V \hookrightarrow H \hookrightarrow U$$ and ask that there is a continuous bilinear form $\inner{\cdot}{\cdot}_{U \times V}: U \times V \rightarrow \R$ such that for $\phi \in H$ and $\psi \in V$, \begin{equation} \label{bilinear formog}
    \inner{\phi}{\psi}_{U \times V} =  \inner{\phi}{\psi}_{H}.
\end{equation}
The operators $\mathcal{A},\mathcal{G}$ are such that for any $T>0$,
    $\mathcal{A}:[0,T] \times V \rightarrow U,
    \mathcal{G}:[0,T] \times V \rightarrow \mathscr{L}^2(\mathfrak{U};H)$ are measurable. We assume that $V$ is dense in $H$ which is dense in $U$. 

\begin{assumption} \label{assumption fin dim spaces}
For a system $(a_n)$ of elements of $V$, define the spaces $V_n:= \textnormal{span}\left\{a_1, \dots, a_n \right\}$ and $\mathcal{P}_n$ as the orthogonal projection to $V_n$ in $U$, $\mathcal{P}_n:U \rightarrow V_n$. Then:
\begin{enumerate}
    \item There exists some constant $c$ independent of $n$ such that for all $\phi\in H$,
\begin{equation} \label{projectionsboundedonH}
    \norm{\mathcal{P}_n \phi}_H^2 \leq c\norm{\phi}_H^2.
\end{equation}
\item There exists a real valued sequence $(\mu_n)$ with $\mu_n \rightarrow \infty$ such that for any $\phi \in H$, \begin{align}
     \label{mu2}
    \norm{(I - \mathcal{P}_n)\phi}_U \leq \frac{1}{\mu_n}\norm{\phi}_H
\end{align}
where $I$ represents the identity operator in $U$.
\end{enumerate}
\end{assumption} 

\begin{remark}
The property (\ref{projectionsboundedonH}) would be immediate for the $U$ norm (and $c=1$) by definition of $\mathcal{P}_n$ as an orthogonal projection in $U$. This does not necessarily translate to the $H$ norm so it is a required assumption. We rely on this property when showing the uniform boundedness of the Galerkin Equations (Proposition \ref{theorem:uniformbounds}) as the bounds we produce are dependent on the $H$ norm of their initial conditions. 
\end{remark}

\begin{remark}
As well as the density of $V$ in $H$, the purpose of (\ref{mu2}) is to give us a rate of approximation of elements of $H$ by those in $V$ for the $U$ norm. This will be necessary in showing the convergence of the Galerkin System (Theorems \ref{theorem:cauchygalerkin}, \ref{existence of local strong V solution}).
\end{remark}

We shall use general notation $c_t$ to represent a function $c_\cdot:[0,\infty) \rightarrow \R$ bounded on $[0,T]$ for any $T > 0$, evaluated at the time $t$. Moreover we define functions $K$, $\tilde{K}$ relative to some non-negative constants $p,\tilde{p},q,\tilde{q}$. We use a generic notation to define the functions $K: U \rightarrow \R$, $K: U \times U \rightarrow \R$, $\tilde{K}: H \rightarrow \R$ and $\tilde{K}: H \times H \rightarrow \R$ by
\begin{align*}
    K(\phi)&:= 1 + \norm{\phi}_U^{p},\\
    K(\phi,\psi)&:= 1+\norm{\phi}_U^{p} + \norm{\psi}_U^{q},\\
    \tilde{K}(\phi) &:= K(\phi) + \norm{\phi}_H^{\tilde{p}},\\
    \tilde{K}(\phi,\psi) &:= K(\phi,\psi) + \norm{\phi}_H^{\tilde{p}} + \norm{\psi}_H^{\tilde{q}}
\end{align*}
  In the case of $\tilde{K}$, when $\tilde{p}, \tilde{q} = 2$ then we shall denote the general $\tilde{K}$ by $\tilde{K}_2$. In this case no further assumptions are made on the $p,q$. That is, $\tilde{K}_2$ has the general representation \begin{equation}\label{Ktilde2}\tilde{K}_2(\phi,\psi) = K(\phi,\psi) + \norm{\phi}_H^2 + \norm{\psi}_H^2\end{equation} and similarly as a function of one variable.\\
 
 We state the subsequent assumptions for arbitrary elements $\phi,\psi \in V$, $\phi^n \in V_n$, $\eta \in H$ and $t \in [0,\infty)$, and a fixed $\kappa > 0$. Understanding $\mathcal{G}$ as an operator $\mathcal{G}: [0,\infty) \times V \times \mathfrak{U} \rightarrow H$, we introduce the notation $\mathcal{G}_i(\cdot,\cdot):= \mathcal{G}(\cdot,\cdot,e_i)$.
 
 
  \begin{assumption} \label{new assumption 1} \begin{align}
     \label{111} \norm{\mathcal{A}(t,\boldsymbol{\phi})}^2_U +\sum_{i=1}^\infty \norm{\mathcal{G}_i(t,\boldsymbol{\phi})}^2_H &\leq c_t K(\boldsymbol{\phi})\left[1 + \norm{\boldsymbol{\phi}}_V^2\right],\\ \label{222}
     \norm{\mathcal{A}(t,\boldsymbol{\phi}) - \mathcal{A}(t,\boldsymbol{\psi})}_U^2 &\leq  c_t\left[K(\phi,\psi) + \norm{\phi}_V^p + \norm{\psi}_V^q\right]\norm{\phi-\psi}_V^2,\\ \label{333}
    \sum_{i=1}^\infty \norm{\mathcal{G}_i(t,\boldsymbol{\phi}) - \mathcal{G}_i(t,\boldsymbol{\psi})}_U^2 &\leq c_tK(\phi,\psi)\norm{\phi-\psi}_H^2.
 \end{align}
 \end{assumption}

\begin{assumption} \label{assumptions for uniform bounds2}
 \begin{align}
   \label{uniformboundsassumpt1}  2\inner{\mathcal{P}_n\mathcal{A}(t,\boldsymbol{\phi}^n)}{\boldsymbol{\phi}^n}_H + \sum_{i=1}^\infty\norm{\mathcal{P}_n\mathcal{G}_i(t,\boldsymbol{\phi}^n)}_H^2 &\leq c_t\tilde{K}_2(\boldsymbol{\phi}^n)\left[1 + \norm{\boldsymbol{\phi}^n}_H^2\right] - \kappa\norm{\boldsymbol{\phi}^n}_V^2,\\  \label{uniformboundsassumpt2}
    \sum_{i=1}^\infty \inner{\mathcal{P}_n\mathcal{G}_i(t,\boldsymbol{\phi}^n)}{\boldsymbol{\phi}^n}^2_H &\leq c_t\tilde{K}_2(\boldsymbol{\phi}^n)\left[1 + \norm{\boldsymbol{\phi}^n}_H^4\right].
\end{align}
\end{assumption}

\begin{assumption} \label{therealcauchy assumptions}
\begin{align}
  \nonumber 2\inner{\mathcal{A}(t,\boldsymbol{\phi}) - \mathcal{A}(t,\boldsymbol{\psi})}{\boldsymbol{\phi} - \boldsymbol{\psi}}_U &+ \sum_{i=1}^\infty\norm{\mathcal{G}_i(t,\boldsymbol{\phi}) - \mathcal{G}_i(t,\boldsymbol{\psi})}_U^2\\ \label{therealcauchy1} &\leq  c_{t}\tilde{K}_2(\boldsymbol{\phi},\boldsymbol{\psi}) \norm{\boldsymbol{\phi}-\boldsymbol{\psi}}_U^2 - \kappa\norm{\boldsymbol{\phi}-\boldsymbol{\psi}}_H^2,\\ \label{therealcauchy2}
    \sum_{i=1}^\infty \inner{\mathcal{G}_i(t,\boldsymbol{\phi}) - \mathcal{G}_i(t,\boldsymbol{\psi})}{\boldsymbol{\phi}-\boldsymbol{\psi}}^2_U & \leq c_{t} \tilde{K}_2(\boldsymbol{\phi},\boldsymbol{\psi}) \norm{\boldsymbol{\phi}-\boldsymbol{\psi}}_U^4.
\end{align}
\end{assumption}

\begin{assumption} \label{assumption for prob in V}
\begin{align}
   \label{probability first} 2\inner{\mathcal{A}(t,\boldsymbol{\phi})}{\boldsymbol{\phi}}_U + \sum_{i=1}^\infty\norm{\mathcal{G}_i(t,\boldsymbol{\phi})}_U^2 &\leq c_tK(\boldsymbol{\phi})\left[1 +  \norm{\boldsymbol{\phi}}_H^2\right],\\\label{probability second}
    \sum_{i=1}^\infty \inner{\mathcal{G}_i(t,\boldsymbol{\phi})}{\boldsymbol{\phi}}^2_U &\leq c_tK(\boldsymbol{\phi})\left[1 + \norm{\boldsymbol{\phi}}_H^4\right].
\end{align}
\end{assumption}

\begin{assumption} \label{finally the last assumption}
 \begin{equation} \label{lastlast assumption}
    \inner{\mathcal{A}(t,\phi)-\mathcal{A}(t,\psi)}{\eta}_U \leq c_t(1+\norm{\eta}_H)\left[K(\phi,\psi) + \norm{\phi}_V + \norm{\psi}_V\right]\norm{\phi-\psi}_H.
    \end{equation}
\end{assumption}

We now briefly address the purpose of these assumptions.

\begin{itemize}
    \item Assumption \ref{new assumption 1} provides the growth and local Lipschitz type constraints, which ensure that the integrals in (\ref{identityindefinitionoflocalsolution}) are well defined and that solutions to the Galerkin Equations (\ref{nthorderGalerkin}) exist. The growth restriction (\ref{111}) also facilitates the convergence of the Galerkin Approximations by controlling the difference of terms with their finite dimensional projections (Proposition \ref{theorem:cauchygalerkin}, Theorem \ref{existence of local strong V solution}), whilst we use (\ref{333}) to show that the approximating sequence of stochastic integrals converge to the appropriate limit (Theorem \ref{existence of local strong V solution}).
    \item Assumption \ref{assumptions for uniform bounds2} contains a coercivity type constraint and facilitates uniform boundedness of the solutions of the Galerkin Equations (Proposition \ref{theorem:uniformbounds}).
    \item Assumption \ref{therealcauchy assumptions} is a monotonicity requirement, necessary for the Cauchy Property of the Galerkin Approximations (Proposition \ref{theorem:cauchygalerkin}) and the uniqueness of solutions (Theorem \ref{uniqueness for V valued solutions}).
    \item Assumption \ref{assumption for prob in V} is used to show the uniform rate of convergence of the Galerkin Approximations to their initial conditions (Theorem \ref{theorem: probability one}).
    \item Assumption \ref{finally the last assumption} is used to show the convergence of the time integrals in the Galerkin Approximations to the appropriate limit (Theorem \ref{existence of local strong V solution}).
\end{itemize}

\subsection{Definitions and Main Results} \label{subsection:notionsofsolution}

With these assumptions in place we state the relevant definitions and results.

\begin{definition}[$H-$valued local strong solution] \label{definitionofregularsolution}
Let $\sy_0:\Omega \rightarrow H$ be $\mathcal{F}_0-$ measurable. A pair $(\sy,\tau)$ where $\tau$ is a $\mathbb{P}-a.s.$ positive stopping time and $\sy$ is a process such that for $\mathbb{P}-a.e.$ $\omega$, $\sy_{\cdot}(\omega) \in C\left([0,T];H\right)$ and $\sy_{\cdot}(\omega)\mathbbm{1}_{\cdot \leq \tau(\omega)} \in L^2\left([0,T];V\right)$ for all $T>0$ with $\sy_{\cdot}\mathbbm{1}_{\cdot \leq \tau}$ progressively measurable in $V$, is said to be an $H-$valued local strong solution of the equation (\ref{thespde}) if the identity
\begin{equation} \label{identityindefinitionoflocalsolution}
    \sy_{t} = \sy_0 + \int_0^{t\wedge \tau} \mathcal{A}(s,\sy_s)ds + \int_0^{t \wedge \tau}\mathcal{G} (s,\sy_s) d\mathcal{W}_s
\end{equation}
holds $\mathbb{P}-a.s.$ in $U$ for all $t \geq 0$. \footnote{A detailed justification that the terms in this definition are well defined is given in \cite{goodair2022stochastic} Subsections 2.2 and 2.4, referring to (\ref{111}).}
\end{definition}

\begin{remark} \label{remark1}
If $(\sy,\tau)$ is an $H-$valued local strong solution of the equation (\ref{thespde}), then $\sy_\cdot = \sy_{
\cdot \wedge \tau}$ due to the identity (\ref{identityindefinitionoflocalsolution}).
\end{remark}

\begin{remark} \label{remark on prog meas equivalence}
The progressive measurability condition on $\sy_{\cdot}\mathbbm{1}_{\cdot \leq \tau}$ may look a little suspect as $\sy_0$ itself may only belong to $H$ and not $V$ making it impossible for $\sy_{\cdot}\mathbbm{1}_{\cdot \leq \tau}$ to be even adapted in $V$. We are mildly abusing notation here; what we really ask is that there exists a process $\py$ which is progressively measurable in $V$ and such that $\py_{\cdot} = \sy_{\cdot}\mathbbm{1}_{\cdot \leq \tau}$ almost surely over the product space $\Omega \times [0,\infty)$ with product measure $\mathbbm{P}\times \lambda$ for $\lambda$ the Lebesgue measure on $[0,\infty)$. 
\end{remark}

\begin{definition}[$H-$valued maximal strong solution] \label{V valued maximal definition}
A pair $(\sy,\Theta)$ such that there exists a sequence of stopping times $(\theta_j)$ which are $\mathbb{P}-a.s.$ monotone increasing and convergent to $\Theta$, whereby $(\sy_{\cdot \wedge \theta_j},\theta_j)$ is an $H-$valued local strong solution of the equation (\ref{thespde}) for each $j$, is said to be an $H-$valued maximal strong solution of the equation (\ref{thespde}) if for any other pair $(\py,\Gamma)$ with this property then $\Theta \leq \Gamma$ $\mathbb{P}-a.s.$ implies $\Theta = \Gamma$ $\mathbb{P}-a.s.$.
\end{definition}

\begin{remark}
We do not require $\Theta$ to be finite in this definition, in which case we mean that the sequence $(\theta_j)$ is monotone increasing and unbounded for such $\omega$. 
\end{remark}

\begin{definition} \label{v valued maximal unique}
An $H-$valued maximal strong solution $(\sy,\Theta)$ of the equation (\ref{thespde}) is said to be unique if for any other such solution $(\py,\Gamma)$, then $\Theta = \Gamma$ $\mathbb{P}-a.s.$ and \begin{equation} \nonumber\mathbb{P}\left(\left\{\omega \in \Omega: \sy_{t}(\omega) =  \py_{t}(\omega)  \quad \forall t \in [0,\Theta) \right\} \right) = 1. \end{equation}
\end{definition}

The following is the main result of the section, and holds true if the assumptions of Subsection \ref{assumptionschapter} are met. 

\begin{theorem} \label{theorem1}
For any given $\mathcal{F}_0-$ measurable $\sy_0:\Omega \rightarrow H$, there exists a unique $H-$valued maximal strong solution $(\sy,\Theta)$ of the equation (\ref{thespde}). Moreover at $\mathbb{P}-a.e.$ $\omega$ for which $\Theta(\omega)<\infty$, we have that \begin{equation}\label{actualblowup}\sup_{r \in [0,\Theta(\omega))}\norm{\sy_r(\omega)}_H^2 + \int_0^{\Theta(\omega)}\norm{\sy_r(\omega)}_V^2dr = \infty.\end{equation}
\end{theorem}

Our method has already been discussed in the introduction and in the necessity of each assumption in \ref{assumptionschapter}, though now we explicitly lay out the steps taken in proving Theorem \ref{theorem1}. The remainder of Section \ref{Section v valued} follows the path portrayed here. 

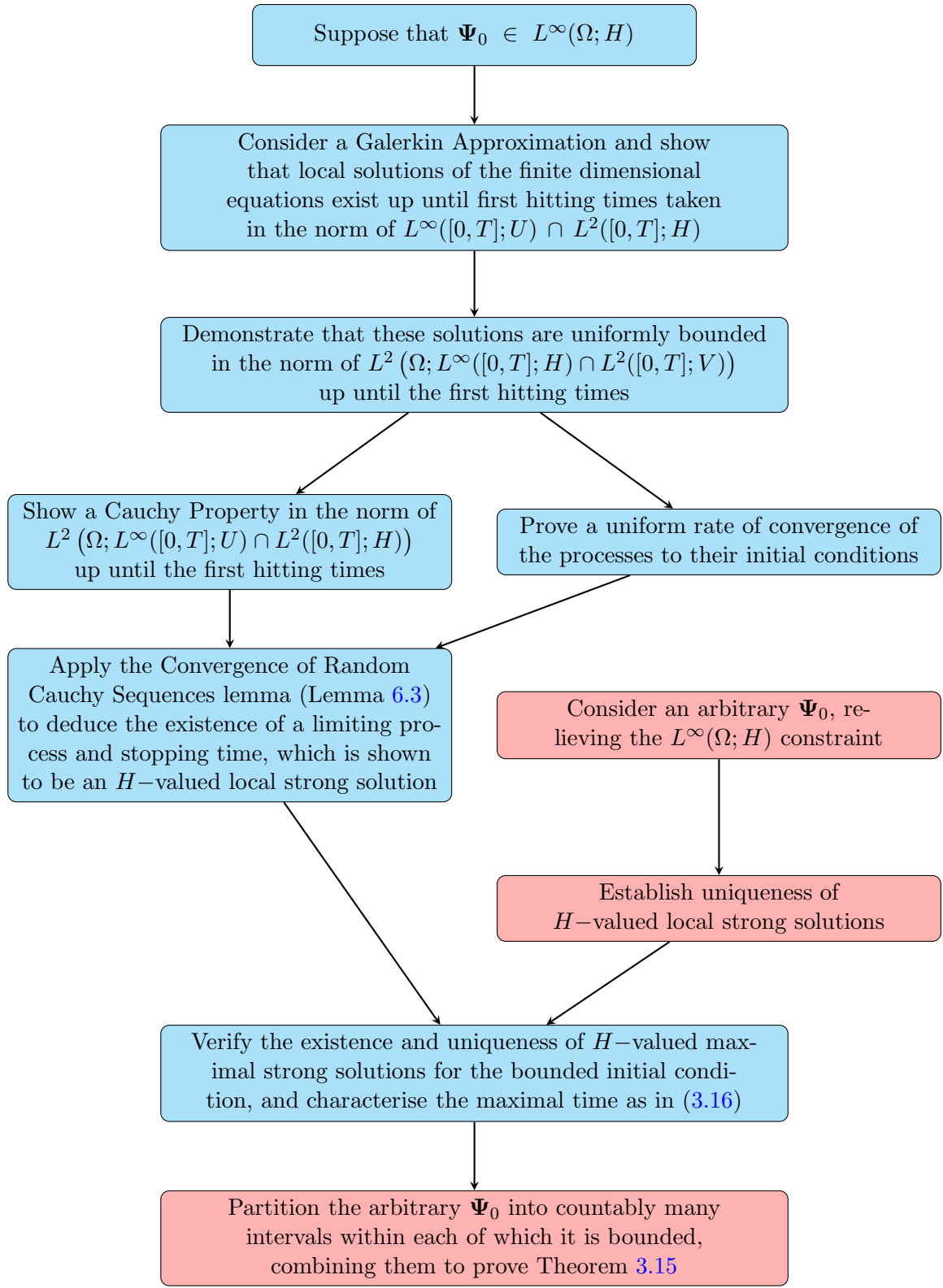
\begin{tikzpicture}[node distance=2cm]

\node (1) [startstop, xshift=10cm] {Suppose that $\sy_0 \in L^{\infty}(\Omega;H)$};

\node(2)[startstop2, below of=1, yshift=-0.5cm]{Consider a Galerkin Approximation and show that local solutions of the finite dimensional equations exist up until first hitting times taken in the norm of $L^\infty([0,T];U) \cap L^2([0,T];H)$ };

\node(3)[startstop2, below of=2, yshift=-0.9cm]{Demonstrate that these solutions are uniformly bounded in the norm of $L^2\left(\Omega;L^\infty([0,T];H) \cap L^2([0,T];V) \right)$ up until the first hitting times};

\node(4)[startstop, below of=3, yshift=-0.9cm, xshift= -4cm]{Show a Cauchy Property in the norm of $L^2\left(\Omega;L^\infty([0,T];U) \cap L^2([0,T];H) \right)$ up until the first hitting times};

\node(5)[startstop, right of=4, xshift= 6cm]{Prove a uniform rate of convergence of the processes to their initial conditions};

\node(6)[startstop, below of=4, yshift=-1cm]{Apply the Convergence of Random Cauchy Sequences lemma (Lemma \ref{greenlemma}) to deduce the existence of a limiting process and stopping time, which is shown to be an $H-$valued local strong solution};

\node(7)[startstop', right of=6, xshift= 6cm]{Consider an arbitrary $\sy_0$, relieving the $L^{\infty}(\Omega;H)$ constraint};

\node(8)[startstop', below of=7, yshift=-1cm]{Establish uniqueness of $H-$valued local strong solutions};

\node(9)[startstop2, below of=8, yshift=-0.7cm, xshift=-4cm]{Verify the existence and uniqueness of $H-$valued maximal strong solutions for the bounded initial condition, and characterise the maximal time as in (\ref{actualblowup})};

\node(10)[startstop2', below of=9, yshift=-0.7cm]{Partition the arbitrary $\sy_0$ into countably many intervals within each of which it is bounded, combining them to prove Theorem \ref{theorem1}};

\draw[arrow] (1) -- (2);
\draw[arrow] (2) -- (3);
\draw[arrow] (3) -- (4);
\draw[arrow] (3) -- (5);
\draw[arrow] (4) -- (6);
\draw[arrow] (5) -- (6);
\draw[arrow] (7) -- (8);
\draw[arrow] (6) -- (9);
\draw[arrow] (8) -- (9);
\draw[arrow] (9) -- (10);

\end{tikzpicture}

\subsection{The Galerkin System}
\label{subsection:Existence Method for a Regular Truncated Initial Condition}

For now we shall assume that  \begin{equation} \label{boundedinitialcondition}
    \sy_0 \in L^{\infty}(\Omega;H).
\end{equation} We consider the Galerkin Equations
\begin{equation} \label{nthorderGalerkin}
       \sy^n_t = \sy^n_0 + \int_0^t \mathcal{P}_n\mathcal{A}(s,\sy^n_s)ds + \int_0^t\mathcal{P}_n\mathcal{G} (s,\sy^n_s) d\mathcal{W}_s
\end{equation}
for the initial condition $\sy^n_0:= \mathcal{P}_n\sy_0$ and $\mathcal{P}_n\mathcal{G} (e_i,s,\cdot):=\mathcal{P}_n\mathcal{G}_{i}(s,\cdot).$ Note that from (\ref{boundedinitialcondition}), (\ref{projectionsboundedonH}) and the embedding of $H$ into $U$ we have that
\begin{equation} \label{uniformboundofinitialcondition}
    \sup_{n\in\N}\norm{\sy^n_0}_{L^{\infty}(\Omega,H)}^2 < \infty, \qquad \sup_{n\in\N}\norm{\sy^n_0}_{L^{\infty}(\Omega,U)}^2 < \infty.
\end{equation}

\begin{remark}
We may consider the finite dimensional $V_n$ as a Hilbert Space equipped with any of the equivalent $V,H,U$ inner products. 
\end{remark}

\begin{definition} \label{definition of local strong solution to galerkin equation}
A pair $(\sy^n,\tau)$ where $\tau$ is a $\mathbbm{P}-a.s.$ positive stopping time and $\sy^n$ is an adapted process in $V_n$ such that for $\mathbbm{P}-a.e.$ $\omega$, $\sy^n_{\cdot}(\omega) \in C\left([0,T];V_n\right)$ for all $T>0$, is said to be a local strong solution of the equation (\ref{nthorderGalerkin}) if the identity
\begin{equation} \label{identityindefinitionoflocalgalerkinsolution}
    \sy^n_{t} = \sy^n_0 + \int_0^{t\wedge \tau} \mathcal{P}_n\mathcal{A}(s,\sy^n_s)ds + \int_0^{t \wedge \tau}\mathcal{P}_n\mathcal{G} (s,\sy^n_s) d\mathcal{W}_s
\end{equation}
holds $\mathbbm{P}-a.s.$ in $V_n$ for all $t \geq 0$.
\end{definition}

A justification that this formulation makes sense is largely similar to that of \ref{definitionofregularsolution}. Using that $\mathcal{P}_n$ is an orthogonal projection in $U$, then we have the bounds
\begin{align*}
    \int_0^{t\wedge \tau(\omega)} \norm{\mathcal{P}_n\mathcal{A}\left(s,\sy^n_s(\omega)\right)}_U^2ds  &\leq \int_0^{t\wedge \tau(\omega)} \norm{\mathcal{A}\left(s,\sy^n_s(\omega)\right)}_U^2ds < \infty\\
    \int_0^{t\wedge \tau(\omega)}\sum_{i=1}^\infty \norm{\mathcal{P}_n\mathcal{G}_i\left(s,\sy^n_s(\omega)\right)}_U^2ds &\leq \int_0^{t\wedge \tau(\omega)}\sum_{i=1}^\infty c\norm{\mathcal{G}_i\left(s,\sy^n_s(\omega)\right)}_H^2ds  < \infty
\end{align*}
from which we are deferred to Definition \ref{definitionofregularsolution}. The continuity of the $\mathcal{P}_n$ in $U$ ensures that the measurability is preserved, so identity (\ref{identityindefinitionoflocalgalerkinsolution}) makes sense. Note that in this case, we do not need to pass to some $\mathbbm{P} \times \lambda$ almost everywhere equal process to satisfy the progressive measurability. In looking to deduce the existence of such a solution, we first consider a truncated version of the equation. For any fixed $R>0$, we introduce the function $f_R: [0,\infty) \rightarrow [0,1]$ constructed such that $$f_R \in C^{\infty}\left([0,\infty);[0,1]\right), \qquad f_R(x)=1 \ \forall x \in [0,R], \qquad f_R(x)=0 \ \forall x \in [2R,\infty).$$ We consider now the equation
\begin{equation} \label{truncatedgalerkin}
     \sy^{n,R}_t = \sy^{n,R}_0 + \int_0^t f_R\left(\norm{\sy^{n,R}_s}_H^2\right)\mathcal{P}_n\mathcal{A}(s,\sy^{n,R}_s)ds + \int_0^tf_R\left(\norm{\sy^{n,R}_s}_H^2\right)\mathcal{P}_n\mathcal{G} (s,\sy^{n,R}_s) d\mathcal{W}_s
\end{equation}
for $\sy^{n,R}_0:=\sy^{n}_0$ which we use as an intermediary step to deduce the existence of solutions as in Definition \ref{definition of local strong solution to galerkin equation}. Solutions of the truncated equation are defined in the sense of Proposition \ref{Skorotheorem} (See Appendix II, \ref{subsection useful results}), and indeed due to (\ref{boundedinitialcondition}) and Assumption \ref{new assumption 1} we can apply this theorem in the case of $\mathcal{H}:=V_n$, $ \mathscr{A}:=\mathcal{P}_n\mathcal{A}$, $ \mathscr{G}:=\mathcal{P}_n\mathcal{G}$ to deduce the existence of solutions to (\ref{truncatedgalerkin}) for all $n\in\N, R>0$.\\

The motivation for considering (\ref{truncatedgalerkin}) is to prove the existence of local strong solutions to (\ref{nthorderGalerkin}) by considering local intervals of existence on which the truncation threshold isn't reached. More than this, for any first hitting time in the fundamental $L^\infty([0,T];U) \cap L^2([0,T];H)$ norm, we show that a large enough truncation can be chosen so that solutions to (\ref{nthorderGalerkin}) exist up until the first hitting time. The local strong solutions can then be controlled uniformly across $\omega$ on their time of existence, and these times are intrinsic to the application of the Convergence of Random Cauchy Sequences lemma (Lemma \ref{greenlemma}) which we use to justify the existence of an $H-$valued local strong solution of the equation (\ref{thespde}).

\begin{lemma} \label{localexistenceofsolutionsofGalerkin}
For any $t>0$, $M>1$ and fixed $n \in N$, there exists an $R>0$ such that the following holds; letting $\sy^{n,R}$ be the solution of (\ref{truncatedgalerkin}) and $$\tau^{M,t}_{n,R}(\omega) := t \wedge \inf\left\{s \geq 0: \sup_{r \in [0,s]}\norm{\sy^{n,R}_{r}(\omega)}^2_U + \int_0^s\norm{\sy^{n,R}_{r}(\omega)}^2_Hdr \geq M + \norm{\sy^n_0(\omega)}_U^2 \right\}$$ then $(\sy^{n,R}_{\cdot \wedge \tau^{M,t}_{n,R}}, \tau^{M,t}_{n,R})$ is a local strong solution of (\ref{nthorderGalerkin}).
\end{lemma}

\begin{proof}
See Appendix I, \ref{proofs section 3 appendix}.
\end{proof}

\begin{remark}
For notational convenience we define $\sy^{n}_{\cdot}:=\sy^{n,R}_{\cdot \wedge \tau^{M,t}_{n,R}}$ and \begin{equation}\label{tauMtn}\tau^{M,t}_n(\omega) := t \wedge \inf\left\{s \geq 0: \sup_{r \in [0,s]}\norm{\sy^{n}_{r}(\omega)}^2_U + \int_0^s\norm{\sy^{n}_{r}(\omega)}^2_Hdr \geq M + \norm{\sy^n_0(\omega)}_U^2 \right\}.\end{equation} Of course  $\tau^{M,t}_n$ is nothing but $\tau^{M,t}_{n,R}$, but we shall henceforth work with the local strong solution $(\sy^n, \tau^{M,t}_n)$.
\end{remark}

\begin{remark}
The stopping time $\tau^{M,t}_n$ is the first hitting time with respect to a norm which is central to the arguments of the paper, alongside the symmetrical norm for the smaller spaces $V$ and $H$. As such for functions $\py \in L^\infty([S,T];U) \cap L^2([S,T];H)$, $\sy \in L^\infty([S,T];H) \cap L^2([S,T];V)$ we define the norms \begin{align} \label{new norm}
    \norm{\py}^2_{UH,S,T}:&= \sup_{r \in [S,T]}\norm{\py_{r}(\omega)}^2_U + \int_S^T\norm{\py_{r}(\omega)}^2_Hdr\\
     \norm{\sy}^2_{HV,S,T}:&= \sup_{r \in [S,T]}\norm{\sy_{r}(\omega)}^2_H + \int_S^T\norm{\sy_{r}(\omega)}^2_Vdr
\end{align}
making explicit the dependence on the times $S,T$. During commentary to express an idea we may refer to these as the '$UH$' and '$HV$' norms respectively. In the case $S=0$ we reduce the notation to 
\begin{align*}
    \norm{\py}^2_{UH,T} :&= \norm{\py}^2_{UH,0,T}\\
     \norm{\sy}^2_{HV,T} :&=  \norm{\sy}^2_{HV,0,T}.
\end{align*}
\end{remark}

We fix arbitrary $t> 0$ and $M > 1$ in the definition (\ref{tauMtn}). The stopped process $\sy^n_{\cdot \wedge \tau^{M,t}_n}$ is a genuine square integrable semimartingale in $V_n$ (where the stochastic integral is a true square integrable martingale). Indeed the stopped process is bounded uniformly in $s$ and $\mathbbm{P}-a.e.$ $\omega$, and satisfies the bound \begin{equation} \label{first bound on galerkin}
    \norm{\sy^n(\omega)}_{UH,\tau^{M,t}_n(\omega)}^2 \leq M + \norm{\sy_0}_{L^{\infty}(\Omega,U)}^2
\end{equation} for every $n$, $\mathbbm{P}-a.e.$ $\omega$, or equivalently with the notation
\begin{equation} \label{tildesynotation}
    \tilde{\sy}^n_\cdot:= \sy^n_{\cdot}\mathbbm{1}_{\cdot \leq \tau^{M,t}_n}
\end{equation} that
\begin{equation} \label{galerkinboundsatisfiedbystoppingtime}
    \norm{\tilde{\sy}^n(\omega)}_{UH,T}^2 \leq M + \norm{\sy_0}_{L^{\infty}(\Omega,U)}^2
\end{equation}
for any $T>0$. The significance of working with an initial condition (\ref{boundedinitialcondition}) is highlighted by the bounds (\ref{first bound on galerkin}),(\ref{galerkinboundsatisfiedbystoppingtime}) being constant. The passage to a general initial condition will take place in Subsection \ref{subsection maximality for unbounded}. Recalling the definitions of $K,\tilde{K}_2$ as in (\ref{Ktilde2}), $$K(\phi):= 1 + \norm{\phi}_U^p, \qquad \tilde{K}_2(\phi)= K(\phi) + \norm{\phi}_H^2,$$ then (\ref{galerkinboundsatisfiedbystoppingtime}) implies the existence of a constant $c$ dependent on $\sy_0,M,t$ but independent of $n,\omega$ such that \begin{equation} \label{galerkinboundsatisfiedbystoppingtime2}
    \sup_{r \in [0,t]}K(\tilde{\sy}^n_r(\omega)) \leq c, \qquad \tilde{K}_2(\tilde{\sy}^n_r) \leq c\left(1 + \norm{\tilde{\sy}^n_r}_H^2\right). 
\end{equation}

\subsection{Existence Method for a Bounded Initial Condition}

We now proceed to formalise and demonstrate the steps laid out in the introduction to prove the existence of $H-$valued local strong solutions, starting with the uniform boundedness property. The significance of the result is that we work up to first hitting times giving us a trivial control (\ref{first bound on galerkin}) in the $UH$ norm $\mathbbm{P}-$a.s., but in fact with Assumption \ref{assumptions for uniform bounds2} we can generate a control in expectation for the finer $HV$ norm. Such a property is necessary in proving the Cauchy attribute for the $UH$ norm to come in Proposition \ref{theorem:cauchygalerkin}, where the initial idea is that we must take the limit in $m \rightarrow \infty$ to nullify the term $\sum_{i=1}^\infty\norm{[I-\mathcal{P}_m]\mathcal{G}_i(s,\tilde{\sy}^m_s)}_U^2$. From (\ref{mu2}) this is bounded by $$\sum_{i=1}^\infty\frac{1}{\mu_m^2}\norm{\mathcal{G}_i(s,\tilde{\sy}^m_s)}_H^2 $$ and so we would be happy if there was some workable bound on $\sum_{i=1}^\infty\norm{\mathcal{G}_i(s,\tilde{\sy}^m_s)}_H^2$ uniform in $m$. As elucidated in the introduction this is where our method must depart from that of Glatt-Holtz and Ziane in \cite{glatt2009strong}, as heuristically in their case $$\sum_{i=1}^\infty\norm{\mathcal{G}_i(s,\tilde{\sy}^m_s)}_H^2 \leq c\left(\norm{\tilde{\sy}^m_s}_H^2 + 1 \right)$$ and the necessary control arrives courtesy of (\ref{galerkinboundsatisfiedbystoppingtime}). For applications to a gradient-dependent noise and thus working with the assumption (\ref{111}) instead, we must have control in the $V$ norm hence the need for Proposition \ref{theorem:uniformbounds}. Even where not directly instructed by application, this facilitates some more leeway in the assumptions which we hope can be appreciated in other employment.

\begin{proposition} \label{theorem:uniformbounds}
There exists a constant $C$ dependent on $M,t$ but independent of $n$ such that for the local strong solution $(\sy^n, \tau^{M,t}_n)$ of (\ref{nthorderGalerkin}), \begin{equation} \label{firstresultofuniformbounds}
    \mathbbm{E}\norm{\sy^n}_{HV,\tau^{M,t}_{n}}^2\leq C\left[\mathbbm{E}\left(\norm{\sy^n_{0}}_H^2\right) + 1\right].
\end{equation}

\end{proposition}

\begin{proof}
See Appendix I, \ref{proofs section 3 appendix}.
\end{proof}

\begin{remark}
In particular through (\ref{projectionsboundedonH}) and (\ref{uniformboundofinitialcondition}) we have that
\begin{equation} \label{secondresultofuniformbounds}
    \mathbbm{E}\norm{\sy^n}_{HV,\tau^{M,t}_{n}}^2\leq C.
\end{equation}

\end{remark}

In order to show the aforementioned Cauchy property we use an intermediary lemma, combining our assumptions to give precisely what we need in our energy estimate for the difference of two solutions of (\ref{nthorderGalerkin}).

\begin{lemma} \label{lemma to use cauchy}
With notation as in Subsection \ref{assumptionschapter}, for arbitrary $m<n$, $\boldsymbol{\phi} \in V_n, \boldsymbol{\psi} \in V_m$, define
\begin{align*}
    A &= 2\inner{\mathcal{P}_n\mathcal{A}(s,\boldsymbol{\phi}) - \mathcal{P}_m\mathcal{A}(s,\boldsymbol{\psi})}{\boldsymbol{\phi} - \boldsymbol{\psi}}_U + \sum_{i=1}^\infty\norm{\mathcal{P}_n\mathcal{G}_i(s,\boldsymbol{\phi}) - \mathcal{P}_m\mathcal{G}_i(s,\boldsymbol{\psi})}_U^2\\
    B &= \sum_{i=1}^\infty \inner{\mathcal{P}_n\mathcal{G}_i(s,\boldsymbol{\phi}) - \mathcal{P}_m\mathcal{G}_i(s,\boldsymbol{\psi})}{\boldsymbol{\phi}-\boldsymbol{\psi}}^2_U.
\end{align*}
Then for a sequence $(\lambda_m)$ with $\lambda_m \rightarrow \infty$, we have that
\begin{align}
\nonumber A &\leq c_{s}\tilde{K}_2(\boldsymbol{\phi},\boldsymbol{\psi}) \norm{\boldsymbol{\phi}-\boldsymbol{\psi}}_U^2 - \frac{\kappa}{2}\norm{\boldsymbol{\phi}-\boldsymbol{\psi}}_H^2 + \frac{c_s}{\lambda_m}K(\boldsymbol{\phi},\boldsymbol{\psi})\left[1 + \norm{\boldsymbol{\phi}}_V^2 + \norm{\boldsymbol{\psi}}_V^2\right] ;\\\nonumber
  B &\leq c_{s} \tilde{K}_2(\boldsymbol{\phi},\boldsymbol{\psi}) \norm{\boldsymbol{\phi}-\boldsymbol{\psi}}_U^4  + \frac{c_{s}}{\lambda_m} K(\boldsymbol{\phi},\boldsymbol{\psi})\left[1 + \norm{\boldsymbol{\psi}}_V^2\right].
\end{align}
\end{lemma}

\begin{proof}
See Appendix I, \ref{proofs section 3 appendix}.
\end{proof}

We are now all set to prove the Cauchy property, as required to apply the Convergence of Random Cauchy Sequences lemma (Lemma \ref{greenlemma}) in order to deduce the existence of a limiting process and stopping time. The idea behind showing the Cauchy property (\ref{cauchy result pt 2}) is that for each given $j \in \N$, there exists an $n_j$ such that for all $k \geq n_j$,
$$\mathbbm{E}\norm{\sy^{k}-\sy^{n_j}}_{UH,\tau^{M,t}_{n_j}\wedge \tau^{M,t}_k}^2 \leq 2^{-2j}.$$ Thus by defining the sets $$\Omega_j:=\left\{\omega \in \Omega: \norm{\sy^{n_j+1}(\omega)-\sy^{n_j}(\omega)}_{UH,\tau^{M,t}_{n_j}(\omega)\wedge \tau^{M,t}_{n_j+1}(\omega)}^2 \leq 2^{-(j+2)} \right\} $$ we can justify that $$\mathbbm{P}\left(\bigcap_{K=1}^\infty \bigcup_{j=K}^\infty \Omega_j^C\right)=0$$ from Borel-Cantelli and the Chebyshev Inequality. The complement set is thus of full measure and the desired pointwise Cauchy property holds on that set. In fact we would work with a slightly different set of first hitting times, but this is the concept and motivation behind the following theorem. 

\begin{proposition} \label{theorem:cauchygalerkin}
For any $m,n \in \N$ with $m<n$, define the process $\sy^{m,n}$ by $$\sy^{m,n}_r(\omega) := \sy^n_r(\omega) - \sy^m_r(\omega).$$
Then for the sequence $(\lambda_j)$ proposed in Lemma \ref{lemma to use cauchy} and $m$ sufficiently large such that $\lambda_m > 1$, there exists a constant $C$ dependent on $M,t$ but independent of $m,n$ such that \begin{equation} \label{cauchy result pt 1}\mathbbm{E}\norm{\sy^{m,n}}_{UH,\tau^{M,t}_m \wedge \tau^{M,t}_n}^2  \leq C\left[\mathbbm{E}\norm{\sy^{m,n}_0}_U^2 + \frac{1}{\sqrt{\lambda_m}} \right] \end{equation} and in particular,
\begin{equation} \label{cauchy result pt 2}\lim_{m \rightarrow \infty}\sup_{n \geq m}\left[\mathbbm{E}\norm{\sy^{m,n}}_{UH,\tau^{M,t}_m \wedge \tau^{M,t}_n}^2\right] = 0.\end{equation}
\end{proposition}

\begin{proof}
See Appendix I, \ref{proofs section 3 appendix}.
\end{proof}

The following theorem justifies the second property required in invoking the Convergence of Random Cauchy Sequences lemma (Lemma \ref{greenlemma}), which is a uniform rate of convergence of the $(\sy^n)$ to their initial conditions. The idea is for none of the $\sy^n$ to reach the threshold of the first hitting time $\tau^{M,t}_n$ in an arbitrarily small time. Out of this we can construct a limiting stopping time which is less than at least a subsequence of the $(\tau^{M,t}_n)$ and is ensured to be strictly positive $\mathbbm{P}-a.s.$. 

\begin{proposition} \label{theorem: probability one}
We have that 
\begin{equation} \label{sufficient thing in probability condition}
    \lim_{S \rightarrow 0}\sup_{n \in \mathbb{N}}\mathbb{E}\left[\norm{\sy^{n}}^2_{UH,\tau^{M,t}_n \wedge S} - \norm{\sy^n_0}_{U}^2 \right] = 0
\end{equation}
and in particular, 
\begin{equation}\label{qwerty}\lim_{S \rightarrow 0}\sup_{n \in \mathbb{N}}\mathbb{P}\left(\left\{\norm{\sy^{n}}^2_{UH,\tau^{M,t}_n \wedge S} \geq M-1+\norm{\sy^n_0}_{U}^2 \right\}\right) = 0.\end{equation}
\end{proposition}

\begin{proof}
See Appendix I, \ref{proofs section 3 appendix}.
\end{proof}

We are now in a position to apply the Convergence of Random Cauchy Sequences lemma (Lemma \ref{greenlemma}) and would like to put this limiting pair forward as a candidate solution. The convergence is deduced only in the $UH$ norm, though, so we need a little additional work to justify the $HV$ regularity and progressive measurability in $V$. This is the content of the following theorem. 

\begin{theorem} \label{existence of limiting pair theorem}
There exists a stopping time $\tau^{M,t}_{\infty}$, a subsequence $(\sy^{n_l})$ and a process $\sy_\cdot= \sy_{\cdot \wedge \tau^{M,t}_
\infty}$ whereby $\sy_{\cdot}\mathbbm{1}_{\cdot \leq \tau^{M,t}_{\infty}}$ is progressively measurable in $V$ and such that:
\begin{itemize}
    \item $\mathbb{P}\left(\left\{ 0 < \tau^{M,t}_{\infty} \leq \tau^{M,t}_{n_l}\right)\right\} = 1$;
    \item For $\mathbb{P}-a.e.$ $\omega$, $\sy_{\cdot}(\omega)\mathbbm{1}_{\cdot \leq \tau^{M,t}_{\infty}(\omega)} \in  L^2\left([0,T];V\right)$ for all $T>0$;
    \item For $\mathbb{P}-a.e.$ $\omega$, $\sy^{n_l}(\omega) \rightarrow \sy(\omega)$ in $ L^\infty\left([0,\tau^{M,t}_{\infty}(\omega)];U\right) \cap L^2\left([0,\tau^{M,t}_{\infty}(\omega)];H\right)$, i.e. 
\begin{equation} \label{as convergence} \norm{\sy^{n_l}(\omega) - \sy(\omega)}_{UH,\tau^{M,t}_{\infty}(\omega)}^2 \longrightarrow 0;\end{equation}
\item $\sy^{n_l}\rightarrow \sy$ holds in the sense that \begin{equation}\label{expectation convergence}\mathbb{E}\norm{\sy^{n_l} - \sy}_{UH,\tau^{M,t}_{\infty}}^2 \longrightarrow 0.\end{equation}
\end{itemize}
\end{theorem}

\begin{proof}
From Propositions \ref{theorem:cauchygalerkin} and \ref{theorem: probability one}, we can directly apply the Convergence of Random Cauchy Sequences lemma (Lemma \ref{greenlemma}) to the sequence of stopped processes $(\sy^{n}_{\cdot\wedge\tau^{M,t}_n})$ and the spaces $\mathcal{H}_1 = H$, $\mathcal{H}_2 = U$. This guarantees us the existence of our stopping time, subsequence and limit process satisfying the first and third bullet points. It remains to show the progressive measurability, the regularity specified by the second bullet point, and that the convergence (\ref{expectation convergence}) holds. We immediately note that as $\tau^{M,t}_{\infty} \leq \tau^{M,t}_{n_l}$ for every $n_l$ $\mathbb{P}-a.s.$ we need not make any reference to the fact that the processes $\sy^{n_l}$ are stopped at $\tau^{M,t}_{n_l}$. Due to this property, (\ref{cauchy result pt 2}) implies that the sequence $(\sy^{n_l}_\cdot \mathbbm{1}_{\cdot \leq \tau^{M,t}_{\infty}})$ is genuinely Cauchy in $L^2\left[\Omega;L^\infty\left([0,T];U\right) \cap L^2\left([0,T];H\right)\right]$ so it admits a limit process. The limit (\ref{as convergence}) informs us that the sequence converges $\mathbbm{P}-a.s.$ to $\sy\mathbbm{1}_{\cdot \leq \tau^{M,t}_\infty}$, and this must agree with the limit in $L^2\left[\Omega;L^\infty\left([0,T];U\right) \cap L^2\left([0,T];H\right)\right]$ which provides (\ref{expectation convergence}). It should be noted that we could not jump straight to (\ref{expectation convergence}) for the whole sequence $(\sy^n)$ as we are only guaranteed that $\tau^{M,t}_{\infty} \leq \tau^{M,t}_{n_l}$ $\mathbbm{P}-a.s.$ for the subsequence.\\

We look to show the second bullet point in Theorem \ref{existence of limiting pair theorem}. Proposition \ref{theorem:uniformbounds} asserts that the sequence of truncated processes $(\sy^{n_l}_{\cdot}\mathbbm{1}_{\cdot \wedge \tau^{M,t}_{\infty}})$ are uniformly bounded in $L^2\left(\Omega; L^2\left([0,T];V\right)\right)$. Moreover we can extract a weakly convergent subsequence $(\sy^{n_j}_{\cdot \wedge \tau^{M,t}_{\infty}})$ in this space, with limit process (which we shall call $\hat{\sy}$) respecting the same bound. The goal is to show that $\hat{\sy}_{\cdot} = \sy_{\cdot}\mathbbm{1}_{\cdot \wedge \tau^{M,t}_{\infty}}$ as an element of $L^2\left([0,T];V\right)$, $\mathbb{P}-a.s.$. To this end we note that the subsequence $(\sy^{n_j}_{\cdot}\mathbbm{1}_{\cdot \wedge \tau^{M,t}_{\infty}})$ which is weakly convergent to $\hat\sy$ in $L^2\left[\Omega; L^2\left([0,T];V\right)\right]$ is also weakly convergent to $\hat\sy$ in $L^2\left[\Omega; L^2\left([0,T];H\right)\right]$ from the continuous embedding of $V$ into $H$. But we have already established the strong convergence of $(\sy^{n_l}_{\cdot}\mathbbm{1}_{\cdot \wedge \tau^{M,t}_{\infty}})$ to $\sy_{\cdot}\mathbbm{1}_{\cdot \wedge \tau^{M,t}_{\infty}}$ in this space (contained in (\ref{expectation convergence}), and hence the strong and therefore weak convergence of the further subsequence $(\sy^{n_l}_{\cdot}\mathbbm{1}_{\cdot \wedge \tau^{M,t}_{\infty}})$. By the uniqueness of limits in the weak topology we have succeeded in our goal, so for $\mathbb{P}-a.e.$ $\omega$, $\sy_{\cdot}(\omega)\mathbbm{1}_{\cdot \wedge \tau^{M,t}_{\infty}(\omega)} \in L^2\left([0,T];V\right)$. \\

It only remains to show that $\sy_{\cdot}\mathbbm{1}_{\cdot \leq \tau^{M,t}_{\infty}}$ is progressively measurable in $V$. We use an identical argument here, noting that as $(\sy^{n_l}_{\cdot}\mathbbm{1}_{\cdot \leq \tau^{M,t}_\infty})$ is a sequence of progressively measurable processes in $V$, then for each fixed time $T$ we have that each $\sy^{n_l}_{\cdot}\mathbbm{1}_{\cdot \leq \tau^{M,t}_\infty}$ is measurable with respect to the product sigma algebra $\mathcal{F}_T \times \mathcal{B}([0,T])$ and can in fact be considered as a uniformly bounded sequence in the space $L^2\left(\Omega \times [0,T]; V\right)$ where $\Omega \times [0,T]$ is equipped with the described product sigma algebra. Moreover obtaining $\sy_\cdot \mathbbm{1}_{\cdot \leq \tau^{M,t}_{\infty}}$ as a weak limit in this space demonstrates the progressive measurability. The proof is concluded.
\end{proof}

\begin{remark}
Just to be precise once more, in the progressive measurability argument here we understand that the $(\sy^{n_l}_{\cdot}\mathbbm{1}_{\cdot \leq \tau^{M,t}_\infty})$ are identical to the genuinely progressively measurable $(\py^{n_l}_{\cdot}\mathbbm{1}_{\cdot \leq \tau^{M,t}_\infty})$ stipulated in Remark \ref{remark on prog meas equivalence} in the norm of $L^2\left(\Omega \times [0,T]; V\right)$. We then use the same identification for the limit. 
\end{remark}

\begin{theorem} \label{existence of local strong V solution}
The pair $(\sy,\tau^{M,t}_{\infty})$ specified in \ref{existence of limiting pair theorem} is an $H-$valued local strong solution of the equation (\ref{thespde}) as defined in \ref{definitionofregularsolution}.
\end{theorem}

\begin{proof}
After Theorem \ref{existence of limiting pair theorem} it only remains to show the identity (\ref{identityindefinitionoflocalsolution}) in $U$, and the continuity in $H$. We look to first show the identity (\ref{identityindefinitionoflocalsolution}) and use this in conjunction with Proposition \ref{rockner prop} to deduce the continuity. It is sufficient to demonstrate that
\begin{equation}
    \inner{\sy_{s\wedge \tau^{M,t}_{\infty}}}{\phi}_U = \inner{\sy_0}{\phi}_U + \left\langle\int_0^{s\wedge \tau^{M,t}_{\infty}} \mathcal{A}(r,\sy_r)dr, \phi\right\rangle_U + \left\langle\int_0^{s \wedge \tau^{M,t}_{\infty}}\mathcal{G} (r,\sy_r) d\mathcal{W}_r,\phi \right\rangle_U
\end{equation}
holds $\mathbbm{P}-a.s.$ for all $s \geq 0$ and $\phi \in H$ (from the density of $H$ in $U$). To this end we consider the limit $\lim_{n_l \rightarrow \infty}\sy^{n_l}_{s \wedge \tau^{M,t}_{\infty}}$ (which in $U$ is of course $\sy_{s\wedge\tau^{M,t}_{\infty}}$ from (\ref{as convergence})) with the idea to show that \begin{align}
    \label{convergence 1}\lim_{n_l \rightarrow \infty}\inner{\sy^{n_l}_0}{\phi}_U &= \inner{\sy_0}{\phi}_U\\\label{convergence 2}
   \lim_{n_l \rightarrow \infty}\left\langle\int_0^{s \wedge \tau^{M,t}_{\infty}}\mathcal{P}_{n_l}\mathcal{A}(r,\sy^{n_l}_r)dr,\phi\right\rangle_U &= \left\langle\int_0^{s \wedge \tau^{M,t}_{\infty}}\mathcal{A}(r,\sy_r)dr,\phi \right\rangle_U\\
   \lim_{n_l \rightarrow \infty}\left\langle\int_0^{s \wedge \tau^{M,t}_{\infty}}\mathcal{P}_{n_l}\mathcal{G}(r,\sy^{n_l}_r)d\mathcal{W}_r,\phi \right\rangle_U &= \left\langle \int_0^{s \wedge \tau^{M,t}_{\infty}}\mathcal{G}(r,\sy_r)d\mathcal{W}_r,\phi \right\rangle_U \label{convergence 3}.
\end{align}
$\mathbbm{P}-a.s.$ for all $\phi \in H$ (or at least for a further subsequence). Firstly from (\ref{mu2}) we have that $$ \norm{\sy^{n_l}_{0}-\sy_{0}}_U = \norm{(I-\mathcal{P}_{n_l})\sy_{0}}_U \leq \frac{1}{\mu_{n_l}}\norm{\sy_0}_H$$ so in particular (\ref{convergence 1}) holds. To show (\ref{convergence 2}), we consider the term \begin{equation}
    \nonumber J_1:= \mathbbm{E}\left\vert \left \langle \int_0^{s \wedge \tau^{M,t}_{\infty}}\mathcal{P}_{n_l}\mathcal{A}(r,\sy^{n_l}_r)dr,\phi\right\rangle_U- \left\langle  \int_0^{s \wedge \tau^{M,t}_{\infty}}\mathcal{A}(r,\sy_r)dr, \phi \right\rangle_U \right\vert
\end{equation}
and have that 
\begin{align*}
    J_1 & \leq  \mathbbm{E}\int_0^{s \wedge \tau^{M,t}_{\infty}}\left\vert \inner{\mathcal{P}_{n_l}\mathcal{A}(r,\sy^{n_l}_r) - \mathcal{A}(r,\sy_r)}{\phi}_U \right\vert dr\\
    &\leq \mathbbm{E}\int_0^{s \wedge \tau^{M,t}_{\infty}}\left\vert \inner{ \mathcal{P}_{n_l}\left[\mathcal{A}(r,\sy^{n_l}_r) - \mathcal{A}(r,\sy_r)\right]}{\phi}_U \right\vert + \vert\inner{(I - \mathcal{P}_{n_l})\mathcal{A}(r,\sy_r)}{\phi}_U\vert dr\\
    &\leq \mathbbm{E}\int_0^{s \wedge \tau^{M,t}_{\infty}}c_r(1+\norm{\mathcal{P}_{n_l}\phi}_H)\left[K(\sy^{n_l}_r,\sy_r) + \norm{\sy^{n_l}_r}_V + \norm{\sy_r}_V \right]\norm{\sy^{n_l}_r-\sy_r}_Hdr\\ & \qquad \qquad \qquad \qquad \qquad \qquad   + \mathbbm{E}\int_0^{s \wedge \tau^{M,t}_{\infty}}\frac{1}{\mu_{n_l}}\norm{\phi}_H\norm{\mathcal{A}(r,\sy_r)}_U dr\\
    &\leq c(1+\norm{\phi}_H)\left(\mathbbm{E}\int_0^{s \wedge \tau^{M,t}_{\infty}}K(\sy^{n_l}_r,\sy_r) + \norm{\sy^{n_l}_r}_V^2 + \norm{\sy_r}_V^2dr\right)^{\frac{1}{2}}\\& \qquad \cdot \left(\mathbbm{E}\int_0^{s \wedge \tau^{M,t}_{\infty}}\norm{\sy^{n_l}_r-\sy_r}_H^2dr \right)^{\frac{1}{2}} + \frac{\norm{\phi}_H}{\mu_{n_l}}\mathbbm{E}\int_0^{s \wedge \tau^{M,t}_{\infty}}c_rK(\sy_r)\left[1 + \norm{\sy_r}_V\right]dr\\
    &\leq c(1+\norm{\phi}_H)\left(\mathbbm{E}\int_0^{s \wedge \tau^{M,t}_{\infty}}K(\sy^{n_l}_r,\sy_r) + \norm{\sy^{n_l}_r}_V^2 + \norm{\sy_r}_V^2dr\right)^{\frac{1}{2}}\\& \qquad \cdot \left(\mathbbm{E}\int_0^{s \wedge \tau^{M,t}_{\infty}}\norm{\sy^{n_l}_r-\sy_r}_H^2dr \right)^{\frac{1}{2}} + \frac{c\norm{\phi}_H}{\mu_{n_l}}\left(\mathbbm{E}\int_0^{s \wedge \tau^{M,t}_{\infty}}K(\sy_r)\left[1 + \norm{\sy_r}_V^2\right]dr\right)^{\frac{1}{2}}
\end{align*}
having employed the assumptions (\ref{projectionsboundedonH}), (\ref{mu2}), (\ref{111}), (\ref{lastlast assumption}) and using H\"{o}lder's Inequality on the product space. It is swiftly noted that $\sy$ retains the bounds (\ref{galerkinboundsatisfiedbystoppingtime2}) and (\ref{secondresultofuniformbounds}) up until $\tau^{M,t}_{\infty}$ based on the convergences (\ref{as convergence}),(\ref{expectation convergence}) in these spaces, so substituting into the above we arrive at $$J_1 \leq c\left(\mathbbm{E}\int_0^{s \wedge \tau^{M,t}_{\infty}}\norm{\sy^{n_l}_r-\sy_r}_H^2dr \right)^{\frac{1}{2}} + \frac{c}{\mu_{n_l}}$$
at which point taking the limit $n_l \rightarrow \infty$ with (\ref{expectation convergence}) in mind sends this to zero. This is of course not quite the statement (\ref{convergence 2}) that we wished to prove as we have taken the limit in expectation, but we can extract a new subsequence $\sy^{m_j}$ along which (\ref{convergence 2}) holds. Let's continue to work with this subsequence and turn our attention to (\ref{convergence 3}), where we consider \begin{equation}
    \nonumber J_2:= \mathbbm{E}\left\Vert\int_0^{s \wedge \tau^{M,t}_{\infty}}\mathcal{P}_{m_j}\mathcal{G}(r,\sy^{m_j}_r)dr -  \int_0^{s \wedge \tau^{M,t}_{\infty}}\mathcal{G}(r,\sy_r)d\mathcal{W}_r \right\Vert_U^2
\end{equation}
and have that
\begin{align*}
    J_2 &= \mathbbm{E}\int_0^{s \wedge \tau^{M,t}_{\infty}}\sum_{i=1}^\infty\norm{\mathcal{P}_{m_j}\mathcal{G}_i(r,\sy^{m_j}_r) - \mathcal{G}_i(r,\sy_r)}_U^2dr\\
    &\leq 2\mathbbm{E}\int_0^{s \wedge \tau^{M,t}_{\infty}}\sum_{i=1}^\infty\left\Vert \mathcal{P}_{m_j}\left[\mathcal{G}_i(r,\sy^{m_j}_r) - \mathcal{G}_i(r,\sy_r)\right] \right\Vert_U^2 + \norm{(I - \mathcal{P}_{m_j})\mathcal{G}_i(r,\sy_r)}_U^2 dr\\
    &\leq 2\mathbbm{E}\int_0^{s \wedge \tau^{M,t}_{\infty}} c_rK(\sy^{m_j}_r,\sy_r)\norm{\sy^{m_j}_r-\sy_r}_H^2 dr + 2\mathbbm{E}\int_0^{s \wedge \tau^{M,t}_{\infty}}\frac{c}{\mu_{m_j}^2}\norm{\mathcal{G}_i(r,\sy_r)}_H^2dr\\
    &\leq c\mathbbm{E}\int_0^{s \wedge \tau^{M,t}_{\infty}} \norm{\sy^{m_j}_r-\sy_r}_H^2 dr + \frac{c}{\mu_{m_j}^2}\mathbbm{E}\int_0^{s \wedge \tau^{M,t}_{\infty}}1 + \norm{\sy_r}_V^2 dr
\end{align*}
similar to controlling $J_1$ just now with the assumptions (\ref{mu2}), (\ref{111}), (\ref{333}) and applying the It\^{o} Isometry. We conclude here in the same way as (\ref{convergence 2}), showing in fact a stronger convergence which implies (\ref{convergence 3}), ultimately deducing the existence of a further subsequence along which (\ref{convergence 1}), (\ref{convergence 2}) and (\ref{convergence 3}) all hold. The identity (\ref{identityindefinitionoflocalsolution}) thus holds in $U$, and it is clear that we can apply Proposition \ref{rockner prop} to deduce that $\sy_{\cdot \wedge \tau^{M,t}_{\infty}}\in C([0,T];H)$, concluding the proof.

\end{proof}

\subsection{Uniqueness}

Having shown the existence of $H-$valued local strong solutions for a bounded initial condition, we move on now to uniqueness but show this for any given initial condition $\sy_0$ (relieving the $L^\infty(\Omega;H)$ constraint).

\begin{theorem} \label{uniqueness for V valued solutions}
Suppose that $(\sy^1,\tau_1)$ and $(\sy^2,\tau_2)$ are two $H-$valued local strong solutions of the equation (\ref{thespde}) for a given $\mathcal{F}_0-$measurable $\sy_0:\Omega \rightarrow H$, and introduce the notation $$\py:=\sy^1-\sy^2, \qquad  \tau:= \tau_1 \wedge \tau_2.$$
Then \begin{equation} \label{uh oh}\mathbbm{E}\norm{\py}_{UH,\tau}^2 = 0\end{equation} and in particular \begin{equation} \label{oh uh}\mathbbm{P}\left(\left\{\omega \in \Omega: \sy^1_{t \wedge \tau(\omega) }(\omega) =  \sy^2_{t \wedge \tau(\omega)}(\omega)  \quad \forall t \in [0,\infty) \right\}  \right) = 1. \end{equation}
\end{theorem}

\begin{proof}
We follow a similar procedure to that used in Proposition \ref{theorem:cauchygalerkin}. To this end we appreciate that the difference process satisfies the identity
\begin{align*}
    \py_{t \wedge \tau} = \int_0^{t\wedge\tau}\mathcal{A}(s,\sy^1_s) - \mathcal{A}(s,\sy^2_s)ds + \int_0^{t\wedge\tau}\mathcal{G}(s,\sy^1_s) - \mathcal{G}(s,\sy^2_s)d\mathcal{W}_s
\end{align*}
for all $t \geq 0$, $\mathbbm{P}-a.s.$ in $U$. Through the It\^{o} Formula we see that
\begin{align*}
    \norm{\py_{t \wedge \tau}}_U^2 &= 2\int_0^{t\wedge\tau}\inner{\mathcal{A}(s,\sy^1_s) - \mathcal{A}(s,\sy^2_s)}{\py_s}_Uds\\ &+ \int_0^{t\wedge\tau}\sum_{i=1}^\infty\norm{\mathcal{G}_i(s,\sy^1_s) - \mathcal{G}_i(s,\sy^2_s)}_U^2ds + 2\int_0^{t\wedge\tau}\inner{\mathcal{G}(s,\sy^1_s) - \mathcal{G}(s,\sy^2_s)}{\py_s}_Ud\mathcal{W}_s.
\end{align*}
to which we rapidly go through the steps of Proposition \ref{theorem:cauchygalerkin}, introducing analogous notation $$\tilde{\sy}^1_\cdot = \sy^1_\cdot \mathbbm{1}_{\cdot \leq \tau}, \qquad \tilde{\sy}^2_\cdot = \sy^2_\cdot \mathbbm{1}_{\cdot \leq \tau}, \qquad \tilde{\py}_{\cdot} = \py_\cdot \mathbbm{1}_{\cdot \leq \tau}$$ fixing any time $t > 0$ and stopping times $0 \leq \theta_j < \theta_k \leq t$, and applying (\ref{therealcauchy1}) to deduce that 
\begin{align*}
    \norm{\tilde{\py}_{r}}_U^2 + \kappa\int_{\theta_j}^r\norm{\tilde{\py}}_H^2ds &\leq \norm{\tilde{\py}_{\theta_j}}_U^2\\ &+ \int_{\theta_j}^{r}c_s\tilde{K}_2(\tilde{\sy}^1_s,\tilde{\sy}^2_s)\norm{\tilde{\py}_s}_U^2 + 2\int_{\theta_j}^{r}\inner{\mathcal{G}(s,\tilde{\sy}^1_s) - \mathcal{G}(s,\tilde{\sy}^2_s)}{\tilde{\py}_s}_Ud\mathcal{W}_s
\end{align*}
for all $\theta_j \leq r \leq \theta_k$ $\mathbbm{P}-a.s.$. We need to be slightly careful before employing the same procedure, as we no longer have any assumptions towards integrability over the probability space (owing to the fact that we are not restricted to an interval on which bounds such as (\ref{galerkinboundsatisfiedbystoppingtime}) necessarily hold). In order to continue with the same idea we have to introduce a further sequence of stopping times $(\gamma_R)$, \begin{align*}
\gamma_R^j(\omega)&:= R \wedge \inf\left\{s \geq 0: \sup_{r \in [0,s]}\norm{\tilde{\sy}^j_r(\omega)}_U^2 + \int_0^s\norm{\tilde{\sy}^j_r(\omega)}_H^2  \geq R\right\}\\
\gamma_R&:= \gamma_R^1 \wedge \gamma_R^2\end{align*}
with corresponding notation $$\tilde{\sy}^{1,R}_\cdot = \tilde{\sy}^1_\cdot \mathbbm{1}_{\cdot \leq \gamma_R}, \qquad \tilde{\sy}^{2,R}_\cdot = \tilde{\sy}^2_\cdot \mathbbm{1}_{\cdot \leq \gamma_R}, \qquad \tilde{\py}^{R}_\cdot = \tilde{\py}_\cdot \mathbbm{1}_{\cdot \leq \gamma_R}.$$ Multiplying our inequality by the indicator function 
$\mathbbm{1}_{\cdot \leq \gamma_R}$, it is clear that 
\begin{align*}
    \norm{\tilde{\py}^R_{r}}_U^2 &+ \kappa\int_{\theta_j}^r\norm{\tilde{\py}^R_{s}}_H^2ds \leq \norm{\tilde{\py}^R_{\theta_j}}_U^2\\ &+  \int_{\theta_j}^{r}c_s\tilde{K}_2(\tilde{\sy}^{1,R}_s,\tilde{\sy}^{2,R}_s)\norm{\tilde{\py}^R_{s}}_U^2 + 2\int_{\theta_j}^{r}\inner{\mathcal{G}(s,\tilde{\sy}^{1,R}_s) - \mathcal{G}(s,\tilde{\sy}^{2,R}_s)}{\tilde{\py}^R_s}_Ud\mathcal{W}_s
\end{align*}
holds for every $R >0$. The boundedness of the processes $\tilde{\sy}^{1,R},\tilde{\sy}^{2,R}$ in conjunction with the assumption (\ref{therealcauchy2}) affords us the right to take expectation and apply the BDG Inequality once we take the supremum, landing us at the inequality
\begin{align*}
    \mathbbm{E}\norm{\tilde{\py}^R}_{UH,\theta_j,\theta_k}^2 &\leq c\mathbbm{E}\norm{\tilde{\py}^R_{\theta_j}}_U^2 \\&+ c\mathbbm{E}\int_{\theta_j}^{\theta_k}c_s\tilde{K}_2(\tilde{\sy}^{1,R}_s,\tilde{\sy}^{2,R}_s)\norm{\tilde{\py}^R_{s}}_U^2 + 2c\mathbbm{E}\left(\int_{\theta_j}^{\theta_k}c_s\tilde{K}_2(\tilde{\sy}^{1,R}_s,\tilde{\sy}^{2,R}_s)\norm{\tilde{\py}^R_{s}}_U^4 ds\right)^{\frac{1}{2}}.
\end{align*}
having scaled again to remove the $\kappa$ and use the $UH$ norm. We had precisely the same terms in the proof of Proposition \ref{theorem:cauchygalerkin}, so following the same steps used to reach (\ref{lelabel}) we see that for some constant $\hat{c}$,
\begin{align} \label{someconstanthatc} \mathbbm{E}\norm{\tilde{\py}^R}_{UH,\theta_j,\theta_k}^2 \leq \hat{c}\mathbbm{E}\left(\norm{\tilde{\py}^R_{\theta_j}}_U^2 + \int_{\theta_j}^{\theta_k}\tilde{K}_2(\tilde{\sy}^{1,R}_s,\tilde{\sy}^{2,R}_s)\norm{\tilde{\py}^R_s}_U^2ds\right).\end{align} We again apply the Stochastic Gr\"{o}nwall lemma (Lemma \ref{gronny}) for the processes $$\boldsymbol{\phi} = \norm{\tilde{\sy}^{1,R}-\tilde{\sy}^{2,R}}_U^2, \qquad \boldsymbol{\psi} = \norm{\tilde{\sy}^{1,R}-\tilde{\sy}^{2,R}}_H^2, \qquad \boldsymbol{\eta} = \tilde{K}_2(\tilde{\sy}^{1,R},\tilde{\sy}^{2,R})$$ and critically for $\tilde{c}=0$ to see that $$\mathbbm{E}\norm{\tilde{\py}^R}_{UH,\theta_j,\theta_k}^2 \leq C\mathbbm{E}\norm{\tilde{\py}^R_{0}}_U^2 = 0.$$ We can drop the dependence on $\gamma_R$ by observing that the sequence $(\gamma_R)$ tends to infinity $\mathbbm{P}-a.s.$ as $R \rightarrow \infty$. Moreover the sequence of random variables $$\left(\norm{\tilde{\py}^R}_{UH,\theta_j,\theta_k}^2 \right)$$ is monotone increasing in $R$, and convergent to $\norm{\tilde{\py}}_{UH,\theta_j,\theta_k}^2$ $\mathbbm{P}-a.s.$. Thus we may apply the monotone convergence theorem to this sequence of random variables to see that
\begin{align*}
    \mathbbm{E}\norm{\tilde{\py}}_{UH,\theta_j,\theta_k}^2 &= \lim_{R \rightarrow \infty}\mathbbm{E}\norm{\tilde{\py}^R}_{UH,\theta_j,\theta_k}^2\\
    &= 0.
\end{align*}
This is equivalent to the statement (\ref{uh oh}), which implies (\ref{oh uh}) by noting that $$\mathbbm{E}\sup_{r \in [0,t]}\norm{\sy^1_{r\wedge\tau}-\sy^2_{r\wedge\tau}}_U^2 = 0.$$

\end{proof}

\subsection{Maximality for the Bounded Initial Condition} \label{subsection: maximality for bounded in V}

In this subsection we show the existence and uniqueness of an $H-$valued maximal strong solution for the bounded initial condition (\ref{boundedinitialcondition}) as defined in \ref{V valued maximal definition}, \ref{v valued maximal unique}. Moreover we prove Theorem \ref{theorem1} for this bounded initial condition, and pass to the unbounded case in Subsection \ref{subsection maximality for unbounded}.\\
To this end we define $\mathscr{X}$ as the set of all stopping times $\sigma$ such that there exists a process $\sy$ for which $(\sy,\sigma)$ is an $H-$valued local strong solution. We also define $\mathscr{Y}$ as the set of all stopping times given by the $\mathbbm{P}-a.s.$ limit of monotone increasing elements of $\mathscr{X}$. We prove the existence of a maximal solution by showing that the maximum of any two elements of $\mathscr{X}$ is again in $\mathscr{X}$, a property which we use for the sequences in $\mathscr{X}$ to bound sequences in $\mathscr{Y}$ which then enables an application of Zorn's Lemma to deduce a maximal element.

\begin{remark}
For the bounded initial condition (\ref{boundedinitialcondition}), the set $\mathscr{X}$ is non-empty from Theorem \ref{existence of local strong V solution}.
\end{remark}

\begin{lemma} \label{lemma for maximality}
 For any $\sigma_1,\sigma_2 \in \mathscr{X}$ we have that $\sigma_1 \vee \sigma_2 \in \mathscr{X}.$
\end{lemma}

\begin{proof}
It is clear that $\sigma_1 \vee \sigma_2$ is again a $\mathbbm{P}-a.s.$ positive stopping time, so we must simply show the existence of a $\sy$ such that $(\sy,\sigma_1 \vee \sigma_2)$ is an $H-$valued local strong  solution. By definition of $\mathscr{X}$ we have that $(\sy^1,\sigma_1)$ and $(\sy^2, \sigma_2)$ are such solutions for some processes $\sy^1,\sy^2$. From the uniqueness result Theorem \ref{uniqueness for V valued solutions} we have that $\sy^1_{\cdot \wedge \sigma_1 \wedge \sigma_2}$ and $\sy^2_{\cdot \wedge \sigma_1 \wedge \sigma_2}$ are indistinguishable, which ensures that the definition of $\sy$ at $\mathbbm{P}-a.e.$ $\omega$ by $$\sy(\omega):=\sy^k(\omega) \qquad \textnormal{for $\sigma_1(\omega) \vee \sigma_2(\omega) = \sigma_k(\omega)$}$$ is consistent, or more specifically 
$$
   \sy(\omega):=
   \begin{cases}
   \sy^1(\omega)=\sy^2(\omega) & \textnormal{if $\sigma_1(\omega) = \sigma_2(\omega)$}\\
   \sy^1(\omega) & \textnormal{if $\sigma_2(\omega) < \sigma_1(\omega)$}\\
   \sy^2(\omega) & \textnormal{if $\sigma_1(\omega) < \sigma_2(\omega).$}\\
   \end{cases}
$$
 It is clear that $\sy$ satisfies the pathwise properties required in Definition \ref{definitionofregularsolution}. Progressive measurability follows too as the sigma algebra generated by $\sy$ is contained in the maximum of those generated by $\sy^1$ and $\sy^2$, which up until any $T$ is contained in $\mathcal{F}_T \times \mathcal{B}\left([0,T]\right)$ by progressive measurability of the processes $\sy^1$, $\sy^2$ and standard properties of the sigma algebra.
\end{proof}

\begin{theorem} \label{existence of maximal solution}
For the bounded initial condition (\ref{boundedinitialcondition}), there exists an $H-$valued maximal strong solution of the equation (\ref{thespde}). 
\end{theorem}

\begin{proof}
We wish to show that there exists a $\Theta \in \mathscr{Y}$ such that for any $\Gamma \in \mathscr{Y}$, $\Theta \leq \Gamma$ $\mathbbm{P}-a.s.$ implies $\Theta = \Gamma$ $\mathbbm{P}-a.s.$. This will be sufficient to conclude the proof, as for $\Theta$ given by the limit of $(\sigma_j)$ with corresponding solutions $(\sy^j,\sigma_j)$, our process $\sy$ can be consistently defined on $[0,\Theta)$ through \begin{equation}\label{constructing sy}\sy(\omega):= \sy^j(\omega) \qquad \textnormal{on $[0,\sigma_j(\omega)]$}\end{equation} in the same manner as in Lemma \ref{lemma for maximality}. We apply Zorn's Lemma on $\mathscr{Y}$, which we understand to be a partially ordered set for the relation $'\leq'$ defined by $\Gamma_1 \leq \Gamma_2$ if and only if for $\mathbbm{P}-a.e.$ $\omega$, $\Gamma_1(\omega) \leq \Gamma_2(\omega)$. The result would then follow from Zorn's Lemma if we can prove that for every sequence $(\Gamma_k)$ in $\mathscr{Y}$ with $\Gamma_1 \leq \dots \leq \Gamma_k \leq \Gamma_{k+1} \leq \dots$ there exists a $\Lambda \in \mathscr{Y}$ whereby $\Gamma_k \leq \Lambda$ for all $k$. Suppose now that each $\Gamma_k$ is given by the increasing limit of $(\sigma^k_j)$. Let's define the sequence $(\gamma_n)$ as $$\gamma_n:= \bigvee_{k=1}^n \sigma^k_n$$ which by virtue of Lemma \ref{lemma for maximality} is a sequence in $\mathscr{X}$. Inherited from each $(\sigma^k_n)$, note that this is a $\mathbbm{P}-a.s.$ monotone increasing sequence and therefore admits a limiting stopping time which we claim to be our $\Lambda$. By definition $\Lambda \in \mathscr{Y}$, and for each fixed $k$ we see that $\gamma_n \geq \sigma^k_n$ for $n \geq k$. We thus have that the limit of the $(\gamma_n)$ dominates the limit of the $(\sigma^k_n)$, which proves the result. 
\end{proof}

\begin{remark}
An alternative approach could be taken without appealing to Zorn's Lemma, which is given in \cite{doob2012measure} Section 18 pp.71. The method is very involved and we would have to work to extract the results we need from the proof of Doob's given theorem. We find our method to be simpler and more direct, hence of our preference. 
\end{remark}

\begin{theorem} \label{uniqueness of maximal solution}
Let $(\sy,\Theta)$ be an $H-$valued maximal strong solution of the equation (\ref{thespde}). Then $(\sy,\Theta)$ is unique in the sense of Definition \ref{v valued maximal unique}.
\end{theorem}

\begin{proof}
We start by showing that $\Theta = \Gamma$ $\mathbbm{P}-a.s.$. Suppose that $(\theta_j)$, $(\gamma_j)$ are the sequences in $\mathscr{X}$ convergent to $\Theta,\Gamma$ respectively as stipulated in Definition \ref{V valued maximal definition}. From Lemma \ref{lemma for maximality} we have that the sequence $(\theta_j \vee \gamma_j)$ lives in $\mathscr{X}$, and is clearly monotone increasing and convergent to $\Theta \vee \Gamma$ $\mathbbm{P}-a.s.$. But then $\Theta \vee \Gamma \in \mathscr{Y}$, which by the fact that $\Theta, \Gamma \leq \Theta \vee \Gamma$ $\mathbbm{P}-a.s.$ and the definition of the maximal solution, $\Theta = \Theta \vee \Gamma = \Gamma$ $\mathbbm{P}-a.s.$ 
The first part is complete. The second then follows from Theorem \ref{uniqueness for V valued solutions} as seen in the consistent construction (\ref{constructing sy}). 
\end{proof}

The way we chose to define the maximal time $\Theta$ did not exclude the possibility that some $\Gamma \in \mathscr{Y}$ could be such that $\Gamma > \Theta$ on a set of positive measure, which is often excluded from the definition: see \cite{hairer2009introduction}, \cite{crisan2019solution}. The following corollary shows that such a scenario cannot occur. 

\begin{corollary} \label{greatness corollary}
Let $(\sy,\Theta)$ be the unique $H-$valued maximal strong solution of the equation (\ref{thespde}). Then for any $\Gamma \in \mathscr{Y}$, $\Theta \geq \Gamma$ $\mathbbm{P}-a.s.$. 
\end{corollary}

\begin{proof}
Via the same arguments we have that $\Theta = \Theta \vee \Gamma$ $\mathbbm{P}-a.s.$ which concludes the proof.
\end{proof}

Corollary \ref{greatness corollary} is necessary in establishing the blow-up criterion (\ref{blowupproperty}) below. We consider the maximal solution up until the minimum of a first hitting time in the $HV$ norm and the maximal time, we show that a solution exists up until this time, and deduce from Lemma \ref{extension lemma} and Corollary \ref{greatness corollary} that it must therefore be the first hitting time (as the solution can be extended, but no solution can exist beyond the maximal time on a set of positive measure). So the maximal time is greater than the first hitting time for any hitting threshold, from which the result is deduced. The details are given in the following. 

\begin{lemma} \label{first lemma to blow up}
Let $(\sy,\Theta)$ be the $H-$valued maximal strong solution of the equation (\ref{thespde}) and $\tau$ be a stopping time. Then if $\rho:= \tau \wedge \Theta$ is in $\mathscr{X}$ we have that $$\mathbbm{P}\left(\left\{\omega \in \Omega: \tau(\omega) < \Theta(\omega)     \right\}\right) = 1.$$
\end{lemma}

\begin{proof}
Suppose the contrary, that is there exists a set $D \subset \Omega$ on which $\rho(\omega) = \Theta(\omega)$ and $\mathbbm{P}(D) > 0$. By Lemma \ref{extension lemma} there exists a new stopping time $\gamma$ such that $\gamma > \rho$ $\mathbbm{P}-a.s.$ and $\gamma \in \mathscr{X}$. Evidently then $\gamma \in \mathscr{Y}$ as we can simply take the constant sequence, but then Corollary \ref{greatness corollary} asserts that $\Theta \geq \gamma$ $\mathbbm{P}-a.s.$. We have thus reached our contradiction as on $D$, $\Theta = \rho < \gamma$. 
\end{proof}

We now look to construct a solution up until the minimum of the first hitting time and the maximal time, which is a somewhat technical task. Of course the point is to prove that this is in fact a solution only up until the first hitting time which is strictly less than the maximal time, but a priori we do not know this and so immediately we would have to ask how the maximal solution $\sy$ is defined at $\Theta$ (by definition $\sy$ is only defined on $[0,\Theta)$). We circumvent the issue by constructing a new process which must agree with $\sy$ on $[0,\Theta)$ and is continuous on $[0,\infty)$.\\

To simplify proceedings we make the details of this setting concrete before stating the result. $(\sy,\Theta)$ continues to be the $H-$valued maximal strong solution of the equation (\ref{thespde}) and  let $(\theta_j)$ be the sequence stipulated in Definition \ref{V valued maximal definition}. We suppose without loss of generality that $\sy$ is a process defined on $[0,\infty)$ by $\sy_r(\omega) = 0$ for $\Theta(\omega) \leq r$, in order for the first hitting time to be well defined. Indeed for any $M > 1$ and $t>0$ we define \begin{equation}
    \label{tauMt no n} \tau^M_t(\omega):=t \wedge \inf\left\{s \geq 0: \norm{\sy(\omega)}_{HV,s}^2 \geq M + \norm{\sy_0(\omega)}_H^2 \right\}
\end{equation} and proceed to dissect this definition. Firstly note that $\sy \in C\left([0,\Theta);H\right)$ as $\sy_{
\cdot \wedge \theta_j}\in C\left([0,\theta_j ];H\right)$ $\mathbbm{P}-a.s.$ for every $j$ and $\theta_j \rightarrow \Theta$. As we set $\sy$ to be zero at $\Theta$ then $\sup_{r \in [0,\Theta(\omega))}\norm{\sy_r(\omega)}_H^2 = \sup_{r \in [0,\Theta(\omega)]}\norm{\sy_r(\omega)}_H^2$ so the process $$\mathbf{\eta}_\cdot:=\norm{\sy}_{HV,\cdot}^2$$ doesn't just belong to $C\left([0,\Theta);\R\right)$, but would only fail to be continuous at $\Theta$ if $\lim_{r \rightarrow \Theta}\mathbf{\eta}_r = \infty$. The point, therefore, is that there are no jump type discontinuities so the threshold of the first hitting time (\ref{tauMt no n}) is met continuously, i.e. $\mathbf{\eta}\in C\left([0,\tau^M_t];\R\right)$. In particular $\tau^M_t$ is a well defined stopping time and $\sy$ satisfies the bound \begin{equation} \label{control on sy} \norm{\sy}_{HV,\tau^{M}_t}^2 \leq M + \norm{\sy_0}_H^2\end{equation}analogous to (\ref{first bound on galerkin}). Towards an application of Lemma \ref{first lemma to blow up} we introduce the stopping time \begin{equation}\label{rhoMt}\rho^M_t:=\tau^M_t \wedge \Theta.\end{equation} Through the same reasoning that the terms of (\ref{identityindefinitionoflocalsolution}) are well defined, and the control (\ref{control on sy}), we are justified in defining the process \begin{equation}\label{definition of py}\py_r := \sy_0 + \int_0^{r \wedge \rho^M_t}\mathcal{A}(s,\sy_s)ds + \int_0^{r \wedge \rho^M_t}\mathcal{G}(s,\sy_s)d\mathcal{W}_s\end{equation} in $U$. Similarly to Remark \ref{remark1} we have the property \begin{equation}\label{remark1 property} \py_{\cdot}=\py_{\cdot \wedge \rho^M_t}.
\end{equation}

\begin{proposition} \label{prop to blow up}
Let $(\py,\rho^M_t)$ be as defined in (\ref{definition of py}),(\ref{rhoMt}). Then $(\py,\rho^M_t)$ is an $H-$valued local strong solution of the equation (\ref{thespde}).
\end{proposition}

\begin{proof}
Observe that for any time $r < \rho^M_t$ $\mathbbm{P}-a.s.$, 
\begin{align}
    \nonumber \py_r &= \lim_{j \rightarrow \infty}\py_{r \wedge \theta_j}\\  \nonumber &= \lim_{j \rightarrow \infty}\left[\sy_0 + \int_0^{r \wedge \theta_j}\mathcal{A}(s,\sy_s)ds + \int_0^{r \wedge \theta_j}\mathcal{G}(s,\sy_s)d\mathcal{W}_s \right]\\ \nonumber 
    &= \lim_{j \rightarrow \infty} \sy_{r \wedge \theta_j}\\ \label{lastlastlastalign} &= \sy_r
\end{align}
taking the limits $\mathbbm{P}-a.s.$ in $U$, using that $\sy \in C\left([0,\Theta);H\right)$ so in particular $\sy \in C\left([0,\rho^M_t);U\right)$ and similarly $\py \in C\left([0,\infty);U\right)$ as the sum of integrals in $U$. Combining this with the construction (\ref{definition of py}) we see that $$\py_r := \py_0 + \int_0^{r \wedge \rho^M_t}\mathcal{A}(s,\py_s)ds + \int_0^{r \wedge \rho^M_t}\mathcal{G}(s,\py_s)d\mathcal{W}_s$$ so $(\py,\rho^M_t)$ satisfies the required identity (\ref{identityindefinitionoflocalsolution}). In the same vein we have that for $\mathbbm{P}-a.e.$ $\omega$ and every $T>0$, $\py_{\cdot}(\omega)\mathbbm{1}_{\cdot \leq \rho^M_t(\omega)}$ is identical to $\sy_{\cdot}(\omega)\mathbbm{1}_{\cdot \leq \rho^M_t(\omega)}$ as an element of $L^2\left([0,T];V\right)$ where we know the latter belongs to this space from (\ref{control on sy}). The continuity is more involved; as specified we know at least that $\py(\omega) \in C\left([0,\rho^M_t(\omega));H\right)$ but we must address a potential discontinuity at $\rho^M_t$. To this end we claim that for $\mathbbm{P}-a.e.$ $\omega$, \begin{equation}\label{pysyequivalence} \py_{\cdot \wedge \theta_j(\omega)}(\omega) = \sy_{\cdot \wedge \theta_j(\omega) \wedge \rho^M_t(\omega)}(\omega).\end{equation} We consider two cases, the first that $\theta_j(\omega) < \rho^M_t(\omega)$. In this case the property (\ref{pysyequivalence}) is immediate by simply replacing $r$ with $\theta_j(\omega)$ in (\ref{lastlastlastalign}). The alternative is that $\theta_j(\omega) \geq \rho^M_t(\omega)$, for which we note that $\sy_{\cdot \wedge \theta_j \wedge \rho^M_t}$ is again a local strong solution and therefore $\sy_{\cdot \wedge \theta_j(\omega) \wedge \rho^M_t(\omega)}(\omega) \in C\left([0,t];H\right)$ by definition. Moreover in this case, $\sy_\cdot(\omega) \in C\left([0,\rho^M_t(\omega)];H\right)$ so $\sy_\cdot(\omega)  \in C\left([0,\rho^M_t(\omega)];U\right)$. Recalling again that $\py_\cdot(\omega) \in C\left([0,\infty);U\right)$ and $\sy_r(\omega) = 
\py_r(\omega)$ for $r < \rho^M_t(\omega)$, we see that $\sy$ and $\py$ do not just agree up until $\rho^M_t(\omega)$ but must also agree at this time due to the continuity in $U$. This is summarised as the equivalence $\py_{\cdot}(\omega) = \sy_{\cdot \wedge \rho^M_t(\omega)}(\omega)$. Moreover from (\ref{remark1 property}) and being in the case $\theta_j(\omega) \geq \rho^M_t(\omega)$ we have that $$ \py_{\cdot}(\omega) = \py_{\cdot \wedge \rho^M_t(\omega)}(\omega) = \py_{\cdot \wedge \rho^M_t(\omega) \wedge \theta_j(\omega)}(\omega) =  \py_{\cdot \wedge \theta_j(\omega)}(\omega)$$ and similarly $\sy_{\cdot \wedge \rho^M_t(\omega)}(\omega) = \sy_{\cdot \wedge \theta_j(\omega) \wedge \rho^M_t(\omega)}(\omega).$ Putting all of this together, $$\py_{\cdot \wedge \theta_j(\omega)}(\omega) = \py_{\cdot}(\omega) = \sy_{\cdot \wedge \rho^M_t(\omega)}(\omega) = \sy_{\cdot \wedge \theta_j(\omega) \wedge \rho^M_t(\omega)}(\omega)$$ justifying (\ref{pysyequivalence}).\\

We shall use this property to first show the progressive measurability in $V$, which follows from the fact that the sequence $\py_{\cdot \wedge \theta_j}\mathbbm{1}_{\cdot \wedge \rho^M_t} = \sy_{\cdot \wedge \theta_j}\mathbbm{1}_{\cdot \wedge \rho^M_t}$ is uniformly bounded in the space $L^2\left(\Omega \times [0,T]; V\right)$ courtesy of (\ref{control on sy}). Once more $\Omega \times [0,T]$ is equipped with product sigma algebra $\mathcal{F}_T \times [0,T]$ using the progressive measurability of $\sy_{\cdot \wedge \theta_j}\mathbbm{1}_{\cdot \wedge \rho^M_t}$. We use a similar argument to that seen in Theorem \ref{existence of limiting pair theorem}, deducing a weakly convergent subsequence in this space which gives rise to a progressively measurable process in $V$. To show that this limit is in fact $\py_{\cdot}\mathbbm{1}_{\cdot \wedge \rho^M_t}$, we appreciate that the convergence $\py_{\cdot \wedge \theta_j} \rightarrow \py$ $\mathbbm{P}-a.s.$ in the space $C([0,T];U)$ implies that of $\py_{\cdot \wedge \theta_j}\mathbbm{1}_{\cdot \wedge \rho^M_t} \rightarrow \py_{\cdot}\mathbbm{1}_{\cdot \wedge \rho^M_t}$ $\mathbbm{P}-a.s.$ in $L^2([0,T];U)$. Furthermore this limit holds in $L^2\left(\Omega; L^2([0,T];U)\right)$ by applying the dominated convergence theorem to the sequence $\left((\norm{\py_{\cdot \wedge \theta_j}\mathbbm{1}_{\cdot \wedge \rho^M_t} - \py_{\cdot}\mathbbm{1}_{\cdot \wedge \rho^M_t}}_{L^2([0,T];U)}^2\right)$ with domination coming from (\ref{control on sy}). We already know that $\py$ is progressively measurable in $U$ as a continuous and adapted process in this space (adaptedness follows from the convergence $\sy_{r \wedge \theta_j \wedge \rho^M_t} \rightarrow \py_r$ $\mathbbm{P}-a.s.$) so $\py_{\cdot \wedge \theta_j}\mathbbm{1}_{\cdot \wedge \rho^M_t} \rightarrow \py_{\cdot}\mathbbm{1}_{\cdot \wedge \rho^M_t}$ holds too in $L^2\left(\Omega \times [0,T]; U\right)$. The limit trivially holds weakly in this space as well, and must agree with the weak limit taken in $L^2\left(\Omega \times [0,T]; V\right)$ from the continuous embedding $V \xhookrightarrow{} U$, so by uniqueness of limits in the weak topology the progressive measurability is shown. Moreover similarly to Theorem \ref{existence of local strong V solution} we apply Proposition \ref{rockner prop} to deduce that $\py_\cdot \in C([0,T];H)$ $\mathbbm{P}-a.s.$ and the result is shown.
\end{proof}

We have now all but proved Theorem \ref{theorem1} in the case of the bounded initial condition. The following Theorem summarises the work of this subsection and rounds off the proof.

\begin{theorem} \label{blow up property theorem}
Let $(\sy,\Theta)$ be the $H-$valued maximal strong solution of the equation (\ref{thespde}). Then at $\mathbbm{P}-a.e.$ $\omega$ for which $\Theta(\omega) < \infty$, we have that \begin{equation}\label{blowupproperty}\sup_{s \in [0,\Theta(\omega))}\norm{\sy(\omega)}_{HV,s}^2 = \infty.\end{equation} Furthermore $\tau^M_t$ given in (\ref{tauMt no n}) is well defined and such that $(\sy_{\cdot \wedge \tau^M_t},\tau^M_t)$ is an $H-$valued local strong solution of the equation (\ref{thespde}). 
\end{theorem}

\begin{remark}
The property (\ref{blowupproperty}) is precisely (\ref{actualblowup}). 
\end{remark}

\begin{proof}
Proposition \ref{prop to blow up} informs us that $\rho^M_t \in \mathscr{X}$, so we can apply Lemma \ref{first lemma to blow up} to deduce that $$\mathbbm{P}\left(\left\{\omega \in \Omega: \tau^M_t(\omega) < \Theta(\omega)     \right\}\right) = 1.$$ Of course $M>1, t>0$ were arbitrary so from the characterisation of $\tau^M_t$ we observe that $$ \sup_{s \in [0,\Theta(\omega))}\norm{\sy(\omega)}_{HV,s}^2 > M + \norm{\sy_0(\omega)}_H^2$$ for $\mathbbm{P}-a.e.$ $\omega$, from which we infer (\ref{blowupproperty}). In addition we must have that $\rho^M_t = \tau^M_t$ so $(\py,\tau^M_t)$ is an $H-$valued local strong solution and therefore $(\sy_{\cdot \wedge \tau^M_t},\tau^M_t)$ is too, immediate from (\ref{pysyequivalence}) and taking $j$ large enough so that $\theta_j(\omega) > \tau^M_t(\omega)$. 
\end{proof}




\subsection{Existence and Maximality for an Unbounded Initial Condition} \label{subsection maximality for unbounded}

In the previous subsection we showed the existence and maximality of an  $H-$valued local strong solution of the equation (\ref{thespde}) for any given $H$-valued initial condition $\sy_0$ with $\sy_0 \in L^{\infty}(\Omega;H)$. We now show that such solutions exist for an arbitrary $H$-valued initial condition (not necessarily bounded), proving Theorem \ref{theorem1}. 

\begin{theorem} \label{v valued existence result}
For any given $\mathcal{F}_0-$measurable $\sy_0:\Omega \rightarrow H$, there exists an $H-$valued maximal strong solution $(\sy,\Theta)$ of the equation (\ref{thespde}).
\end{theorem}

\begin{proof}
The idea is to piece together solutions obtained for the bounded initial condition, as seen in \cite{glatt2012local} (and generalised from \cite{glatt2009strong}). For any given $k \in \N \cup \{0\}$ we know that there exists a unique $H-$valued maximal strong solution $(\sy^k,\Theta^k)$ for the initial condition $\sy_0\mathbbm{1}_{\{k \leq \norm{\sy_0}_H < k+1\}}$, and we claim that the pair $(\sy,\Theta)$ defined at each time $t \in [0,T]$ and $\omega \in \Omega$ by $$ \sy_t(\omega):= \sum_{k=1}^\infty \sy^k_t(\omega)\mathbbm{1}_{\{k \leq \norm{\sy_0(\omega)}_H < k+1\}}, \qquad \Theta(\omega):=\sum_{k=1}^\infty \Theta^k(\omega)\mathbbm{1}_{\{k \leq \norm{\sy_0(\omega)}_H < k+1\}}$$
is our desired solution for the initial condition $\sy_0$ (where for each fixed $\omega$ the infinite sum is simply a single element). The first task is to show the existence of a sequence of stopping times $(\theta_j)$ which are $\mathbbm{P}-a.s.$ monotone increasing and convergent to $\Theta$, whereby $(\sy_{\cdot \wedge \theta_j},\theta_j)$ is a local strong $H-$valued solution of the equation (\ref{thespde}) for each $j$. There is only one natural choice for the $\theta_j$, which is using that for each fixed $k$ there is a sequence $(\theta^k_j)$ with the properties above for the maximal solution $(\sy^k,\Theta^k)$, and constructing $$\theta_j(\omega):=\sum_{k=1}^\infty \theta^k_j(\omega)\mathbbm{1}_{\{k \leq \norm{\sy_0(\omega)}_H < k+1\}}.$$ Immediately it's clear that for $\mathbbm{P}-a.e.$ $\omega$, $(\sy(\omega),\Theta(\omega)) =(\sy^k(\omega),\Theta^k(\omega))$ for some $k$ and similarly $(\sy(\omega),\theta_j(\omega)) =(\sy^k(\omega),\theta^k_j(\omega))$. The $\mathbbm{P}-a.s.$ monotonicity and convergence $\theta_j \rightarrow \Theta$ comes promptly from the construction, and for each $j$ the pair $(\sy(\omega),\theta_j(\omega))$ inherits the pathwise properties of a solution from the solutions to the bounded problems. In particular $\theta_j$ is a $\mathbbm{P}-a.s.$ positive stopping time and for $\mathbbm{P}-a.e.$ $\omega$, $\sy_{\cdot \wedge \theta_j(\omega)}(\omega) \in C\left([0,T];H\right)$ and  $\sy_{\cdot}(\omega)\mathbbm{1}_{\cdot \leq\theta_j(\omega)} \in L^2\left([0,T];V\right)$ for all $T>0$. In assessing the identity (\ref{identityindefinitionoflocalsolution}), let's now introduce the more compact notation $$A_k:= \left\{ \omega \in \Omega: k \leq \norm{\sy_0(\omega)}_H < k+1\right\}.$$ Then $\mathbbm{P}-a.s.$ in $A_k$ we have the identity 
$$\sy^k_{t} = \sy^k_0 + \int_0^{t\wedge \theta_j^k} \mathcal{A}(s,\sy^k_s)ds + \int_0^{t \wedge \theta_j^k}\mathcal{G} (s,\sy^k_s) d\mathcal{W}_s$$ or equivalently on this set that
$$\sy_{t} = \sy_0 + \int_0^{t\wedge \theta_j} \mathcal{A}(s,\sy_s)ds + \int_0^{t \wedge \theta_j^k}\mathcal{G} (s,\sy^k_s) d\mathcal{W}_s$$ though we have to be a bit more precise with the stochastic integral. We need to justify that $$\mathbbm{1}_{A_k}\int_0^{t \wedge \theta_j^k}\mathcal{G} (s,\sy^k_s) d\mathcal{W}_s = \mathbbm{1}_{A_k}\int_0^{t \wedge \theta_j}\mathcal{G} (s,\sy_s) d\mathcal{W}_s$$ $\mathbbm{P}-a.s.$,  which we do via the manipulations
\begin{align*}
    \mathbbm{1}_{A_k}\int_0^{t \wedge \theta_j^k}\mathcal{G} (s,\sy^k_s) d\mathcal{W}_s &= \mathbbm{1}_{A_k}\int_0^{t} \mathbbm{1}_{\theta_j^k}\mathcal{G} (s,\sy^k_s) d\mathcal{W}_s\\
    &= \int_0^{t} \mathbbm{1}_{A_k}\mathbbm{1}_{s \leq \theta_j^k}\mathcal{G} (s,\sy^k_s) d\mathcal{W}_s\\
    &= \int_0^{t} \mathbbm{1}_{A_k}\mathbbm{1}_{A_k \cap \{s \leq \theta_j^k\}}\mathcal{G} (s,\sy^k_s\mathbbm{1}_{A_k}) d\mathcal{W}_s\\
    &= \int_0^{t} \mathbbm{1}_{A_k}\mathbbm{1}_{A_k \cap \{s \leq \theta_j^k\}}\mathcal{G} (s,\sy^k_s\mathbbm{1}_{A_k}) d\mathcal{W}_s\\
    &= \int_0^{t} \mathbbm{1}_{A_k}\mathbbm{1}_{A_k \cap \{s \leq \theta_j\}}\mathcal{G} (s,\sy_s\mathbbm{1}_{A_k}) d\mathcal{W}_s\\
    &= \int_0^{t} \mathbbm{1}_{A_k}\mathbbm{1}_{s \leq \theta_j}\mathcal{G} (s,\sy_s) d\mathcal{W}_s\\
     &= \mathbbm{1}_{A_k}\int_0^{t \wedge \theta_j}\mathcal{G} (s,\sy_s) d\mathcal{W}_s\\
\end{align*}
where we require that $\mathbbm{1}_{A_k}$ is $\mathcal{F}_0-$measurable which is owing to the $\mathcal{F}_0-$measurability of $\sy_0$, and then apply Proposition 1.6.14 of \cite{goodair2022stochastic}. Therefore the identity (\ref{identityindefinitionoflocalsolution}) holds $\mathbbm{P}-a.s.$ on every $A_k$ hence $\mathbbm{P}-a.s.$ on the whole of $\Omega$. To conclude that $(\sy,\theta_j)$ is a local strong solution for the initial condition $\sy_0$ it only remains to show that $\sy_{\cdot}\mathbbm{1}_{\cdot \leq \theta_j}$ is progressively measurable in $V$, which we deduce from
\begin{align*}
    \sy_{t}(\omega)\mathbbm{1}_{t \leq \theta_j(\omega)} = \sum_{k=1}^\infty \sy^k_{t}(\omega)\mathbbm{1}_{t \leq \theta_j(\omega)}\mathbbm{1}_{\{k \leq \norm{\sy_0(\omega)}_H < k+1\}} =  \sum_{k=1}^\infty \sy^k_{t}(\omega)\mathbbm{1}_{t \leq \theta_j^k(\omega)}\mathbbm{1}_{\{k \leq \norm{\sy_0(\omega)}_H < k+1\}}
\end{align*}
by construction, or equivalently that 
$$\sy_{\cdot}\mathbbm{1}_{\cdot \leq \theta_j} =  \sum_{k=1}^\infty \sy^k_{\cdot}\mathbbm{1}_{\cdot \leq \theta_j^k}\mathbbm{1}_{\{k \leq \norm{\sy_0}_H < k+1\}}$$ for the limit taken pointwise almost everywhere over the product space  $\Omega \times [0,T]$ in $V$ (again for each fixed element of this product space, the infinite sum is in reality just a single term which we know to belong to $V$ from the progressive measurability of the $\sy^k\mathbbm{1}_{\cdot \leq \theta_j^k}$). Moreover for each fixed $T$
we can consider the sequence of processes $(\omega,t) \mapsto \left(\sum_{k=1}^N\sy^k_{t}(\omega)\mathbbm{1}_{t \leq \theta_j^k(\omega)}\mathbbm{1}_{\{k \leq \norm{\sy_0(\omega)}_H < k+1\}}\right)$ as mappings $\Omega \times [0,T] \rightarrow V$ where we equip $\Omega \times [0,T]$ with the sigma algebra $\mathcal{F}_T \times \mathcal{B}([0,T])$. The pointwise limit preserves the measurability which concludes the argument that $(\sy,\theta_j)$ is an $H-$valued local strong solution.\\

It remains to show that if $(\py,\Gamma)$ were any other pair with this property, then $\Theta \leq \Gamma$ $\mathbbm{P}-a.s.$ implies $\Theta = \Gamma$ $\mathbbm{P}-a.s.$. To this end suppose that $(\gamma_j)$ is the sequence of stopping times for $\Gamma$, and define now for each fixed $k,j$ the stopping time $$\gamma^k_j:=\gamma_j\mathbbm{1}_{\{k \leq \norm{\sy_0}_H < k+1\}}$$ along with the process
\begin{align*}
    \hat{\sy}(\omega)&:=\sy^k(\omega) \qquad \textnormal{for $\theta_j^k(\omega) \vee \gamma_j^k(\omega) = \theta_j^k(\omega)$}\\
    \hat{\sy}(\omega)&:=\py(\omega) \qquad  \textnormal{for $\theta_j^k(\omega) \vee \gamma_j^k(\omega) = \gamma_j^k(\omega)$}.
\end{align*}
Whilst we do not claim that $(\py,\gamma_j^k)$ is a local strong solution for the initial condition $\sy_0\mathbbm{1}_{\{k \leq \norm{\sy_0}_H < k+1\}}$ (indeed $\gamma_j^k$ is unlikely to be $\mathbbm{P}-a.s.$ positive), note that from the uniqueness Theorem \ref{uniqueness for V valued solutions} we have that $\sy_{\cdot \wedge \theta_j \wedge \gamma_j}$ and $\py_{\cdot \wedge \theta_j \wedge \gamma_j}$ are indistinguishable. This implies indistinguishability of $\sy^k_{\cdot \wedge \theta^k_j \wedge \gamma^k_j}\mathbbm{1}_{\{k \leq \norm{\sy_0}_H < k+1\}}$ and $\py_{\cdot \wedge \theta^k_j \wedge \gamma^k_j}\mathbbm{1}_{\{k \leq \norm{\sy_0}_H < k+1\}}$ so the definition of $\hat{\sy}$ is consistent. Moreover the exact arguments of Lemma \ref{lemma for maximality} continue to apply here to demonstrate that $(\hat{\sy},\theta_j^k \vee \gamma_j^k)$ is a local strong solution for $\sy_0\mathbbm{1}_{\{k \leq \norm{\sy_0}_H < k+1\}}$ (we of course rely on $(\py,\gamma_j)$ being a solution for $\sy_0$). Thus from the maximality of $\Theta^k$ and Corollary \ref{greatness corollary}, we have $\theta_j^k \vee \gamma_j^k \leq \Theta^k$ $\mathbbm{P}-a.s.$ and in particular $\gamma_j^k \leq \Theta^k$. Defining $$\Gamma^k := \Gamma\mathbbm{1}_{\{k \leq \norm{\sy_0}_H < k+1\}}$$ we have that $\gamma^k_j \rightarrow \Gamma^k$ $\mathbbm{P}-a.s.$ therefore $\Gamma^k \leq \Theta^k$ $\mathbbm{P}-a.s.$. Evidently though $$\Gamma = \sum_{k=1}^\infty\Gamma^k\mathbbm{1}_{\{k \leq \norm{\sy_0}_H < k+1\}}$$ demonstrating that $\Gamma \leq \Theta$ $\mathbbm{P}-a.s.$.


\end{proof}

\begin{theorem}
For any given $\mathcal{F}_0-$ measurable $\sy_0:\Omega \rightarrow H$, there exists a unique $H-$valued maximal strong solution $(\sy,\Theta)$ of the equation (\ref{thespde}). Moreover at $\mathbb{P}-a.e.$ $\omega$ for which $\Theta(\omega)<\infty$, we have that \begin{equation}\nonumber \sup_{r \in [0,\Theta(\omega))}\norm{\sy_r(\omega)}_H^2 + \int_0^{\Theta(\omega)}\norm{\sy_r(\omega)}_V^2dr = \infty.\end{equation}
Furthermore for any $M > 1$, $t>0$ the stopping time \begin{equation}
    \nonumber \tau^M_t(\omega):=t \wedge \inf\left\{s \geq 0: \norm{\sy(\omega)}_{HV,s}^2 \geq M + \norm{\sy_0(\omega)}_H^2 \right\}
\end{equation} is well defined and such that $(\sy_{\cdot \wedge \tau^M_t},\tau^M_t)$ is an $H-$valued local strong solution of the equation (\ref{thespde}).
\end{theorem}

\begin{remark}
This contains the statement of Theorem \ref{theorem1}. 
\end{remark}

\begin{proof}
 We have justified the existence in Theorem \ref{v valued existence result}, and the uniqueness is immediate from Theorem \ref{uniqueness of maximal solution} as the proof uses only general properties of stopping times associated to solutions (and Theorem \ref{uniqueness for V valued solutions} was shown for the unbounded $\sy_0$). As for the properties shown in Theorem \ref{blow up property theorem} for $\sy_0 \in L^{\infty}(\Omega;H)$, it is sufficient to show that $(\sy_{\cdot \wedge \tau^M_t},\tau^M_t)$ is a local strong solution as once again (\ref{blowupproperty}) would follow (as $\Theta > \tau^M_t$ $\mathbbm{P}-a.s.$ for every $M,t$). We take the same approach as in Theorem \ref{v valued existence result}, defining for each $k$ the stopping time $$\tau^{k,M}_t(\omega):=t \wedge \inf\left\{s \geq 0: \norm{\sy^k(\omega)}_{HV,s}^2 \geq M + \norm{\sy_0(\omega)\mathbbm{1}_{\{k \leq \norm{\sy_0(\omega)}_H <k+1\}}}_H^2 \right\}$$ which is such that $(\sy^k_{\cdot \wedge \tau^{k,M}_t},\tau^{k,M}_t)$ is a local strong solution for $\sy_0\mathbbm{1}_{\{k \leq \norm{\sy_0}_H <k+1\}}$ by Theorem \ref{blow up property theorem}. Noting that \begin{align*}
  \sy_{s \wedge \tau^M_t}(\omega) &=    \sum_{k=1}^\infty\sy_{s \wedge \tau^{k,M}_t}^k(\omega)\mathbbm{1}_{\{k \leq \norm{\sy_0(\omega)}_H <k+1\}}\\
\tau^M_t(\omega) &= \sum_{k=1}^\infty\tau^{k,M}_t(\omega)\mathbbm{1}_{\{k \leq \norm{\sy_0(\omega)}_H <k+1\}}
 \end{align*}
 then the arguments of Theorem \ref{v valued existence result} immediately apply to show that $(\sy_{\cdot \wedge \tau^M_t},\tau^M_t)$ is a local strong solution as required.
\end{proof}

\section{U-Valued Solutions} \label{h valued}

We now extend the framework of Section \ref{Section v valued} to cover the existence, uniqueness and maximality results for an SPDE (\ref{thespde}) satisfying both the assumptions of Subsection \ref{assumptionschapter} and a new set of assumptions relative to a fourth Hilbert Space to be introduced. The need to extend this framework is motivated by application, see Subsection \ref{subsection velocity}. There it is explicitly demonstrated that the SALT Navier-Stokes Equation in Velocity Form does not satisfy the assumptions of Subsection \ref{assumptionschapter} for the optimal spaces, but in the associated paper \cite{goodair2022inprep} we show that the assumptions of this Section (given in Subsection \ref{subsection for assumptions 2}) are satisfied in order to deduce the existence of a solution in the optimal spaces. Supplementing this is Subsection \ref{subsection vorticity} in which we show that Section \ref{Section v valued} is sufficient to deduce the existence of a solution to the SALT Navier-Stokes Equation in Vorticity Form in the optimal spaces, hence why we distinguish between the criteria of Sections \ref{Section v valued} and \ref{h valued}.\\

Our method of proof was just lightly touched upon in the introduction, so before giving the assumptions we expand on that description here. We wish to again apply the Convergence of Random Cauchy Sequences lemma (Lemma \ref{greenlemma}) to deduce the existence of a local strong solution, and continue to prove maximality as seen in Section \ref{Section v valued}. The Cauchy sequence here will not be a Galerkin Approximation, but rather a sequence of solutions of the full equation (\ref{thespde}) which we know to exist from Theorem \ref{theorem1}. The purpose of this section is to find solutions for $\sy_0 \in U$ as opposed to $H$, so the approximating sequence is given by the solutions to (\ref{thespde}) corresponding to the initial conditions $(\mathcal{P}_n\sy_0)$ where $\mathcal{P}_n\sy_0 \in H$ and $\mathcal{P}_n\sy_0 \longrightarrow \sy_0$ in $U$ as $n \rightarrow \infty$. The most delicate issue of the matter is for what type of solutions we should take in our approximation, as the $H-$valued ones only exist up until a blow up of the $HV$ norm as specified in (\ref{actualblowup}). We need to work with solutions existing up until a blow up in the $UH$ norm to take first hitting times in this  norm as done for (\ref{tauMtn}). The fix is to extend the $H-$valued solutions into an intermediary notion of solution (which we call a $U/H$ solution, Definition \ref{definitionofVHsolution}) at the cost of some regularity, but retaining the property that these processes exist in $V$. 

\subsection{Assumptions} \label{subsection for assumptions 2}

Suppose now that $X$ is a Hilbert Space with embedding $U \xhookrightarrow{} X$.  We ask that there is a continuous bilinear form $\inner{\cdot}{\cdot}_{X \times H}: X \times H \rightarrow \R$ such that for $\phi \in U$ and $\psi \in H$, \begin{equation} \label{bilinear form}
    \inner{\phi}{\psi}_{X \times H} =  \inner{\phi}{\psi}_{U}.
\end{equation}
Moreover it is now necessary that the system $(a_n)$ forms an orthogonal basis of $U$. We state the remaining assumptions now for arbitrary elements $\phi,\psi \in H$ and $t \in [0,\infty)$, and continue to use the $c,K,\tilde{K}, \kappa$ notation of Subsection \ref{assumptionschapter}. The operators $\mathcal{A}$ and $\mathcal{G}$ must now be extended to the larger spaces, and are such that for any $T>0$, $\mathcal{A}:[0,T] \times H \rightarrow X$ and  $\mathcal{G}:[0,T] \times H \rightarrow \mathscr{L}^2(\mathfrak{U};U)$ are measurable.

\begin{assumption} \label{gg3}
\begin{align}
    \label{wellposedinX}
    \norm{\mathcal{A}(t,\boldsymbol{\phi})}^2_X +  \sum_{i=1}^\infty \norm{\mathcal{G}_i(t,\boldsymbol{\phi})}^2_U &\leq c_tK(\boldsymbol{\phi})\left[1 + \norm{\boldsymbol{\phi}}_H^2\right],\\ \label{222*} \norm{\mathcal{A}(t,\boldsymbol{\phi}) - \mathcal{A}(t,\boldsymbol{\psi})}_X^2 &\leq  c_t\tilde{K}_2(\phi,\psi)\norm{\phi-\psi}_H^2
\end{align}

\end{assumption}

\begin{assumption} \label{uniqueness for H valued}
\begin{align}
    2\inner{\mathcal{A}(t,\boldsymbol{\phi}) - \mathcal{A}(t,\boldsymbol{\psi})}{\boldsymbol{\phi} - \boldsymbol{\psi}}_X + \sum_{i=1}^\infty\norm{\mathcal{G}_i(t,\boldsymbol{\phi}) - \mathcal{G}_i(t,\boldsymbol{\psi})}_X^2 &\leq \label{therealcauchy1*} c_{t}\tilde{K}_2(\boldsymbol{\phi},\boldsymbol{\psi}) \norm{\boldsymbol{\phi}-\boldsymbol{\psi}}_X^2,\\
    \sum_{i=1}^\infty \inner{\mathcal{G}_i(t,\boldsymbol{\phi}) - \mathcal{G}_i(t,\boldsymbol{\psi})}{\boldsymbol{\phi}-\boldsymbol{\psi}}^2_X & \leq \label{therealcauchy2*} c_{t} \tilde{K}_2(\boldsymbol{\phi},\boldsymbol{\psi}) \norm{\boldsymbol{\phi}-\boldsymbol{\psi}}_X^4
\end{align}
\end{assumption}

\begin{assumption} \label{assumption for probability in H}
With the stricter requirement that $\phi\in V$ then 
\begin{align}
   \label{probability first H} 2\inner{\mathcal{A}(t,\boldsymbol{\phi})}{\boldsymbol{\phi}}_U + \sum_{i=1}^\infty\norm{\mathcal{G}_i(t,\boldsymbol{\phi})}_U^2 &\leq c_tK(\boldsymbol{\phi}) -  \kappa\norm{\boldsymbol{\phi}}_H^2,\\\label{probability second H}
    \sum_{i=1}^\infty \inner{\mathcal{G}_i(t,\boldsymbol{\phi})}{\boldsymbol{\phi}}^2_U &\leq c_tK(\boldsymbol{\phi}).
\end{align}

\begin{remark}
This is a stronger assumption than Assumption \ref{assumption for prob in V}.
\end{remark}
\end{assumption}

As in Subsection \ref{assumptionschapter} we briefly address the purpose of these assumptions.

\begin{itemize}
    \item The significance of the property (\ref{bilinear form}) was discussed in the introduction, but to recapitulate this is necessary in applying the It\^{o} Formula to deduce results pertaining to the $U$ inner product for solutions of $H$ regularity only satisfying an identity in $X$. In particular it is needed for the analysis of Propositions \ref{theorem:cauchy reduce regularity} and \ref{prob theorem two} in order to once more apply the Convergence of Random Cauchy Sequences lemma (Lemma \ref{greenlemma}).
    \item The system $(a_n)$ is now required to form an orthogonal basis of $U$ for the property that given any $\phi \in U$, $\norm{(I-\mathcal{P}_n)\phi}_U \longrightarrow 0$ as $n \rightarrow \infty$. This is applied to show that the sequence of projected initial conditions converge to the original $\sy_0$ in $U$, necessary in Proposition \ref{theorem:cauchy reduce regularity} and Theorem \ref{existence of local strong H solution}. In Section \ref{Section v valued} we had that $\sy_0 \in H$ so this characteristic came from (\ref{mu2}). 
    \item Assumption \ref{gg3} provides the growth constraint (\ref{wellposedinX}) which ensures that the integrals in (\ref{identityindefinitionoflocalsolution**}) are well defined. We also impose a local Lipschitz type condition on $\mathcal{A}$ in (\ref{222*}) which we use to show the convergence of the approximating sequence of time intergrals to the appropriate limit (Theorem \ref{existence of local strong H solution}).
    \item Assumption \ref{uniqueness for H valued} is used to show the uniqueness of solutions (Theorem \ref{uniqueness for H valued solutions}).
    \item Assumption \ref{assumption for probability in H} is used to show the uniform rate of convergence of the approximating solutions to their initial conditions (Proposition \ref{prob theorem two}).

\end{itemize}


\subsection{Definitions and Main Results}

Similarly to Subsection \ref{subsection:notionsofsolution} we now state the relevant definitions and main results of this section. 

\begin{definition}[$U-$valued local strong solution] \label{definitionofHsolution}
Let $\sy_0: \Omega \rightarrow U$ be $\mathcal{F}_0-$measurable. A pair $(\sy,\tau)$ where $\tau$ is a $\mathbbm{P}-a.s.$ positive stopping time and $\sy$ is a process such that for $\mathbbm{P}-a.e.$ $\omega$, $\sy_{\cdot}(\omega) \in C\left([0,T];U\right)$ and $\sy_{\cdot}(\omega)\mathbbm{1}_{\cdot \leq \tau(\omega)} \in L^2\left([0,T];H\right)$ for all $T>0$ with $\sy_{\cdot}\mathbbm{1}_{\cdot \leq \tau}$ progressively measurable in $H$, is said to be a $U-$valued local strong solution of the equation (\ref{thespde}) if the identity
\begin{equation} \label{identityindefinitionoflocalsolution**}
    \sy_{t} = \sy_0 + \int_0^{t\wedge \tau} \mathcal{A}(s,\sy_s)ds + \int_0^{t \wedge \tau}\mathcal{G} (s,\sy_s) d\mathcal{W}_s
\end{equation}
holds $\mathbbm{P}-a.s.$ in $X$ for all $t \geq 0$.
\end{definition}

\begin{remark} \label{remark for H}
The Remarks \ref{remark1}, \ref{remark on prog meas equivalence} hold in the analagous way here. The same justification of \cite{goodair2022stochastic} Subsections 2.2 and 2.4 shows that $\int_0^{t\wedge \tau} \mathcal{A}(s,\sy_s)ds$ is well defined in $X$ and  $\int_0^{\cdot \wedge \tau}\mathcal{G} (s,\sy_s) d\mathcal{W}_s$ as a local martingale in $U$, referring to Assumption (\ref{wellposedinX}).
\end{remark}

\begin{definition}[$U-$valued maximal strong solution] \label{H valued maximal definition}
A pair $(\sy,\Theta)$ such that there exists a sequence of stopping times $(\theta_j)$ which are $\mathbb{P}-a.s.$ monotone increasing and convergent to $\Theta$, whereby $(\sy_{\cdot \wedge \theta_j},\theta_j)$ is a $U-$valued local strong solution of the equation (\ref{thespde}) for each $j$, is said to be a $U-$valued maximal strong solution of the equation (\ref{thespde}) if for any other pair $(\py,\Gamma)$ with this property then $\Theta \leq \Gamma$ $\mathbb{P}-a.s.$ implies $\Theta = \Gamma$ $\mathbb{P}-a.s.$.
\end{definition}

\begin{definition} \label{definition unique}
A $U-$valued maximal strong solution $(\sy,\Theta)$ of the equation (\ref{thespde}) is said to be unique if for any other such solution $(\py,\Gamma)$, then $\Theta = \Gamma$ $\mathbb{P}-a.s.$ and \begin{equation} \nonumber\mathbb{P}\left(\left\{\omega \in \Omega: \sy_{t}(\omega) =  \py_{t}(\omega) \quad \forall t \in [0,\Theta) \right\} \right) = 1. \end{equation}
\end{definition}

The following is the main result of the section, and holds true if the assumptions of Subsections \ref{assumptionschapter} and \ref{subsection for assumptions 2} are met.

\begin{theorem} \label{theorem2}
For any given $\mathcal{F}_0-$ measurable $\sy_0:\Omega \rightarrow U$, there exists a unique $U-$valued maximal strong solution $(\sy,\Theta)$ of the equation (\ref{thespde}). Moreover at $\mathbb{P}-a.e.$ $\omega$ for which $\Theta(\omega)<\infty$, we have that \begin{equation}\label{blowuppropertyVH} \sup_{r \in [0,\Theta(\omega))}\norm{\sy_r(\omega)}_U^2 + \int_0^{\Theta(\omega)}\norm{\sy_r(\omega)}_H^2dr = \infty.\end{equation}
\end{theorem}

The remainder of Section \ref{h valued} follows the path portrayed here.

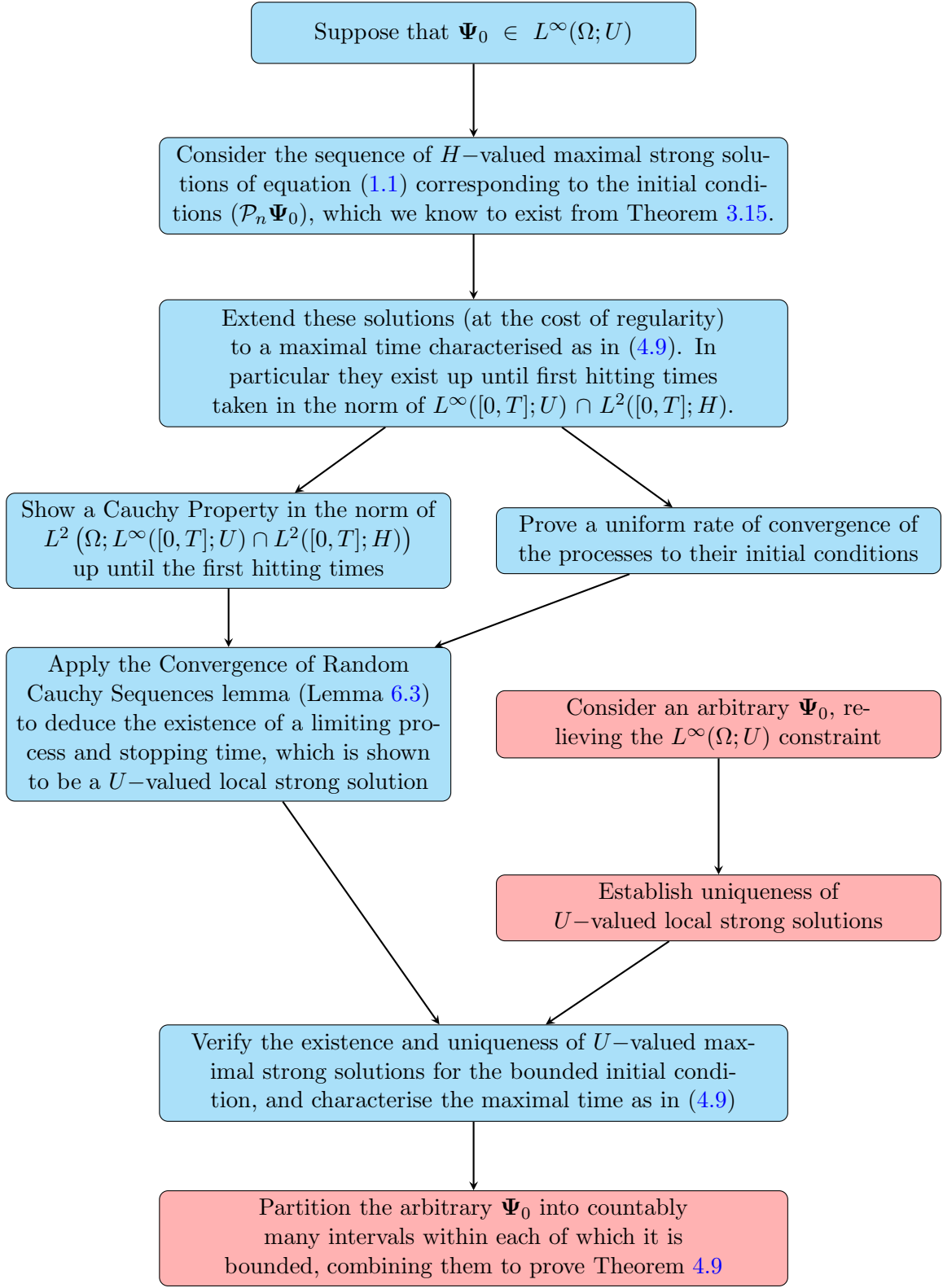
\begin{tikzpicture}[node distance=2cm]

\node (1) [startstop, xshift=10cm] {Suppose that $\sy_0 \in L^{\infty}(\Omega;U)$};

\node(2)[startstop2, below of=1, yshift=-0.5cm]{Consider the sequence of $H-$valued maximal strong solutions of equation (\ref{thespde}) corresponding to the initial conditions $(\mathcal{P}_n\sy_0)$, which we know to exist from Theorem \ref{theorem1}.};

\node(3)[startstop2, below of=2, yshift=-0.9cm]{Extend these solutions (at the cost of regularity) to a maximal time characterised as in (\ref{blowuppropertyVH}). In particular they exist up until first hitting times taken in the norm of $L^\infty([0,T];U) \cap L^2([0,T];H)$.};

\node(4)[startstop, below of=3, yshift=-0.9cm, xshift= -4cm]{Show a Cauchy Property in the norm of $L^2\left(\Omega;L^\infty([0,T];U) \cap L^2([0,T];H) \right)$ up until the first hitting times};

\node(5)[startstop, right of=4, xshift= 6cm]{Prove a uniform rate of convergence of the processes to their initial conditions};

\node(6)[startstop, below of=4, yshift=-1cm]{Apply the Convergence of Random Cauchy Sequences lemma (Lemma \ref{greenlemma}) to deduce the existence of a limiting process and stopping time, which is shown to be a $U-$valued local strong solution};

\node(7)[startstop', right of=6, xshift= 6cm]{Consider an arbitrary $\sy_0$, relieving the $L^{\infty}(\Omega;U)$ constraint};

\node(8)[startstop', below of=7, yshift=-1cm]{Establish uniqueness of $U-$valued local strong solutions};

\node(9)[startstop2, below of=8, yshift=-0.7cm, xshift=-4cm]{Verify the existence and uniqueness of $U-$valued maximal strong solutions for the bounded initial condition, and characterise the maximal time as in (\ref{blowuppropertyVH})};

\node(10)[startstop2', below of=9, yshift=-0.7cm]{Partition the arbitrary $\sy_0$ into countably many intervals within each of which it is bounded, combining them to prove Theorem \ref{theorem2}};

\draw[arrow] (1) -- (2);
\draw[arrow] (2) -- (3);
\draw[arrow] (3) -- (4);
\draw[arrow] (3) -- (5);
\draw[arrow] (4) -- (6);
\draw[arrow] (5) -- (6);
\draw[arrow] (7) -- (8);
\draw[arrow] (6) -- (9);
\draw[arrow] (8) -- (9);
\draw[arrow] (9) -- (10);

\end{tikzpicture}

\subsection{Uniqueness}

Before showing the existence of such solutions, we immediately show uniqueness and do so for any given initial condition $\sy_0 \in U$.

\begin{theorem} \label{uniqueness for H valued solutions}
Suppose that $(\sy^1,\tau_1)$ and $(\sy^2,\tau_2)$ are two $U-$valued local strong solutions of the equation (\ref{thespde}) for a given initial condition $\sy_0$. Then \begin{equation} \nonumber\mathbbm{P}\left(\left\{\omega \in \Omega: \sy^1_{t \wedge \tau_1(\omega) \wedge \tau_2(\omega) }(\omega) =  \sy^2_{t \wedge \tau_1(\omega) \wedge \tau_2(\omega)}(\omega) \quad \forall t \in [0,\infty) \right\} \right) = 1. \end{equation}
\end{theorem}

\begin{proof}
 We shall spare most of the details in this argument as it follows in the exact same way as Theorem \ref{uniqueness for V valued solutions}. We apply the It\^{o} Formula to the same difference process but this time in $X$, making the same definitions of $\tilde{\sy}^1$ and $\tilde{\sy}^2$ and employing Assumption \ref{uniqueness for H valued}. In this iteration we define $\gamma_R$ with $\tilde{\sy}^{1,R}$ and $\tilde{\sy}^{2,R}$ as before. We then reach the analogy of (\ref{someconstanthatc}) which is $$\mathbbm{E}\sup_{r \in [\theta_j,\theta_k]}\norm{\tilde{\py}^R_{r}}_X^2 \leq \hat{c}\mathbbm{E}\left(\norm{\tilde{\py}^R_{\theta_j}}_X^2 + \int_{\theta_j}^{\theta_k}\tilde{K}_2(\tilde{\sy}^{1,R}_s,\tilde{\sy}^{2,R}_s)\norm{\tilde{\py}^R_s}_X^2ds\right)$$ from which the proof is concluded in the same fashion.
\end{proof}

\subsection{Maximality of Solutions for an H-Valued Initial Condition}

We introduce now an intermediary notion of solution, to help us pass from the $H-$valued solutions shown to exist in Theorem \ref{v valued existence result} to the $U-$valued ones defined in Definition \ref{definitionofHsolution}.

\begin{definition}[$U/H$ local strong solution] \label{definitionofVHsolution}
Let $\sy_0: \Omega \rightarrow H$ be $\mathcal{F}_0-$measurable. A pair $(\sy,\tau)$ where $\tau$ is a $\mathbbm{P}-a.s.$ positive stopping time and $\sy$ is a process such that for $\mathbbm{P}-a.e.$ $\omega$, $\sy_{\cdot}(\omega) \in C\left([0,T];U\right)$ and $\sy_{\cdot}(\omega)\mathbbm{1}_{\cdot \leq \tau(\omega)} \in L^2\left([0,T];H\right)$ for all $T>0$ with $\sy_{\cdot}\mathbbm{1}_{\cdot \leq \tau}$ progressively measurable in $H$, and $\sy_t(\omega)\mathbbm{1}_{t \leq \tau(\omega)} \in V$ almost everywhere on $\Omega \times [0,\infty)$, is said to be a $U/H$ local strong solution of the equation (\ref{thespde}) if the identity
\begin{equation} \nonumber
    \sy_{t} = \sy_0 + \int_0^{t\wedge \tau} \mathcal{A}(s,\sy_s)ds + \int_0^{t \wedge \tau}\mathcal{G} (s,\sy_s) d\mathcal{W}_s
\end{equation}
holds $\mathbbm{P}-a.s.$ in $X$ for all $t \geq 0$.
\end{definition}

\begin{remark} \label{trivially equal}
Trivially any $H-$valued local strong solution is a $U/H$ one.
\end{remark}

The idea behind introducing this notion of solution is to extend the $H-$valued maximal strong solutions to the maximal time characterised by (\ref{blowuppropertyVH}). This comes with the corresponding loss of regularity in the solution, but by requiring that the process is in $V$ almost surely we can apply the Assumption \ref{therealcauchy assumptions} in the context of our energy methods once more. This is why we do not immediately pass to the $U-$valued solutions for the $H-$valued initial condition, and indeed requiring this initial condition to be in $H$ not just $U$ is what facilitates this requirement of belonging to $V$.

\begin{theorem} \label{blow up property theorem VH}
For any given $\mathcal{F}_0-$measurable $\sy_0: \Omega \rightarrow H$, there exists a unique $U/H$ maximal strong solution $(\sy,\Theta)$ of the equation (\ref{thespde}). Moreover at $\mathbbm{P}-a.e.$ $\omega$ for which $\Theta(\omega)<\infty$, we have that \begin{equation}\nonumber \sup_{r \in [0,\Theta(\omega))}\norm{\sy_r(\omega)}_U^2 + \int_0^{\Theta(\omega)}\norm{\sy_r(\omega)}_H^2dr = \infty.\end{equation} Furthermore for any $M > 1$, $t>0$ the stopping time \begin{equation}
    \label{tauMt no n VH} \tau^M_t(\omega):=t \wedge \inf\left\{s \geq 0: \sup_{r \in [0,s]}\norm{\sy_r(\omega)}_U^2 + \int_0^{s}\norm{\sy_r(\omega)}_H^2dr \geq M + \norm{\sy_0(\omega)}_U^2 \right\}
\end{equation} is well defined and such that $(\sy_{\cdot \wedge \tau^M_t},\tau^M_t)$ is a $U/H$ local strong solution of the equation (\ref{thespde}).
\end{theorem}

\begin{remark}
The existence and uniqueness definitions are precisely those of \ref{V valued maximal definition} and \ref{uniqueness of maximal solution} but simply for the $U/H$ solution as defined in Definition \ref{definitionofVHsolution} instead.
\end{remark}

\begin{proof}
 We do not wish to spend too much time on the concepts already covered which is why we immediately state this theorem without needlessly redefining what we mean by a unique $U/H$ maximal strong solution. We reach this result via the same path taken to get the equivalent result for the unique $H-$valued maximal strong solution, starting again from a bounded $\sy_0$ in $H$ and using that we have already a $U/H$ local strong solution for this initial condition from Remark \ref{trivially equal}.  We prove the analogy of Theorem \ref{blow up property theorem} here in the same way, just using the norm bounds from the stopping times in the larger Hilbert Spaces which is sufficient to construct $\py$ in $X$ (as opposed to $U$ in (\ref{definition of py})). Beyond this symmetry the only difference here is the additional requirement that our candidate local strong solution $(\py,\rho^M_t)$ is such that $\py_{\cdot}\mathbbm{1}_{\cdot \leq \rho^M_t}$ belongs to $V$ almost surely over the product space. Note that \begin{align*}\py_{\cdot}\mathbbm{1}_{\cdot \leq \theta_j \wedge \rho^M_t} &= \py_{\cdot \wedge \theta_j \wedge \rho^M_t}\mathbbm{1}_{\cdot \leq \theta_j \wedge \rho^M_t}\\ &= \py_{\cdot \wedge \theta_j}\mathbbm{1}_{\cdot \leq \theta_j \wedge \rho^M_t}\\ &= \sy_{\cdot \leq \theta_j \wedge \rho^M_t}\mathbbm{1}_{\cdot \leq \theta_j \wedge \rho^M_t}
 \end{align*}
$\mathbbm{P}-a.s.$ for every $j$, using (\ref{pysyequivalence}). Thus $\py_{\cdot}\mathbbm{1}_{\cdot \leq \theta_j \wedge \rho^M_t}$ belongs to $V$ almost surely as this is inherited from the local strong solution $(\sy, \theta_j \wedge \rho^M_t)$, so we must also have the same for $\py_{\cdot}\mathbbm{1}_{\cdot \leq  \rho^M_t}$ from the convergence $\theta_j \wedge \rho^M_t \rightarrow \rho^M_t$  $\mathbbm{P}-a.s.$. We can then pass to the unbounded initial condition exactly as in Subsection \ref{subsection maximality for unbounded}. 
\end{proof}

\subsection{Existence Method for a Bounded Initial Condition}

We fix an initial condition $\sy_0 \in L^\infty(\Omega;U)$ and consider the sequence of random variables $(\sy^n_0):=(\mathcal{P}_n\sy_0)$. From the continuity $\mathcal{P}_n: U \rightarrow H$ (as indeed the range $H$ is really just the finite dimensional $V_n$ equipped with the equivalent $H$ inner product) ensures that each $\sy^n_0$ is $H-$measurable. So from Theorem \ref{blow up property theorem VH}, for each $n$ there exists a unique $U/H$ maximal strong solution $(\sy^n,\Theta^n)$ such that for any $M > 1$ and $t >0$, $(\sy^n_{\cdot \wedge \tau^{M,t}_n},\tau^{M,t}_n)$ is a $U/H$ local strong solution where
\begin{equation} \label{the new tau m t n}
     \tau^{M,t}_n(\omega):=t \wedge \inf\left\{s \geq 0: \norm{\sy^n(\omega)}_{UH,s}^2 \geq M + \norm{\sy^n_0(\omega)}_U^2 \right\}.
\end{equation}
This set up is of course reminiscent of that in Subsection \ref{subsection:Existence Method for a Regular Truncated Initial Condition} and we again have the bounds (\ref{galerkinboundsatisfiedbystoppingtime}) and (\ref{galerkinboundsatisfiedbystoppingtime2}).

\begin{proposition} \label{theorem:cauchy reduce regularity}
For any $m,n \in \N$ with $m<n$, define the process $\sy^{m,n}$ by $$\sy^{m,n}_r(\omega) := \sy^n_r(\omega) - \sy^m_r(\omega).$$
Then $$\lim_{m \rightarrow \infty}\sup_{n \geq m}\left[\mathbbm{E}\norm{\sy^{m,n}}_{UH,\tau^{M,t}_m \wedge \tau^{M,t}_n}^2\right] = 0.$$
\end{proposition}

\begin{proof}
 As alluded to we look to go by the same method of proof as the corresponding Proposition \ref{theorem:cauchygalerkin}. The real difference between the proofs lies in the loss of regularity we have for the $U/H$ solutions, as the identity is no longer satisfied in $U$ but only in $X$. On this occasion we must appeal to Proposition \ref{rockner prop}; that is we have the equality
 \begin{align*}
    \norm{\sy^{m,n}_{r \wedge \tau^{M,t}_m \wedge \tau^{M,t}_n} }_U^2 = \norm{\sy^{m,n}_0}_U^2 &+ 2\int_0^{r\wedge \tau^{M,t}_m \wedge \tau^{M,t}_n} \inner{\mathcal{A}(s,\sy^n_s) - \mathcal{A}(s,\sy^m_s)}{\sy^{m,n}_s}_{X \times H}ds\\ &  +\int_0^{r\wedge \tau^{M,t}_m \wedge \tau^{M,t}_n}\sum_{i=1}^\infty\norm{\mathcal{G}_i(s,\sy^n_s) - \mathcal{G}_i(s,\sy^m_s)}_U^2ds\\ &+ 2\int_0^{r\wedge \tau^{M,t}_m \wedge \tau^{M,t}_n}\inner{\mathcal{G}_i(s,\sy^n_s) - \mathcal{G}_i(s,\sy^m_s)}{\sy^{m,n}_s}_UdW^i_s
\end{align*}
$\mathbbm{P}-a.s.$ for any $r \geq 0$. From the defining $V$ regularity of the $U/H$ solutions we do indeed have that $\mathcal{A}(s,\sy^n_s(\omega)) - \mathcal{A}(s,\sy^m_s(\omega))$ belongs to $U$ for almost every $(s,\omega)$. Thus \begin{align*}\int_0^{r\wedge \tau^{M,t}_m \wedge \tau^{M,t}_n} \inner{\mathcal{A}(s,\sy^n_s) &- \mathcal{A}(s,\sy^m_s)}{\sy^{m,n}_s}_{X \times H}ds\\ &= \int_0^{r\wedge \tau^{M,t}_m \wedge \tau^{M,t}_n} \inner{\mathcal{A}(s,\sy^n_s) - \mathcal{A}(s,\sy^m_s)}{\sy^{m,n}_s}_{U}ds\end{align*} $\mathbbm{P}-a.s.$ from (\ref{bilinear form}), putting us in a position to apply Assumption \ref{therealcauchy assumptions}. Again with the notation (\ref{more tilde notation}) and for arbitrary stopping times $0 \leq \theta_j \leq \theta_k \leq t$ $\mathbbm{P}-a.s.$ and any $\theta_j \leq r \leq t$ $\mathbbm{P}-a.s.$, we deduce that
\begin{align*}
    \norm{\tilde{\sy}^{m,n}_{r}}_U^2 &+ \kappa\int_{\theta_j}^r\norm{\tilde{\sy}^n_{s} - \tilde{\sy}^m_{s}}_H^2ds\leq  \norm{\sy^n_{\theta_j}-\sy^m_{\theta_j}}_U^2\\ &+ \int_{\theta_j}^rc\tilde{K}_2(\tilde{\sy}^m_{s},\tilde{\sy}^n_{s})\norm{\tilde{\sy}^n_{s} - \tilde{\sy}^m_{s}}_U^2ds + 2\int_0^{r}\inner{\mathcal{G}_i(s,\tilde{\sy}^n_s) - \mathcal{G}_i(s,\tilde{\sy}^m_s)}{\tilde{\sy}^{m,n}_{s}}_UdW^i_s
\end{align*}
$\mathbbm{P}-a.s.$ from (\ref{therealcauchy1}). Using once more the bound (\ref{galerkinboundsatisfiedbystoppingtime}) and (\ref{therealcauchy2}), we scale to remove the $\kappa$, take the supremum over $r \in [\theta_j,\theta_k]$ followed by the expectation and apply the Burkholder-Davis-Gundy Inequality to deduce 
\begin{align*}
   \mathbbm{E}\norm{\tilde{\sy}^{m,n}}_{UH,\theta_j,\theta_k}^2 \leq c\mathbbm{E}\norm{\tilde{\sy}^{m,n}_{\theta_j}}_U^2   &+ c\mathbbm{E}\int_{\theta_j}^{\theta_k} \tilde{K}_2(\tilde{\sy}^m_s,\tilde{\sy}^n_s) \norm{\tilde{\sy}^{m,n}_{s}}_U^2ds\\
        &+ c\mathbbm{E}\bigg(\int_{\theta_j}^{\theta_k} \tilde{K}_2(\tilde{\sy}^m_s,\tilde{\sy}^n_s) \norm{\tilde{\sy}_s^n-\tilde{\sy}_s^m}_U^4\bigg)^\frac{1}{2}.
\end{align*}
We can conclude the proof via an identical procedure to Proposition \ref{theorem:cauchygalerkin}, deducing this time that $\norm{(I-\mathcal{P}_m)\sy_0}_U^2$ is a ($\mathbbm{P}-a.s.)$ monotone decresasing sequence in $m$ convergent to zero from the fact that the $(\mathcal{P}_m)$ are orthogonal projections onto an orthogonal basis of $U$.
\end{proof}

\begin{proposition} \label{prob theorem two}
We have that $$\lim_{S \rightarrow 0}\sup_{n \in \mathbb{N}}\mathbb{P}\left(\left\{ \norm{\sy^{n}}^2_{UH,\tau^{M,t}_n \wedge S} \geq M-1+\norm{\sy^n_0}_{U}^2 \right\}\right) = 0.$$
\end{proposition}

\begin{proof}
Just as in the corresponding Proposition \ref{theorem: probability one} we argue that it is sufficient to show the property (\ref{sufficient thing in probability condition}). With the same approach but this time using Assumption \ref{assumption for probability in H}, note that
 \begin{align*}
    \mathbb{E}\Bigg[\sup_{r \in [0,\tau^{M,t}_n \wedge S]}\norm{\sy^{n}_{r}}^2_U + \kappa\int_0^{\tau^{M,t}_n \wedge S}\norm{\sy^{n}_{r}}^2_Hdr - \norm{\sy^n_0}_{U}^2 \Bigg] &\leq c\mathbb{E}\int_0^SK(\tilde{\sy}^n_r)dr + c\mathbbm{E}\left(\int_0^S K(\tilde{\sy}^n_r)dr\right)^{\frac{1}{2}}\\
    &\leq c\int_0^S1dr + c\left(\int_0^S1dr\right)^{\frac{1}{2}}\\
    &\leq cS + cS^{\frac{1}{2}}.
\end{align*}
In particular we have both that $$ \mathbb{E}\Bigg[\sup_{r \in [0,\tau^{M,t}_n \wedge S]}\norm{\sy^{n}_{r}}^2_U  - \norm{\sy^n_0}_{U}^2 \Bigg] \leq cS + cS^{\frac{1}{2}} $$ and $$\mathbb{E}\Bigg[\int_0^{\tau^{M,t}_n \wedge S}\norm{\sy^{n}_{r}}^2_Hdr \Bigg] \leq cS + cS^{\frac{1}{2}}$$
where we allow our generic $c$ to depend on $\kappa$, so $$\mathbb{E}\Bigg[\norm{\sy^{n}}^2_{UH,\tau^{M,t}_n \wedge S} - \norm{\sy^n_0}_{U}^2 \Bigg] \leq cS + cS^{\frac{1}{2}}$$ and the result follows.
\end{proof}

\begin{theorem} \label{existence of limiting pair theorem for H}
There exists a stopping time $\tau^{M,t}_{\infty}$, a subsequence $(\sy^{n_l})$ and a process $\sy_\cdot= \sy_{\cdot \wedge \tau^{M,t}_
\infty}$ whereby $\sy_{\cdot}\mathbbm{1}_{\cdot \leq \tau^{M,t}_{\infty}}$ is progressively measurable in $H$ and such that:
\begin{itemize}
    \item $\mathbb{P}\left(\left\{ 0 < \tau^{M,t}_{\infty} \leq \tau^{M,t}_{n_l}\right)\right\} = 1$;
    \item For $\mathbb{P}-a.e.$ $\omega$, $\sy(\omega) \in C\left([0,T];U\right)$ and $\sy_{\cdot}(\omega)\mathbbm{1}_{\cdot \leq \tau^{M,t}_{\infty}(\omega)} \in  L^2\left([0,T];H\right)$ for all $T>0$;
    \item For $\mathbb{P}-a.e.$ $\omega$, $\sy^{n_l}(\omega) \rightarrow \sy(\omega)$ in $ L^\infty\left([0,\tau^{M,t}_{\infty}(\omega)];U\right) \cap L^2\left([0,\tau^{M,t}_{\infty}(\omega)];H\right)$, i.e. 
\begin{equation} \nonumber \norm{\sy^{n_l}(\omega) - \sy(\omega)}_{UH,\tau^{M,t}_{\infty}(\omega)}^2 \longrightarrow 0;\end{equation}
\item $\sy^{n_l}\rightarrow \sy$ holds in the sense that \begin{equation}\nonumber\mathbb{E}\norm{\sy^{n_l} - \sy}_{UH,\tau^{M,t}_{\infty}}^2 \longrightarrow 0.\end{equation}
\end{itemize}
\end{theorem}

\begin{proof}
The proof is identical to that of Theorem \ref{existence of limiting pair theorem}, the only difference being that instead of resorting to Proposition \ref{theorem:uniformbounds} to show the second bullet point we have the required uniform boundedness immediately from (\ref{galerkinboundsatisfiedbystoppingtime}), and the continuity comes for free from the convergence of continuous processes in $L^\infty\left([0,\tau^{M,t}_{\infty}(\omega)];U\right)$.
\end{proof}

\begin{theorem} \label{existence of local strong H solution}
The pair $(\sy,\tau^{M,t}_{\infty})$ specified in Theorem \ref{existence of limiting pair theorem for H} is a $U-$valued local strong solution of the equation (\ref{thespde}) as defined in \ref{definitionofHsolution}.
\end{theorem}

\begin{proof}
We of course take a similar approach to Theorem \ref{existence of local strong V solution}, though with (\ref{222*}) we now don't need to use a weak convergence result as done in the previous theorem. To this end we consider the limit $\lim_{n_l \rightarrow \infty}\sy^{n_l}_{s \wedge \tau^{M,t}_{\infty}}$ taken as a $\mathbbm{P}-a.s.$ limit in $X$, with the idea to show that \begin{align}
    \label{convergence 1*}\lim_{n_l \rightarrow \infty}\sy^{n_l}_0 &= \sy_0\\\label{convergence 2*}
   \lim_{n_l \rightarrow \infty}\int_0^{s \wedge \tau^{M,t}_{\infty}}\mathcal{A}(r,\sy^{n_l}_r)dr &= \int_0^{s \wedge \tau^{M,t}_{\infty}}\mathcal{A}(r,\sy_r)dr\\
   \lim_{n_l \rightarrow \infty}\int_0^{s \wedge \tau^{M,t}_{\infty}}\mathcal{G}(r,\sy^{n_l}_r)d\mathcal{W}_r &= \int_0^{s \wedge \tau^{M,t}_{\infty}}\mathcal{G}(r,\sy_r)d\mathcal{W}_r \label{convergence 3*}.
\end{align}
If we prove the above (at least just for a further subsequence) then (\ref{identityindefinitionoflocalsolution**}) would be justified and we would be done. Firstly (\ref{convergence 1*}) follows from the fact that the $\mathcal{P}_{n_l}$ are orthogonal projections onto an orthogonal basis in $U$. In the vein of showing (\ref{convergence 2*}), we consider the term \begin{equation}
    \label{term for consieration*} \mathbbm{E}\left\Vert\int_0^{s \wedge \tau^{M,t}_{\infty}}\mathcal{A}(r,\sy^{n_l}_r)dr -  \int_0^{s \wedge \tau^{M,t}_{\infty}}\mathcal{A}(r,\sy_r)dr \right\Vert_X
\end{equation}
and have that 
\begin{align*}
    (\ref{term for consieration*}) & \leq  \mathbbm{E}\int_0^{s \wedge \tau^{M,t}_{\infty}}\left\Vert \mathcal{A}(r,\sy^{n_l}_r) - \mathcal{A}(r,\sy_r) \right\Vert_Xdr\\
    &\leq \mathbbm{E}\int_0^{s \wedge \tau^{M,t}_{\infty}}c_r\left[K(\sy^{n_l}_r,\sy_r) + \norm{\sy^{n_l}_r}_H + \norm{\sy_r}_H \right]\norm{\sy^{n_l}_r-\sy_r}_Hdr\\
    &\leq c\left(\mathbbm{E}\int_0^{s \wedge \tau^{M,t}_{\infty}}K(\sy^{n_l}_r,\sy_r) + \norm{\sy^{n_l}_r}_H^2 + \norm{\sy_r}_H^2dr\right)^{\frac{1}{2}}\left(\mathbbm{E}\int_0^{s \wedge \tau^{M,t}_{\infty}}\norm{\sy^{n_l}_r-\sy_r}_H^2dr \right)^{\frac{1}{2}}
\end{align*}
having employed the assumption (\ref{222*}), from which we use the boundedness (\ref{galerkinboundsatisfiedbystoppingtime}) to conclude that (\ref{convergence 2*}) holds along a subsequence $(\sy^{m_j})$ as in Theorem \ref{existence of local strong V solution}. We can in fact show (\ref{convergence 3*}) by taking the convergence first in $U$ (which implies that in $X$), which is contained in the argument to show (\ref{convergence 3}) and the proof is complete.

\end{proof}

\subsection{Maximality for a U-Valued initial condition}

We conclude this section by giving the now very brief proof of Theorem (\ref{theorem2}).

\begin{proof}[Proof of \ref{theorem2}:]
With the the uniqueness and existence results of Theorems \ref{uniqueness for H valued solutions}, \ref{existence of local strong H solution} in place, the procedure is identical to that carried out in Subsections \ref{subsection: maximality for bounded in V} and \ref{subsection maximality for unbounded} for the spaces $H,U,X$ corresponding to $V,H,U$.
\end{proof}

\section{Applications} \label{section applications}

In this section we give two applications of these results, both for the SALT Navier-Stokes Equation with one in velocity form and one in vorticity form. Whilst similar in nature these different forms make clear why we distinguished between the $H-$valued and $U-$valued solutions, and the insufficiency to consider just $H-$valued solutions in general. We only sketch the framework of application here, referring to \cite{goodair2022inprep} for the complete details. 

\subsection{SALT Navier-Stokes in Velocity Form} \label{subsection velocity}

Our object of study is the equation 
\begin{equation} \label{number2equation}
    u_t - u_0 + \int_0^t\mathcal{L}_{u_s}u_s\,ds - \int_0^t \Delta u_s\, ds + \int_0^t Bu_s \circ d\mathcal{W}_s + d \nabla \rho_t = 0
\end{equation}
 supplemented with the divergence-free (incompressibility) and zero-average conditions on the three dimensional torus $\mathbb{T}^3$. The equation is presented here in velocity form where $u$ represents the fluid velocity, $\rho$ the pressure, $\mathcal{L}$ is the mapping corresponding to the nonlinear term, $\mathcal{W}$ is a cylindrical Brownian Motion as throughout the paper and $B$ is the relevant transport operator defined with respect to a collection of functions $(\xi_i)$ which physically represent spatial correlations. These $(\xi_i)$ can be determined at coarse-grain resolutions from finely resolved numerical simulations, and mathematically are derived as eigenvectors of a velocity-velocity correlation matrix (see \cite{cotter2018modelling}, \cite{cotter2019numerically}, \cite{crisan2019solution}). The corresponding stochastic Euler equation was derived in \cite{street2021semi} and the viscous term plays no additional role in the stochastic derivation (without loss of generality we set the viscosity coefficient to be $1$).\\

We detail now the operators involved alongside the function spaces which define the equation. The operator $\mathcal{L}$ is defined for sufficiently regular functions $\phi,\psi:\mathbb{T}^3 \rightarrow \R^3$ by $$\mathcal{L}_{\phi}\psi:= \sum_{j=1}^3\phi^j\partial_j\psi$$ where $\phi^j:\mathbb{T}^3 \rightarrow \R$ is the $j^{\textnormal{th}}$ coordinate mapping of $\phi$ and $\partial_j\psi$ is defined by its $k^{\textnormal{th}}$ coordinate mapping $(\partial_j\psi)^k = \partial_j\psi^k$. The operator $B$ is defined as a linear operator on $\mathfrak{U}$ by its action on the basis vectors $B(e_i,\cdot) := B_i(\cdot)$ by $$B_i = \mathcal{L}_{\xi_i} + \mathcal{T}_{\xi_i}$$ for $\mathcal{L}$ as above and $$\mathcal{T}_{\phi}\psi:= \sum_{j=1}^3\psi^j\nabla\phi^j.$$  A complete discussion of how $B$ is then defined on $\mathfrak{U}$ is given in \cite{goodair2022stochastic} Subsection 2.2. We embed the divergence-free and zero-average conditions into the relevant function spaces and simply define our solutions as belonging to these spaces. To be explicit, by a divergence-free function we mean a $\phi \in W^{1,2}(\mathbb{T}^3;\R^3)$ such that $$\sum_{j=1}^3\partial_j\phi^j = 0$$ and by zero-average we ask for a $\psi \in L^2(\mathbb{T}^3;\R^3)$ with the property $$\int_{\mathbb{T}^3}\psi \ d\lambda = 0$$ for $\lambda$ the Lebesgue measure on $\mathbb{T}^3$. We introduce the space $L^2_{\sigma}(\mathbb{T}^3;\R^3)$ as the subspace of $L^2(\mathbb{T}^3;\R^3)$ consisting of zero-average functions which are 'weakly divergence-free'; see \cite{robinson2016three} Definition 2.1 for the precise construction. For general $m \in \N$ we then define  $W^{m,2}_{\sigma}(\mathbb{T}^3;\R^3)$ as $W^{m,2}(\mathbb{T}^3;\R^3) \cap L^2_{\sigma}(\mathbb{T}^3;\R^3)$.\\

As is standard in the treatment of the incompressible Navier-Stokes Equation we consider a projected version to eliminate the pressure term and facilitate us working in the above spaces. To this end we introduce the standard Leray Projector $\mathcal{P}$ defined as the orthogonal projection in $L^2(\mathbb{T}^3;\R^3)$ onto $L^2_{\sigma}(\mathbb{T}^3;\R^3)$, and assume that the $(\xi_i)$ are such that $\xi_i \in W^{1,2}_{\sigma}(\mathbb{T}^3;\R^3) \cap W^{3,\infty}(\mathbb{T}^3;\R^3)$ and satisfy the bound \begin{equation} \label{xi bound}
    \sum_{i=1}^\infty\norm{\xi_i}_{W^{3,\infty}}^2 < \infty.
\end{equation}
Our new equation is then
\begin{align} 
     u_t - u_0 + \int_0^t\mathcal{P}\mathcal{L}_{u_s}u_s\,ds + \int_0^t A u_sds -  \frac{1}{2}\sum_{i=1}^\infty \int_0^t\mathcal{P}B_i^2u_sds + \sum_{i=1}^\infty\int_0^t \mathcal{P}B_iu_s  dW_s^i = 0 \label{converted equation}
\end{align}
where $A := -\mathcal{P}\Delta$ is known as the Stokes Operator, indicating once more that all details are given in \cite{goodair2022inprep}. We shall use the Stokes operator to define inner products with which we equip our function spaces. Recall from \cite{robinson2016three} Theorem 2.24 for example that there exists a collection of functions $(a_k)$, $a_k \in W^{1,2}_{\sigma}(\mathbb{T}^3;\R^3) \cap C^{\infty}(\mathbb{T}^3;\R^3)$ such that the $(a_k)$ are eigenfunctions of $A$, are an orthonormal basis in $L^2_{\sigma}(\mathbb{T}^3;\R^3)$ and an orthogonal basis in $W^{1,2}_{\sigma}(\mathbb{T}^3;\R^3)$ considered as Hilbert Spaces with standard $L^2(\mathbb{T}^3;\R^3)$, $W^{1,2}(\mathbb{T}^3;\R^3)$ inner products. The corresponding eigenvalues $(\lambda_k)$ are strictly positive and approach infinity as $k \rightarrow \infty$. Thus any $\phi \in W^{1,2}_{\sigma}(\mathbb{T}^3;\R^3)$ admits the representation $$\phi = \sum_{k=1}^\infty \phi_ka_k$$ so for $m \in \mathbb{N}$ we can define $A^{m/2}$ by $$A^{m/2}: \phi \mapsto  \sum_{k=1}^\infty \lambda_k^{m/2}\phi_ka_k$$ which is a well defined element of $L^2_{\sigma}(\mathbb{T}^3;\R^3)$ on any $\phi$
such that \begin{equation} \label{powerproperty} \sum_{k=1}^\infty \lambda_k^m\phi_k^2 < \infty.\end{equation} For $\phi,\psi$ with the property (\ref{powerproperty}) then the bilinear form $$\inner{\phi}{\psi}_{m}:= \inner{A^{m/2}\phi}{A^{m/2}\psi}$$ is well defined. and equivalent to the standard $W^{m,2}(\mathbb{T}^3;\R^3)$ inner product. We equip $L^2_{\sigma}(\T;\R^3)$ with the usual inner product and $W^{m,2}(\T;\R^3)$ with the $\inner{\cdot}{\cdot}_m$ inner product.

\begin{definition} \label{definitionofirregularsolutionNS}
A pair $(u,\tau)$ where $\tau$ is a $\mathbb{P}-a.s.$ positive stopping time and $u$ is a process such that for $\mathbb{P}-a.e.$ $\omega$, $u_{\cdot}(\omega) \in C\left([0,T];W^{1,2}_{\sigma}(\T;\R^N)\right)$ and $u_{\cdot}(\omega)\mathbbm{1}_{\cdot \leq \tau(\omega)} \in L^2\left([0,T];W^{2,2}_{\sigma}(\T;\R^N)\right)$ for all $T>0$ with $u_{\cdot}\mathbbm{1}_{\cdot \leq \tau}$ progressively measurable in $W^{2,2}_{\sigma}(\T;\R^N)$, is said to be a local strong solution of the equation (\ref{converted equation}) if the identity
\begin{equation} \label{identityindefinitionoflocalsolutionHNS}
     u_t = u_0 - \int_0^{t\wedge \tau}\mathcal{P}\mathcal{L}_{u_s}u_s\ ds - \nu\int_0^{t\wedge \tau} A u_s\, ds + \int_0^{t\wedge \tau}\sum_{i=1}^\infty \mathcal{P}B_i^2u_s ds - \int_0^{t\wedge \tau} \mathcal{P}Bu_s d\mathcal{W}_s 
\end{equation}
holds $\mathbb{P}-a.s.$ in $L^2_{\sigma}(\T;\R^N)$ for all $t \geq 0$.
\end{definition}

\begin{definition} \label{H valued maximal definition NS}
A pair $(u,\Theta)$ such that there exists a sequence of stopping times $(\theta_j)$ which are $\mathbb{P}-a.s.$ monotone increasing and convergent to $\Theta$, whereby $(u_{\cdot \wedge \theta_j},\theta_j)$ is a local strong solution of the equation (\ref{converted equation}) for each $j$, is said to be a maximal strong solution of the equation (\ref{converted equation}) if for any other pair $(v,\Gamma)$ with this property then $\Theta \leq \Gamma$ $\mathbb{P}-a.s.$ implies $\Theta = \Gamma$ $\mathbb{P}-a.s.$.
\end{definition}

\begin{definition} \label{definition unique NS}
A maximal strong solution $(u,\Theta)$ of the equation (\ref{converted equation}) is said to be unique if for any other such solution $(v,\Gamma)$, then $\Theta = \Gamma$ $\mathbb{P}-a.s.$ and for all $t \in [0,\Theta)$, \begin{equation} \nonumber\mathbb{P}\left(\left\{\omega \in \Omega: u_{t}(\omega) =  v_{t}(\omega)  \right\} \right) = 1. \end{equation}
\end{definition}

For the spaces \begin{align*}
    V&:= W^{3,2}_{\sigma}(\T;\R^N), \qquad H:= W^{2,2}_{\sigma}(\T;\R^N),\\
    U&:= W^{1,2}_{\sigma}(\T;\R^N), \qquad X:= L^2_{\sigma}(\T;\R^N).
\end{align*}
and operators
\begin{align*}
    \mathcal{A}&:= -\left(\mathcal{P}\mathcal{L} +  A \right) + \frac{1}{2}\sum_{i=1}^\infty \mathcal{P}B_i^2\\
    \mathcal{G}&:= -\mathcal{P}B
\end{align*}
it is shown in \cite{goodair2022inprep} that we can apply Theorem \ref{theorem2} to conclude Theorem \ref{theorem2NS}.

\begin{theorem} \label{theorem2NS}
For any given $\mathcal{F}_0-$ measurable $u_0:\Omega \rightarrow W^{1,2}_{\sigma}(\T;\R^N)$, there exists a unique maximal strong solution $(u,\Theta)$ of the equation (\ref{converted equation}). Moreover at $\mathbb{P}-a.e.$ $\omega$ for which $\Theta(\omega)<\infty$, we have that \begin{equation}\nonumber \sup_{r \in [0,\Theta(\omega))}\norm{u_r(\omega)}_{1}^2 + \int_0^{\Theta(\omega)}\norm{u_r(\omega)}_2^2dr = \infty.\end{equation}
\end{theorem}

It is necessary to address explicitly why Theorem \ref{theorem2NS} could not be achieved by a more simple application of Theorem \ref{theorem1} with the spaces $V:= W^{2,2}_{\sigma}(\T;\R^N), H:= W^{1,2}_{\sigma}(\T;\R^N)$ and $U:= L^2_{\sigma}(\T;\R^N)$. The issue arises from the necessary control of the nonlinear term in showing (\ref{uniformboundsassumpt1}): for $H=W^{2,2}_{\sigma}(\mathbb{T}^3;\R^3)$ we have the algebra property of the Sobolev Space which affords us a bound $$\norm{\mathcal{P}\mathcal{L}_{\phi^n}\phi^n}_2 \leq c\norm{\phi^n}_2\norm{\phi^n}_3$$ using the equivalence of the $\norm{\cdot}_2$ and the standard $W^{2,2}(\T;\R^3)$ one. In the $W^{1,2}(\T;\R^3)$ norm we do not have the same luxury and so this nonlinear term cannot be bounded just in terms of the $W^{1,2}$ and $W^{2,2}$ norms as would be required.

\subsection{SALT Navier-Stokes in Vorticity Form} \label{subsection vorticity}

If we consider (\ref{number2equation}) in the alternative vorticity form, then we can show the optimal existence result with an application of Theorem \ref{theorem1} hence explicitly justifying the need to keep Sections \ref{Section v valued} and \ref{h valued} distinct. The equation is
\begin{equation} \label{number3equation}
    w_t - w_0 + \int_0^t\mathscr{L}(u_s,w_s)\,ds - \int_0^t \Delta w_s\, ds -\frac{1}{2}\sum_{i=1}^\infty\int_0^t\mathscr{L}^2_iw_s\,ds + \int_0^t \mathscr{L}_iw_s dW^i_s = 0
\end{equation}
which is again supplemented with the divergence free condition on $u,w$ and the boundary conditions $$u \cdot \mathbf{n} = 0, \qquad w=0$$ on a smooth bounded domain $\mathscr{O}\subset \R^3$ with $\mathbf{n}$ the outward unit normal. These are the so called Navier Boundary Conditions which are well summarised in \cite{kelliher2006navier}. Technical issues surrounding the noise term prevent us from working on a bounded domain in the velocity form, though these issues aren't present for the vorticity form which is a key motivator for our analysis in this setting. Here $w$ represents the fluid vorticity, $u$ continues to denote the velocity and the new operators are defined for sufficiently regular functions $\phi,\psi:\mathscr{O}\rightarrow\R^3$ by
\begin{align*}
    \mathscr{L}(\phi,\psi) &:= \mathcal{L}_{\phi}\psi -\mathcal{L}_{\psi}\phi\\
    \mathscr{L}_i\phi&:= \mathscr{L}(\xi_i,\phi).
\end{align*}
We can prescribe $u$ for a given $w$ via a Biot-Savart Operator \cite{enciso2018biot} to close (\ref{number3equation}) as an equation in $w$. We shall work with the spaces $L^2_{\sigma}(\mathscr{O};\R^3)$, $W^{m,2}_{\sigma}(\mathscr{O};\R^3)$ which are defined similarly to those on the torus but now to incorporate the zero-normal and zero-trace conditions, with the technicalities again deferred to \cite{goodair2022inprep}.

\begin{definition} \label{definitionofirregularsolutionNSw}
A pair $(w,\tau)$ where $\tau$ is a $\mathbb{P}-a.s.$ positive stopping time and $w$ is a process such that for $\mathbb{P}-a.e.$ $\omega$, $w_{\cdot}(\omega) \in C\left([0,T];W^{1,2}_{\sigma}(\mathscr{O};\R^N)\right)$ and $w_{\cdot}(\omega)\mathbbm{1}_{\cdot \leq \tau(\omega)} \in L^2\left([0,T];W^{2,2}_{\sigma}(\mathscr{O};\R^N)\right)$ for all $T>0$ with $w_{\cdot}\mathbbm{1}_{\cdot \leq \tau}$ progressively measurable in $W^{2,2}_{\sigma}(\mathscr{O};\R^N)$, is said to be a local strong solution of the equation (\ref{converted equation}) if the identity
\begin{equation} \nonumber
     w_t - w_0 + \int_0^{t\wedge \tau}\mathscr{L}(u_s,w_s)\,ds - \int_0^{t \wedge \tau} \Delta w_s\, ds -\frac{1}{2}\sum_{i=1}^\infty\int_0^{t \wedge \tau}\mathscr{L}^2_iw_s\,ds + \int_0^{t\wedge \tau} \mathscr{L}_iw_s dW^i_s = 0
\end{equation}
holds $\mathbb{P}-a.s.$ in $L^2_{\sigma}(\mathscr{O};\R^N)$ for all $t \geq 0$.
\end{definition}

\begin{definition} \label{H valued maximal definition NSw}
A pair $(w,\Theta)$ such that there exists a sequence of stopping times $(\theta_j)$ which are $\mathbb{P}-a.s.$ monotone increasing and convergent to $\Theta$, whereby $(w_{\cdot \wedge \theta_j},\theta_j)$ is a local strong solution of the equation (\ref{number3equation}) for each $j$, is said to be a maximal strong solution of the equation (\ref{number3equation}) if for any other pair $(\eta,\Gamma)$ with this property then $\Theta \leq \Gamma$ $\mathbb{P}-a.s.$ implies $\Theta = \Gamma$ $\mathbb{P}-a.s.$.
\end{definition}

\begin{definition} \label{definition unique NSw}
A maximal strong solution $(w,\Theta)$ of the equation (\ref{number3equation}) is said to be unique if for any other such solution $(\eta,\Gamma)$, then $\Theta = \Gamma$ $\mathbb{P}-a.s.$ and for all $t \in [0,\Theta)$, \begin{equation} \nonumber\mathbb{P}\left(\left\{\omega \in \Omega: w_{t}(\omega) =  \eta_{t}(\omega)  \right\} \right) = 1. \end{equation}
\end{definition}

For the spaces $V:= W^{2,2}_{\sigma}(\mathscr{O};\R^N), H:= W^{1,2}_{\sigma}(\mathscr{O};\R^N)$ and $U:= L^2_{\sigma}(\mathscr{O};\R^N)$ it is shown in \cite{goodair2022inprep} that we can apply Theorem \ref{theorem1} to conclude Theorem \ref{theorem2NSw}.

\begin{theorem} \label{theorem2NSw}
For any given $\mathcal{F}_0-$ measurable $w_0:\Omega \rightarrow W^{1,2}_{\sigma}(\mathscr{O};\R^N)$, there exists a unique maximal strong solution $(w,\Theta)$ of the equation (\ref{number3equation}). Moreover at $\mathbb{P}-a.e.$ $\omega$ for which $\Theta(\omega)<\infty$, we have that \begin{equation}\nonumber \sup_{r \in [0,\Theta(\omega))}\norm{w_r(\omega)}_{1}^2 + \int_0^{\Theta(\omega)}\norm{w_r(\omega)}_2^2dr = \infty.\end{equation}
\end{theorem}

\newpage

\section{Appendix} \label{appendix}

\subsection{Appendix I: Proofs from Section \ref{Section v valued}} \label{proofs section 3 appendix}

\begin{proof}[Proof of \ref{localexistenceofsolutionsofGalerkin}:]
It is clear that for any $R > \sup_{n\in\N}\norm{\sy^n_0}_{L^{\infty}(\Omega,H)}^2$ and the stopping time $$\tau^t_{n,R}(\omega):= t \wedge \inf\left\{s \geq 0: \sup_{r \in [0,s]}\norm{\sy^{n,R}_r(\omega)}_H^2 \geq R \right\}$$ that $(\sy^{n,R}, \tau^t_{n,R})$ is a local strong solution of (\ref{nthorderGalerkin}). We would therefore be done if there exists an $R > \sup_{n\in\N}\norm{\sy^n_0}_{L^{\infty}(\Omega,H)}^2$ such that $\tau^{M,t}_{n,R} \leq \tau^t_{n,R}$ $\mathbbm{P}-a.s.$. Due to the norm equivalence on $V_n$, there exists some constant $c_n$ such that the stopping time $$\tilde{\tau}^t_{n,R}(\omega):= t \wedge \inf\left\{s \geq 0: \sup_{r \in [0,s]}\norm{\sy^{n,R}_r(\omega)}_U^2 \geq c_nR \right\}$$ satisfies $\tilde{\tau}^t_{n,R} \leq \tau^t_{n,R} $ $\mathbbm{P}-a.s.$. Thus if we choose $$R > \min\left\{\sup_{n\in\N}\norm{\sy^n_0}_{L^{\infty}(\Omega,H)}^2, \frac{M + \sup_{n\in\N}\norm{\sy^n_0}_{L^{\infty}(\Omega,U)}^2}{c_n} \right\}$$ then for any $s \in [0,\tau^{M,t}_{n,R}(\omega)]$ we have that 
\begin{align*}
  \sup_{r \in [0,s]}\norm{\sy^{n,R}_{r}(\omega)}^2_U& \leq   \sup_{r \in [0,s]}\norm{\sy^{n,R}_{r}(\omega)}^2_U + \int_0^s\norm{\sy^{n,R}_{r}(\omega)}^2_Hdr\\ &\leq M + \norm{\sy^n_0(\omega)}_U^2\\ &\leq M + \sup_{n\in\N}\norm{\sy^n_0}_{L^{\infty}(\Omega,U)}^2\\ &\leq c_nR
\end{align*}
and hence $\tau^{M,t}_{n,R} \leq \tilde{\tau}^t_{n,R} \leq \tau^t_{n,R}$ $\mathbbm{P}-a.s.$, which implies the result.

\end{proof}

\begin{proof}[Proof of \ref{theorem:uniformbounds}:]
By equipping $V_n$ with the $H$ inner product we can apply the It\^{o} Formula for $\norm{\cdot}_H^2$ to see that for any $0 \leq r \leq t$, the identity \begin{align}\nonumber\norm{\sy^n_{r\wedge\tau^{M,t}_n}}_H^2 &= \norm{\sy^n_0}_H^2 + 2\int_0^{r\wedge\tau^{M,t}_n}\inner{\mathcal{P}_n\mathcal{A}(s,\sy^n_s)}{\sy^n_s}_Hds\\ & \qquad  + \int_0^{r\wedge\tau^{M,t}_n}\sum_{i=1}^\infty\norm{\mathcal{P}_n\mathcal{G}_i(s,\sy^n_s)}_H^2ds + 2\sum_{i=1}^\infty\int_0^{r\wedge\tau^{M,t}_n}\inner{\mathcal{P}_n\mathcal{G}_i(s,\sy^n_s)}{\sy^n_s}_HdW^i_s\nonumber
\end{align}
holds $\mathbbm{P}-a.s.$. Adopting the notation (\ref{tildesynotation}) to simplify proceedings prompts a shift from the above to 
\begin{align}\nonumber\norm{\tilde{\sy}^n_{r}}_H^2 = \norm{\sy^n_0}_H^2 &+ 2\int_0^{r}\inner{\mathcal{P}_n\mathcal{A}(s,\sy^n_s)}{\sy^n_s}_H\mathbbm{1}_{s \leq \tau^{M,t}_n}ds  + \int_0^{r}\sum_{i=1}^\infty\norm{\mathcal{P}_n\mathcal{G}_i(s,\sy^n_s)}_H^2\mathbbm{1}_{s \leq \tau^{M,t}_n}ds\\ & + 2\sum_{i=1}^\infty\int_0^{r}\inner{\mathcal{P}_n\mathcal{G}_i(s,\tilde{\sy}^n_s)}{\tilde{\sy}^n_s}_HdW^i_s.\nonumber
\end{align}
Note that we have left the indicator function outside of the time integral terms as we have no linearity assumption to take it through the  $\norm{\mathcal{P}_n\mathcal{G}_i(s,\sy^n_s)}_H^2$ term: more precisely, it may be the case that $\mathcal{P}_n\mathcal{G}_i(s,0) \neq 0$. To do this for the stochastic integral we are just relying on the linearity of the inner product. Let $0 \leq \theta_j < \theta_k \leq t$ be two arbitrary stopping times. By substituting in $\theta_j$ to the above, and then subtracting this from the identity for any $\theta_j \leq r \leq t$ $\mathbbm{P}-a.s.$, then we also have the equality
\begin{align}\nonumber\norm{\tilde{\sy}^n_{r}}_H^2 = \norm{\tilde{\sy}^n_{\theta_j}}_H^2 &+ 2\int_{\theta_j}^{r}\left(\inner{\mathcal{P}_n\mathcal{A}(s,\sy^n_s)}{\sy^n_s}_H + \sum_{i=1}^\infty\norm{\mathcal{P}_n\mathcal{G}_i(s,\sy^n_s)}_H^2\right)\mathbbm{1}_{s \leq \tau^{M,t}_n}ds\\ &+ 2\sum_{i=1}^\infty\int_{\theta_j}^{r}\inner{\mathcal{P}_n\mathcal{G}_i(s,\tilde{\sy}^n_s)}{\tilde{\sy}^n_s}_HdW^i_s\nonumber
\end{align}
$\mathbbm{P}-a.s.$. Applying (\ref{uniformboundsassumpt1}) to the process $\sy^n$, we deduce the inequality
\begin{align*}\norm{\tilde{\sy}^n_{r}}_H^2 \leq \norm{\tilde{\sy}^n_{\theta_j}}_H^2 &+ \int_{\theta_j}^{r}\left(c_{s}\tilde{K}_2(\tilde{\sy}^n_s)\left[1 + \norm{\tilde{\sy}^n_s}_H^2\right] - \kappa\norm{\tilde{\sy}^n_s}_V^2\right) ds\\  & \qquad \qquad \qquad \qquad + 2\sum_{i=1}^\infty\int_{\theta_j}^{r}\inner{\mathcal{P}_n\mathcal{G}_i(s,\tilde{\sy}^n_s)}{\tilde{\sy}^n_s}_HdW^i_s
\end{align*}
having now assimilated the indicator function through the norms into the $\tilde{\sy}^n$ (and recalling the notation of $\tilde{K}_2$ introduced in (\ref{Ktilde2})). We shall continue to use $c$ to represent a generic constant, which may well depend on the constants involved in our assumptions, as well as on $\sy_0$ and the choices of $M$ and $t$. The constant will not depend on $n,r$ or $\omega$. Using the boundedness of $c_{s}$ on $[0,t]$ and  (\ref{galerkinboundsatisfiedbystoppingtime2}), then we reduce the above to the inequality 
\begin{align*}\norm{\tilde{\sy}^n_{r}}_H^2 +\kappa\int_{\theta_j}^{r}\norm{\tilde{\sy}^n_s}_V^2ds &\leq \norm{\tilde{\sy}^n_{\theta_j}}_H^2\\ &  + c\int_{\theta_j}^{r}(1 + \norm{\tilde{\sy}^n_s}_H^2)^2 ds + 2\sum_{i=1}^\infty\int_{\theta_j}^{r}\inner{\mathcal{P}_n\mathcal{G}_i(s,\tilde{\sy}^n_s)}{\tilde{\sy}^n_s}_HdW^i_s.
\end{align*}
We can freely bound the stochastic integral by its absolute value, and as $r\in[\theta_j,t]$ was arbitrary we may take the supremum over all $r \in [\theta_j,\theta_k]$, deducing that 
\begin{align*}\sup_{r\in [\theta_j,\theta_k]}\norm{\tilde{\sy}^n_{r}}_H^2 + & \kappa\int_{\theta_j}^{\theta_k}\norm{\tilde{\sy}^n_s}_V^2ds \leq \norm{\tilde{\sy}^n_{\theta_j}}_H^2\\ & \qquad + c\int_{\theta_j}^{\theta_k}1 + \norm{\tilde{\sy}^n_s}_H^4 ds + 2\sup_{r \in [\theta_j,\theta_k]}\left\vert\sum_{i=1}^\infty\int_{\theta_j}^{r}\inner{\mathcal{P}_n\mathcal{G}_i(s,\tilde{\sy}^n_s)}{\tilde{\sy}^n_s}_HdW^i_s\right\vert
\end{align*}
having also used the simple manipulation $(1 + \norm{\tilde{\sy}^n_s}_H^2)^2 \leq c(1 + \norm{\tilde{\sy}^n_s}_H^4)$. This can be scaled and rewritten as 
\begin{align*}\norm{\tilde{\sy}^n}_{HV,\theta_j,\theta_k}^2 \leq c\norm{\tilde{\sy}^n_{\theta_j}}_H^2 + c\int_{\theta_j}^{\theta_k}1 + \norm{\tilde{\sy}^n_s}_H^4 ds + 2c\sup_{r \in [\theta_j,\theta_k]}\left\vert\sum_{i=1}^\infty\int_{\theta_j}^{r}\inner{\mathcal{P}_n\mathcal{G}_i(s,\tilde{\sy}^n_s)}{\tilde{\sy}^n_s}_HdW^i_s\right\vert.
\end{align*}
We are justified in taking the expectation here for the $\mathbbm{P}-a.s.$ inequality as the process is a genuine square integrable semimartingale, and again that this stochastic integral is a true square integrable martingale. There is no issue considering the supremum in a random time interval: we could simply absorb the randomness into the integrand through an indicator function and the result would be clear. Moreover
\begin{align*}\mathbbm{E}\norm{\tilde{\sy}^n}_{HV,\theta_j,\theta_k}^2 &\leq c\mathbbm{E}\norm{\tilde{\sy}^n_{\theta_j}}_H^2\\ &+ c\mathbbm{E}\int_{\theta_j}^{\theta_k}1 + \norm{\tilde{\sy}^n_s}_H^4ds + 2c\mathbbm{E}\sup_{r \in [\theta_j,\theta_k]}\left\vert\sum_{i=1}^\infty\int_{\theta_j}^{r}\inner{\mathcal{P}_n\mathcal{G}_i(s,\tilde{\sy}^n_s)}{\tilde{\sy}^n_s}dW^i_s\right\vert
\end{align*}
and we apply the Burkholder-Davis-Gundy Inequality (recalling this once more from \cite{goodair2022stochastic} Theorem 1.6.9) in conjunction with (\ref{uniformboundsassumpt2}) to further deduce that 
\begin{align*}\mathbbm{E}\norm{\tilde{\sy}^n}_{HV,\theta_j,\theta_k}^2 &\leq c\mathbbm{E}\norm{\tilde{\sy}^n_{\theta_j}}_H^2\\ &  + c\mathbbm{E}\int_{\theta_j}^{\theta_k}1 + \norm{\tilde{\sy}^n_s}_H^4ds + 2c\mathbbm{E}\left(\int_{\theta_j}^{\theta_k}c_{s}\tilde{K}_2(\tilde{\sy}^n_s)\left[1+\norm{\tilde{\sy}^n_s}_H^4\right] ds\right)^{\frac{1}{2}}
\end{align*}
and subsequently
\begin{align}\nonumber &\mathbbm{E}\norm{\tilde{\sy}^n}_{HV,\theta_j,\theta_k}^2 \leq c\mathbbm{E}\norm{\tilde{\sy}^n_{\theta_j}}_H^2\\  & \qquad + c\mathbbm{E}\int_{\theta_j}^{\theta_k}(1 + \norm{\tilde{\sy}^n_s}_H^2)^2 ds + c\mathbbm{E}\left(\int_{\theta_j}^{\theta_k} \left(1 + \norm{\tilde{\sy}^n_s}_H^2\right)\left[1+ \norm{\tilde{\sy}^n_s}_H^4\right] ds\right)^{\frac{1}{2}}. \label{beforeshiftinnotation}
\end{align}
Now we observe $\left(1 + \norm{\tilde{\sy}^n_s}_H^2\right)\left[1+ \norm{\tilde{\sy}^n_s}_H^4\right] \leq c\left(1 + \norm{\tilde{\sy}^n_s}_H^6\right)$ and
\begin{align*}
    \left(\int_{\theta_j}^{\theta_k}1 + \norm{\tilde{\sy}^n_s}_H^6 ds\right)^{\frac{1}{2}} 
    \leq t^{\frac{1}{2}} + \left(\int_{
    \theta_j}^{\theta_k}\norm{\tilde{\sy}^n_s}_H^6 ds\right)^{\frac{1}{2}}
\end{align*}
so using that our generic $c$ may depend on $t$ we reduce (\ref{beforeshiftinnotation}) to
\begin{align}\mathbbm{E}\norm{\tilde{\sy}^n}_{HV,\theta_j,\theta_k}^2 &\leq c\mathbbm{E}\norm{\tilde{\sy}^n_{\theta_j}}_H^2 + c + c\mathbbm{E}\int_{\theta_j}^{\theta_k}\norm{\tilde{\sy}^n_s}_H^4 ds + c\mathbbm{E}\left(\int_{\theta_j}^{\theta_k}\norm{\tilde{\sy}^n_s}_H^6 ds\right)^{\frac{1}{2}}. \label{beforeshiftinnotation*}
\end{align}
Now we have that
\begin{align}
    \nonumber c\left(\int_{\theta_j}^{\theta_k} \norm{\tilde{\sy}^n_s}_H^6 ds\right)^{\frac{1}{2}} &= c\left(\int_{\theta_j}^{\theta_k}\norm{\tilde{\sy}^n_s}_H^2\norm{\tilde{\sy}^n_s}_H^4 ds\right)^{\frac{1}{2}}\\ \nonumber
    &\leq c\left(\sup_{s \in [{\theta_j},\theta_k]}\norm{\tilde{\sy}^n_s}_H^2\int_{\theta_j}^{\theta_k} \norm{\tilde{\sy}^n_s}_H^4 ds\right)^{\frac{1}{2}}\\ \nonumber
    &= c\left(\sup_{s \in [{\theta_j},\theta_k]}\norm{\tilde{\sy}^n_s}_H^2\right)^{\frac{1}{2}}\left(\int_{\theta_j}^{\theta_k} \norm{\tilde{\sy}^n_s}_H^4 ds\right)^{\frac{1}{2}}\\ \label{howweusebdg}
    &\leq \frac{1}{2}\sup_{r \in [{\theta_j},\theta_k]}\norm{\tilde{\sy}^n_r}_H^2 + \frac{c^2}{2}\int_{\theta_j}^{\theta_k}\norm{\tilde{\sy}^n_s}_H^4 ds
\end{align}
via an application of Young's Inequality. Taking the expectation and then the supremum term over to the left hand side whilst absorbing the integral term into what we already have, we reduce (\ref{beforeshiftinnotation*}) to \begin{equation} \label{finaluniformboundbeforegronwall}
    \mathbbm{E}\norm{\tilde{\sy}^n}_{HV,\theta_j,\theta_k}^2 \leq c\mathbbm{E}\norm{\sy^n_{\theta_j}}_H^2+c+ c\mathbbm{E}\int_{\theta_j}^{\theta_k}\norm{\tilde{\sy}^n_s}_H^4 ds.
\end{equation}
having scaled once more, and rewrite this as 
\begin{equation} \label{finaluniformboundbeforegronwall2}
    \mathbbm{E}\norm{\tilde{\sy}^n}_{HV,\theta_j,\theta_k}^2 \leq \hat{c}\mathbbm{E}\left(\norm{\sy^n_{\theta_j}}_H^2+1+ \int_{\theta_j}^{\theta_k}\norm{\tilde{\sy}^n_s}_H^4 ds\right).
\end{equation}
for some particular $\hat{c}$ which we shall reference. Now we may apply the Stochastic Gr\'{o}nwall lemma (Lemma \ref{gronny}) for the processes $$\boldsymbol{\phi} = \norm{\tilde{\sy}^n}_H^2, \qquad \boldsymbol{\psi}= \norm{\tilde{\sy}^n}_V^2, \qquad \boldsymbol{\eta}= \norm{\tilde{\sy}^n}_H^2$$ noting that the bound (\ref{boundingronny}) comes from (\ref{galerkinboundsatisfiedbystoppingtime}). The application of this result conclude the proof.

\end{proof}

\begin{proof}[Proof of \ref{lemma to use cauchy}:]
We show the inequalities independently, starting with $A$ through
\begin{align}
    \nonumber A &=  2\inner{\mathcal{P}_n\left[\mathcal{A}(s,\boldsymbol{\phi}) - \mathcal{A}(s,\boldsymbol{\psi})
    \right] + \left[\mathcal{P}_n- \mathcal{P}_m\right]\mathcal{A}(s,\boldsymbol{\psi})}{\boldsymbol{\phi} - \boldsymbol{\psi}}_U \\\nonumber  & \qquad \qquad \qquad \qquad \qquad \qquad + \sum_{i=1}^\infty\norm{\mathcal{P}_n\left[\mathcal{G}_i(s,\boldsymbol{\phi}) - \mathcal{G}_i(s,\boldsymbol{\psi})\right] + \left[\mathcal{P}_n- \mathcal{P}_m\right]\mathcal{G}_i(s,\boldsymbol{\psi})}_U^2\\\nonumber 
    & \leq 2\inner{\mathcal{P}_n\left[\mathcal{A}(s,\boldsymbol{\phi}) - \mathcal{A}(s,\boldsymbol{\psi})
    \right] }{\boldsymbol{\phi} - \boldsymbol{\psi}}_U + \sum_{i=1}^\infty\norm{\mathcal{P}_n\left[\mathcal{G}_i(s,\boldsymbol{\phi}) - \mathcal{G}_i(s,\boldsymbol{\psi})\right]}_U^2\\ \nonumber & \qquad \qquad \qquad + 2\inner{\left[\mathcal{P}_n- \mathcal{P}_m\right]\mathcal{A}(s,\boldsymbol{\psi})}{\boldsymbol{\phi} - \boldsymbol{\psi}}_U + \sum_{i=1}^\infty\norm{\left[\mathcal{P}_n- \mathcal{P}_m\right]\mathcal{G}_i(s,\boldsymbol{\psi})}_U^2\\\nonumber 
    & \qquad \qquad \qquad \qquad \qquad \qquad + 2\sum_{i=1}^\infty\norm{\mathcal{P}_n\left[\mathcal{G}_i(s,\boldsymbol{\phi}) -\mathcal{G}_i(s,\boldsymbol{\psi})\right]}_U\norm{\left[\mathcal{P}_n- \mathcal{P}_m\right]\mathcal{G}_i(s,\boldsymbol{\psi})}_U\\\nonumber
    &\leq 2\inner{\mathcal{A}(s,\boldsymbol{\phi}) - \mathcal{A}(s,\boldsymbol{\psi})
     }{\boldsymbol{\phi} - \boldsymbol{\psi}}_U + \sum_{i=1}^\infty\norm{\left[\mathcal{G}_i(s,\boldsymbol{\phi}) - \mathcal{G}_i(s,\boldsymbol{\psi}\right]}_U^2\\ \nonumber& \qquad \qquad \qquad + 2\inner{\mathcal{A}(s,\boldsymbol{\psi})}{[I-\mathcal{P}_m]\boldsymbol{\phi} - \boldsymbol{\psi}}_U + \sum_{i=1}^\infty\norm{\mathcal{P}_n\left[I - \mathcal{P}_m \right]\mathcal{G}_i(s,\boldsymbol{\psi})}_U^2\\ \nonumber
    & \qquad \qquad \qquad \qquad \qquad \qquad + 2\sum_{i=1}^\infty\norm{\mathcal{G}_i(s,\boldsymbol{\phi}) -\mathcal{G}_i(s,\boldsymbol{\psi})}_U\norm{\mathcal{P}_n\left[I - \mathcal{P}_m \right]\mathcal{G}_i(s,\boldsymbol{\psi})}_U
    \\
    &:=\alpha + \beta +\gamma \nonumber
\end{align}
having used that $\mathcal{P}_n$ is an orthogonal projection on $U$ and $\mathcal{P}_n\mathcal{P}_m = \mathcal{P}_m$.
We look to show the appropriate bounds on $\alpha$, $\beta$ and $\gamma$. For $\alpha$ we simply apply (\ref{therealcauchy1}). Moving on to $\beta$, we use again that $\mathcal{P}_n$ is an orthogonal projection and the property (\ref{mu2}) to see that $$\beta \leq \frac{2}{\mu_m}\norm{\mathcal{A}(s,\boldsymbol{\psi})}_U\norm{\boldsymbol{\phi} - \boldsymbol{\psi}}_H + \sum_{i=1}^\infty \frac{1}{\mu_m^2}\norm{\mathcal{G}_i(s,\boldsymbol{\psi})}_H^2.$$ Through Young's Inequality with a constant $c$ dependent on $\kappa$, we can bound this further by $$\frac{c}{\mu_m^2}\left(\norm{\mathcal{A}(s,\boldsymbol{\psi})}_U^2 + \sum_{i=1}^\infty\norm{\mathcal{G}_i(s,\boldsymbol{\psi})}_H^2 \right) + \frac{\kappa}{2}\norm{\boldsymbol{\phi} - \boldsymbol{\psi}}_H^2$$ to which we apply (\ref{111}) to the bracketed term. As for $\gamma$, we have that
\begin{align*}
   \gamma &\leq \sum_{i=1}^\infty\left(\norm{\mathcal{G}_i(s,\boldsymbol{\phi})}_U + \norm{\mathcal{G}_i(s,\boldsymbol{\psi})}_U\right)\norm{\left[I - \mathcal{P}_m \right]\mathcal{G}_i(s,\boldsymbol{\psi})}_U\\
   &\leq c\sum_{i=1}^\infty\left(\norm{\mathcal{G}_i(s,\boldsymbol{\phi})}_H + \norm{\mathcal{G}_i(s,\boldsymbol{\psi})}_H\right)\norm{\left[I - \mathcal{P}_m \right]\mathcal{G}_i(s,\boldsymbol{\psi})}_U\\
   &\leq \frac{c}{\mu_m}\sum_{i=1}^\infty\left(\norm{\mathcal{G}_i(s,\boldsymbol{\phi})}_H + \norm{\mathcal{G}_i(s,\boldsymbol{\psi})}_H\right)\norm{\mathcal{G}_i(s,\boldsymbol{\psi})}_H\\
   &\leq \frac{c}{\mu_m}\sum_{i=1}^\infty\left(\norm{\mathcal{G}_i(s,\boldsymbol{\phi})}_H^2 + \norm{\mathcal{G}_i(s,\boldsymbol{\psi})}_H^2\right)
\end{align*}
which we handle through (\ref{111}) once more. Altogether then, with notation $\lambda_m:= \min\{\mu_m,\mu_m^2\}$, we have that  
\begin{align*}
    A &\leq c_{t}\tilde{K}_2(\boldsymbol{\phi},\boldsymbol{\psi}) \norm{\boldsymbol{\phi}-\boldsymbol{\psi}}_U^2 - \kappa\norm{\boldsymbol{\phi}-\boldsymbol{\psi}}_H^2 + \frac{c}{\lambda_m}\left(c_s K(\boldsymbol{\phi})\left[1 + \norm{\boldsymbol{\phi}}_V^2\right] \right) + \frac{\kappa}{2} \norm{\boldsymbol{\phi} - \boldsymbol{\psi}}_H^2\\ & \qquad \qquad \qquad + \frac{c}{\lambda_m}\left(c_s K(\boldsymbol{\phi})\left[1 + \norm{\boldsymbol{\phi}}_V^2\right] + c_s K(\boldsymbol{\psi})\left[1 + \norm{\boldsymbol{\psi}}_V^2\right] \right)\\
    &\leq c_{t}\tilde{K}_2(\boldsymbol{\phi},\boldsymbol{\psi}) \norm{\boldsymbol{\phi}-\boldsymbol{\psi}}_U^2 - \frac{\kappa}{2}\norm{\boldsymbol{\phi}-\boldsymbol{\psi}}_H^2 + \frac{c_s}{\lambda_m}K(\boldsymbol{\phi},\boldsymbol{\psi})\left[1 + \norm{\boldsymbol{\phi}}_V^2 + \norm{\boldsymbol{\psi}}_V^2\right] 
\end{align*}
as required. It remains to show the bound on $B$, which we approach in a similar manner:
\begin{align*}
    B &= \sum_{i=1}^\infty \inner{\mathcal{P}_n\left[\mathcal{G}_i(s,\boldsymbol{\phi}) - \mathcal{G}_i(s,\boldsymbol{\psi})\right] + \left[\mathcal{P}_n- \mathcal{P}_m\right]\mathcal{G}_i(s,\boldsymbol{\psi})}{\boldsymbol{\phi}-\boldsymbol{\psi}}^2_U\\
    &\leq 2\sum_{i=1}^\infty\left(\inner{\mathcal{G}_i(s,\boldsymbol{\phi}) - \mathcal{G}_i(s,\boldsymbol{\psi})}{\boldsymbol{\phi}-\boldsymbol{\psi}}^2_U + \inner{\left[I- \mathcal{P}_m\right]\mathcal{G}_i(s,\boldsymbol{\psi})}{\boldsymbol{\phi}-\boldsymbol{\psi}}^2_U\right).
\end{align*}
The first term here is precisely what we have in (\ref{therealcauchy2}). For the second, we have that
\begin{align*}
    \inner{\left[I- \mathcal{P}_m\right]\mathcal{G}_i(s,\boldsymbol{\psi})}{\boldsymbol{\phi}-\boldsymbol{\psi}}^2_U &\leq \frac{1}{\mu_m^2}\norm{\mathcal{G}_i(s,\boldsymbol{\psi})}_H^2\norm{\boldsymbol{\phi}-\boldsymbol{\psi}}_U^2\\
    &\leq \frac{c_s}{\mu_m^2} K(\boldsymbol{\phi},\boldsymbol{\psi})\left[1 + \norm{\boldsymbol{\psi}}_V^2\right]
\end{align*}
as required. We note that this is a coarse bound, though sufficient for our purposes. 

\end{proof}

\begin{proof}[Proof of \ref{theorem:cauchygalerkin}:]
The proof uses very similar methods to those used in Proposition \ref{theorem:uniformbounds}, this time relying on Lemma \ref{lemma to use cauchy} instead of Assumption \ref{assumptions for uniform bounds2}. For any $0 \leq r \leq t$, $\sy^{m,n}$ satisfies the identity
\begin{align*}
    \sy^{m,n}_{r \wedge \tau^{M,t}_m \wedge \tau^{M,t}_n} = \sy^{m,n}_0 & + \int_0^{r\wedge \tau^{M,t}_m \wedge \tau^{M,t}_n}\mathcal{P}_n\mathcal{A}(s,\sy^n_s) - \mathcal{P}_m\mathcal{A}(s,\sy^m_s) ds\\
    & +\sum_{i=1}^\infty \int_0^{r \wedge \tau^{M,t}_m \wedge \tau^{M,t}_n}\mathcal{P}_n\mathcal{G}_i(s,\sy^n_s) - \mathcal{P}_m\mathcal{G}_i(s,\sy^m_s) dW^i_s
\end{align*}
$\mathbbm{P}-a.s.$ in $V_n$, noting that the difference process $\sy^{m,n}_{\cdot \wedge \tau^{M,t}_m \wedge \tau^{M,t}_n}$ is again a genuine square integrable semimartingale in $U$. We thus apply the It\^{o} Formula to this difference process, reaching the equality
\begin{align*}
    &\norm{\sy^{m,n}_{r \wedge \tau^{M,t}_m \wedge \tau^{M,t}_n} }_U^2 = \norm{\sy^{m,n}_0}_U^2 + 2\int_0^{r\wedge \tau^{M,t}_m \wedge \tau^{M,t}_n} \inner{\mathcal{P}_n\mathcal{A}(s,\sy^n_s) - \mathcal{P}_m\mathcal{A}(s,\sy^m_s)}{\sy^{m,n}_s}_Uds\\ & \qquad \qquad +\int_0^{r\wedge \tau^{M,t}_m \wedge \tau^{M,t}_n}\sum_{i=1}^\infty\norm{\mathcal{P}_n\mathcal{G}_i(s,\sy^n_s) - \mathcal{P}_m\mathcal{G}_i(s,\sy^m_s)}_U^2ds\\  & \qquad \qquad \qquad \qquad + 2\sum_{i=1}^\infty \int_0^{r\wedge \tau^{M,t}_m \wedge \tau^{M,t}_n}\inner{\mathcal{P}_n\mathcal{G}_i(s,\sy^n_s) - \mathcal{P}_m\mathcal{G}_i(s,\sy^m_s)}{\sy^{m,n}_s}_UdW^i_s
\end{align*}
to which we introduce notation similar to (\ref{tildesynotation}), that is \begin{equation} \label{more tilde notation}\tilde{\sy}^m_\cdot:= \sy^m_{\cdot}\mathbbm{1}_{\cdot \leq \tau^{M,t}_m \wedge \tau^{M,t}_n}, \qquad \tilde{\sy}^n_\cdot:= \sy^n_{\cdot}\mathbbm{1}_{\cdot \leq \tau^{M,t}_m \wedge \tau^{M,t}_n}, \qquad \tilde{\sy}^{m,n}:= \tilde{\sy}^n-\tilde{\sy}^m\end{equation}
and thus rewrite our equality as
\begin{align*}
        \norm{\tilde{\sy}^{m,n}_{r}}_U^2 &= \norm{\sy^{m,n}_0}_U^2 + 2\int_0^{r} \inner{\mathcal{P}_n\mathcal{A}(s,\sy^n_s) - \mathcal{P}_m\mathcal{A}(s,\sy^m_s)}{\sy^{m,n}_s}_U\mathbbm{1}_{s \leq \tau^{M,t}_m \wedge \tau^{M,t}_n}ds\\ & \qquad  \qquad \qquad  +\int_0^{r}\sum_{i=1}^\infty\norm{\mathcal{P}_n\mathcal{G}_i(s,\sy^n_s) - \mathcal{P}_m\mathcal{G}_i(s,\sy^m_s)}_U^2\mathbbm{1}_{s \leq \tau^{M,t}_m \wedge \tau^{m,t}_n}ds\\ & \qquad\qquad \qquad \qquad \qquad \qquad + 2\sum_{i=1}^\infty \int_0^{r}\inner{\mathcal{P}_n\mathcal{G}_i(s,\tilde{\sy}^n_s) - \mathcal{P}_m\mathcal{G}_i(s,\tilde{\sy}^m_s)}{\tilde{\sy}^{m,n}_{s}}_UdW^i_s.
\end{align*}
Identically to \ref{theorem:uniformbounds}, we fix arbitrary stopping times $0 \leq \theta_j < \theta_k \leq t$ and have that for any $\theta_j \leq r \leq t$ $\mathbbm{P}-a.s.$, 
\begin{align*}
        \norm{\tilde{\sy}^{m,n}_{r}}_U^2 &= \norm{\tilde{\sy}^{m,n}_{\theta_j}}_U^2 + 2\int_{\theta_j}^{r} \inner{\mathcal{P}_n\mathcal{A}(s,\sy^n_s) - \mathcal{P}_m\mathcal{A}(s,\sy^m_s)}{\sy^{m,n}_s}_U\mathbbm{1}_{s \leq \tau^{M,t}_m \wedge \tau^{M,t}_n}ds\\ & \qquad  \qquad \qquad  +\int_{\theta_j}^{r}\sum_{i=1}^\infty\norm{\mathcal{P}_n\mathcal{G}_i(s,\sy^n_s) - \mathcal{P}_m\mathcal{G}_i(s,\sy^m_s)}_U^2\mathbbm{1}_{s \leq \tau^{M,t}_m \wedge \tau^{m,t}_n}ds\\ & \qquad\qquad \qquad \qquad \qquad \qquad + 2\sum_{i=1}^\infty \int_{\theta_j}^{r}\inner{\mathcal{P}_n\mathcal{G}_i(s,\tilde{\sy}^n_s) - \mathcal{P}_m\mathcal{G}_i(s,\tilde{\sy}^m_s)}{\tilde{\sy}^{m,n}_{s}}_UdW^i_s.
\end{align*}
$\mathbbm{P}-a.s.$. Combining the time integrals and applying the bound on $A$ in Lemma \ref{lemma to use cauchy}, we deduce the inequality
\begin{align*}
        &\norm{\tilde{\sy}^{m,n}_{r}}_U^2 \leq \norm{\tilde{\sy}^{m,n}_{\theta_j}}_U^2\\& \quad + \int_{\theta_j}^{r}\bigg[
        c_{s}\tilde{K}_2(\tilde{\sy}^m_s,\tilde{\sy}^n_s)\norm{\tilde{\sy}^{m,n}_{s}}_U^2 - \frac{\kappa}{2}\norm{\tilde{\sy}^{m,n}_{s}}_H^2 + \frac{c_{s}}{\lambda_m} K(\tilde{\sy}^m_s,\tilde{\sy}^n_s)\left[1 + \norm{\tilde{\sy}^m_s}_V^2 + \norm{\tilde{\sy}^n_s}_V^2\right]\bigg]ds \\& \quad \quad 
        + 2\sum_{i=1}^\infty \int_{\theta_j}^{r}\inner{\mathcal{P}_n\mathcal{G}_i(s,\tilde{\sy}^n_s) - \mathcal{P}_m\mathcal{G}_i(s,\tilde{\sy}^m_s)}{\tilde{\sy}^{m,n}_{s}}_UdW^i_s.
\end{align*}
Using (\ref{galerkinboundsatisfiedbystoppingtime2}) and the boundedness of $c_{\cdot}$ on $[0,t]$ then we can reduce the above to 
\begin{align*}
        &\norm{\tilde{\sy}^{m,n}_{r}}_U^2 + \frac{\kappa}{2}\int_{\theta_j}^{r}\norm{\tilde{\sy}^{m,n}_{s}}_H^2ds\leq \norm{\tilde{\sy}^{m,n}_{\theta_j}}_U^2\\& \qquad \qquad \qquad + c\int_{\theta_j}^{r}
        \tilde{K}_2(\tilde{\sy}^m_s,\tilde{\sy}^n_s)\norm{\tilde{\sy}^{m,n}_{s}}_U^2 ds  + \frac{c}{\lambda_m}\int_{\theta_j}^{r} 1 + \norm{\tilde{\sy}^m_s}_V^2 + \norm{\tilde{\sy}^n_s}_V^2ds \\& \qquad \qquad \qquad \qquad \qquad \qquad
        + 2\sum_{i=1}^\infty \int_{\theta_j}^{r}\inner{\mathcal{P}_n\mathcal{G}_i(s,\tilde{\sy}^n_s) - \mathcal{P}_m\mathcal{G}_i(s,\tilde{\sy}^m_s)}{\tilde{\sy}^{m,n}_{s}}_UdW^i_s.
\end{align*}
Just as we did in Proposition \ref{theorem:uniformbounds} we now scale, take the absolute value of the stochastic integral and the supremum over $r\in[\theta_j,\theta_k]$, taking our inequality to
\begin{align*}
        \norm{\tilde{\sy}^{m,n}}_{UH,\theta_j,\theta_k}^2 &\leq c\norm{\tilde{\sy}^{m,n}_{\theta_j}}_U^2\\&   + c\int_{\theta_j}^{\theta_k} \tilde{K}_2(\tilde{\sy}^m_s,\tilde{\sy}^n_s) \norm{\tilde{\sy}^{m,n}_{s}}_U^2ds + \frac{c}{\lambda_m}\int_{\theta_j}^{\theta_k}1 + \norm{\tilde{\sy}^m_s}_V^2 + \norm{\tilde{\sy}^n_s}_V^2 ds\\ &
        + 2c\sup_{r \in [\theta_j,\theta_k]}\left\vert\sum_{i=1}^\infty \int_{\theta_j}^{r}\inner{\mathcal{P}_n\mathcal{G}_i(s,\tilde{\sy}^n_s) - \mathcal{P}_m\mathcal{G}_i(s,\tilde{\sy}^m_s)}{\tilde{\sy}^{m,n}_{s}}_UdW^i_s\right\vert.
\end{align*}
$\mathbbm{P}-a.s.$. All in one step we take the expectation, apply the Burkholder-Davis-Gundy Inequality and employ the inequality for $B$ in Lemma \ref{lemma to use cauchy} to deduce now
\begin{align*}
        \mathbbm{E}\norm{\tilde{\sy}^{m,n}}_{UH,\theta_j,\theta_k}^2 &\leq c\mathbbm{E}\norm{\tilde{\sy}^{m,n}_{\theta_j}}_U^2\\&  + c\mathbbm{E}\int_{\theta_j}^{\theta_k} \tilde{K}_2(\tilde{\sy}^m_s,\tilde{\sy}^n_s) \norm{\tilde{\sy}^{m,n}_{s}}_U^2ds + \frac{c}{\lambda_m}\mathbbm{E}\int_{\theta_j}^{\theta_k} 1 + \norm{\tilde{\sy}^m_s}_V^2 + \norm{\tilde{\sy}^n_s}_V^2 ds\\ &
        + 2c\mathbbm{E}\bigg(\int_{\theta_j}^{\theta_k} c_{s} \tilde{K}_2(\tilde{\sy}^m_s,\tilde{\sy}^n_s) \norm{\tilde{\sy}^{m,n}_{s}}_U^4  + \frac{c_{s}}{\lambda_m}K(\tilde{\sy}^m_s,\tilde{\sy}^n_s)\left[1 + \norm{\tilde{\sy}^m_s}_V^2\right]ds\bigg)^\frac{1}{2}
\end{align*}
and subsequently
\begin{align}
        \nonumber \mathbbm{E}\norm{\tilde{\sy}^{m,n}}_{UH,\theta_j,\theta_k}^2 &\leq \mathbbm{E}\norm{\tilde{\sy}^{m,n}_{\theta_j}}_U^2\\& + c\mathbbm{E}\int_{\theta_j}^{\theta_k} \tilde{K}_2(\tilde{\sy}^m_s,\tilde{\sy}^n_s) \norm{\tilde{\sy}^{m,n}_{s}}_U^2ds + \frac{c}{\lambda_m}\mathbbm{E}\int_{\theta_j}^{\theta_k} 1 + \norm{\tilde{\sy}^m_s}_V^2 + \norm{\tilde{\sy}^n_s}_V^2 ds\nonumber\\ & 
        + c\mathbbm{E}\bigg(\int_{\theta_j}^{\theta_k} \tilde{K}_2(\tilde{\sy}^m_s,\tilde{\sy}^n_s) \norm{\tilde{\sy}_s^n-\tilde{\sy}_s^m}_U^4  + \frac{1}{\lambda_m}\left[1 + \norm{\tilde{\sy}^m_s}_V^2\right]ds\bigg)^\frac{1}{2}. \label{andsubsequently}
\end{align}
We appreciate now that
\begin{align} \label{BDGbound}
    &c\bigg(\int_{\theta_j}^{\theta_k} \tilde{K}_2(\tilde{\sy}^m_s,\tilde{\sy}^n_s) \norm{\tilde{\sy}_s^n-\tilde{\sy}_s^m}_U^4  + \frac{1}{\lambda_m}\left[1 + \norm{\tilde{\sy}^m_s}_V^2\right]ds\bigg)^\frac{1}{2}\\ & \qquad  \leq c\left(\int_{\theta_j}^{\theta_k}\tilde{K}_2(\tilde{\sy}^m_s,\tilde{\sy}^n_s)\norm{\tilde{\sy}_s^n-\tilde{\sy}_s^m}_U^4ds\right)^{\frac{1}{2}}  +c\left(\int_{\theta_j}^{\theta_k}\frac{1}{\lambda_m}\left[1 + \norm{\tilde{\sy}^m_s}_V^2\right]ds \right)^{\frac{1}{2}} \nonumber
\end{align}
and treat these terms individually. Through the same process as (\ref{howweusebdg}), we see that
\begin{align*}
    &c\left(\int_{\theta_j}^{\theta_k}\tilde{K}_2(\tilde{\sy}^m_s,\tilde{\sy}^n_s)\norm{\tilde{\sy}_s^n-\tilde{\sy}_s^m}_U^4ds\right)^{\frac{1}{2}}\\ & \qquad \qquad \qquad \qquad \leq \frac{1}{2}\sup_{r \in [\theta_j,\theta_k]}\norm{\tilde{\sy}_r^n-\tilde{\sy}_r^m}_U^2 + \frac{c^2}{2}\int_{\theta_j}^{\theta_k}\tilde{K}_2(\tilde{\sy}^m_s,\tilde{\sy}^n_s)\norm{\tilde{\sy}_s^n-\tilde{\sy}_s^m}_U^2ds.
\end{align*}
As $m$ is sufficiently large so that $\lambda_m \geq 1$, then $\frac{1}{\lambda_m} \leq \frac{1}{\sqrt{\lambda_m}}$. Introducing this bound into (\ref{andsubsequently}) along with the deduced restraint on (\ref{BDGbound}), we have that
\begin{align}
    \nonumber \norm{\tilde{\sy}^{m,n}}_{UH,\theta_j,\theta_k}^2 &\leq c\mathbbm{E}\norm{\tilde{\sy}^{m,n}_{\theta_j}}_U^2\\\nonumber  & + c\mathbbm{E}\int_{\theta_j}^{\theta_k} \tilde{K}_2(\tilde{\sy}^m_s,\tilde{\sy}^n_s) \norm{\tilde{\sy}^{m,n}_{s}}_U^2ds\\ & + \frac{c}{\sqrt{\lambda_m}}\mathbbm{E}\left[ \int_{\theta_j}^{\theta_k} 1 + \norm{\tilde{\sy}^m_s}_V^2 + \norm{\tilde{\sy}^n_s}_V^2 ds + \left(\int_{\theta_j}^{\theta_k}1 + \norm{\tilde{\sy}^m_s}_V^2ds \right)^{\frac{1}{2}} \right]. \label{cauchyfinalbeforegronwall}
\end{align}
Considering the $\frac{c}{\sqrt{\lambda_m}}$ term, we can of course bound these integrals by integrating over the whole interval $[0,t]$, and we also have $$\mathbbm{E}\left(\int_{0}^{t}1 + \norm{\tilde{\sy}^m_s}_V^2ds \right)^{\frac{1}{2}} \leq \left(\mathbbm{E}\int_{0}^{t}1 + \norm{\tilde{\sy}^m_s}_V^2ds \right)^{\frac{1}{2}}.$$
It is here that we apply the result (\ref{secondresultofuniformbounds}), reducing (\ref{cauchyfinalbeforegronwall}) to
\begin{align}
    \nonumber \mathbbm{E}\norm{\tilde{\sy}^{m,n}}_{UH,\theta_j,\theta_k}^2 \leq c\mathbbm{E}\norm{\tilde{\sy}^{m,n}_{\theta_j}}_U^2  + c\mathbbm{E}\int_{\theta_j}^{\theta_k} \tilde{K}_2(\tilde{\sy}^m_s,\tilde{\sy}^n_s) \norm{\tilde{\sy}^{m,n}_{s}}_U^2ds + \frac{c}{\sqrt{\lambda_m}}
\end{align}
which is reminiscent of (\ref{finaluniformboundbeforegronwall}), and as such we rewrite it in the form (\ref{finaluniformboundbeforegronwall2}) as
\begin{align}
    \mathbbm{E}\norm{\tilde{\sy}^{m,n}}_{UH,\theta_j,\theta_k}^2 \leq \hat{c}\mathbbm{E}\left(\norm{\tilde{\sy}^{m,n}_{\theta_j}}_U^2 + \frac{1}{\sqrt{\lambda_m}} + \int_{\theta_j}^{\theta_k}\tilde{K}_2(\tilde{\sy}^m_s,\tilde{\sy}^n_s)\norm{\tilde{\sy}^{m,n}_{s}}_U^2ds \right). \label{lelabel}
\end{align}
We are again precisely in the setting of the Stochastic Gr\'{o}nwall lemma (Lemma \ref{gronny}), for the processes
$$\boldsymbol{\phi}= \norm{\tilde{\sy}^{m,n}}_U^2, \qquad \boldsymbol{\psi}= \norm{\tilde{\sy}^{m,n}}_H^2, \qquad \boldsymbol{\eta}= \tilde{K}_2(\tilde{\sy}^m_s,\tilde{\sy}^n_s)$$
where we appreciate that this $\boldsymbol{\eta}$ satisfies (\ref{boundingronny}) from (\ref{galerkinboundsatisfiedbystoppingtime}). Applying Lemma \ref{gronny} we deduce the existence of a constant $C$ with the same genericity as our usual $c$ such that
$$\mathbbm{E}\norm{\tilde{\sy}^{m,n}}_{UH,t}^2 \leq C\left[\mathbbm{E}\norm{\sy^{m,n}_0}_U^2 + \frac{1}{\sqrt{\lambda_m}} \right]$$ which implies (\ref{cauchy result pt 1}). As for (\ref{cauchy result pt 2}), note that $$ \norm{\sy^{m,n}_0}_U^2 = \norm{\mathcal{P}_n(I-\mathcal{P}_m)\sy_{0}}_U^2 \leq \norm{(I-\mathcal{P}_m)\sy_{0}}_U^2 \leq \frac{1}{\lambda_m}\norm{\sy_0}_H^2$$ from (\ref{mu2}) which is a ($\mathbbm{P}-a.s.)$  monotone decreasing sequence in $m$ convergent to zero. We therefore have \begin{align*}\lim_{m \rightarrow \infty}\sup_{n \geq m}\mathbbm{E}\norm{\tilde{\sy}^{m,n}}_{UH,0,t}^2
&\leq \lim_{m \rightarrow \infty}\left[\mathbbm{E} \left(\norm{(I-\mathcal{P}_m)\sy_{0}}_U^2\right) + \frac{1}{\sqrt{\lambda_m} } \right]\\
&=0
\end{align*}
from the monotone convergence theorem, and limit of the $(\lambda_m)$ being infinite. This proves (\ref{cauchy result pt 2}).

\end{proof}

\begin{proof}[Proof of \ref{theorem: probability one}:]
We immediately note that (\ref{sufficient thing in probability condition}) implies (\ref{qwerty}) as
\begin{align*}
    \mathbb{P}\Bigg(\Bigg\{ \norm{\sy^{n}}^2_{UH,\tau^{M,t}_n \wedge S} \geq M-1+\norm{\sy^n_0}_{U}^2 \Bigg\}\Bigg) &= \mathbb{P}\left(\left\{ \norm{\sy^{n}}^2_{UH,\tau^{M,t}_n \wedge S} - \norm{\sy^n_0}_{U}^2 \geq M-1 \right\}\right)\\ &\leq \frac{1}{M-1}\mathbb{E}\left[\norm{\sy^{n}}^2_{UH,\tau^{M,t}_n \wedge S} - \norm{\sy^n_0}_{U}^2 \right] 
\end{align*}
having applied Chebyshev's Inequality, appreciating that the random variable in the expectation is non-negative. In order to achieve this we first appreciate that for any $\phi \in V_n$, \begin{align*}2\inner{\mathcal{P}_n\mathcal{A}(s,\boldsymbol{\phi})}{\boldsymbol{\phi}}_U + \sum_{i=1}^\infty\norm{\mathcal{P}_n\mathcal{G}_i(s,\boldsymbol{\phi})}_U^2 &\leq 2\inner{\mathcal{A}(s,\boldsymbol{\phi})}{\boldsymbol{\phi}}_U + \sum_{i=1}^\infty\norm{\mathcal{G}_i(s,\boldsymbol{\phi})}_U^2\\
&\leq c_sK(\boldsymbol{\phi})\left[1 + \norm{\boldsymbol{\phi}}_H^2\right]\\
&= c_sK(\boldsymbol{\phi})\left[1 + \norm{\boldsymbol{\phi}}_H^2\right] - \norm{\boldsymbol{\phi}}_H^2
\end{align*} 
having first applied (\ref{probability first}) and then simply absorbing an additional $\norm{\boldsymbol{\phi}}_H^2$ into the $c_sK(\boldsymbol{\phi})\left[1 + \norm{\boldsymbol{\phi}}_H^2\right]$ for the final equality. In the same vein we also have
\begin{equation}
   \nonumber \sum_{i=1}^\infty \inner{\mathcal{P}_n\mathcal{G}_i(s,\boldsymbol{\phi})}{\boldsymbol{\phi}}^2_U \leq c_sK(\boldsymbol{\phi})\left[1 + \norm{\boldsymbol{\phi}}_H^4\right]
\end{equation}
coming from (\ref{probability second}). In an identical manner to Proposition \ref{theorem:uniformbounds} and as seen again in Proposition \ref{theorem:cauchygalerkin}, we apply the It\^{o} Formula for $\norm{\cdot}_U^2$, take the supremum up to $\tau^{M,t}_n$, take expectation with the Burkholder-Davis-Gundy Inequality, apply the above inequalities and use the boundedness of the generic $c_s$ to arrive at the inequality
\begin{align*}
    \mathbb{E}\Bigg[\norm{\sy^{n}}^2_{UH,\tau^{M,t}_n \wedge S} &- \norm{\sy^n_0}_{U}^2 \Bigg]\\ &\leq c\mathbb{E}\int_0^SK(\tilde{\sy}^n_r)\left[1 + \norm{\tilde{\sy}^n_r}_H^2\right]dr + c\mathbbm{E}\left(\int_0^SK(\tilde{\sy}^n_r) \left[1 + \norm{\tilde{\sy}^n_r}_H^4\right]dr\right)^{\frac{1}{2}}\\
    &\leq c\mathbb{E}\int_0^S 1 + \norm{\tilde{\sy}^n_r}_H^2dr + c\mathbbm{E}\left(\int_0^S 1dr\right)^{\frac{1}{2}} + c\mathbb{E}\left(\int_0^S \norm{\tilde{\sy}^n_r}_H^4dr\right)^{\frac{1}{2}}\\
    &\leq c\int_0^S 1 + C dr + c\left(\int_0^S 1dr\right)^{\frac{1}{2}} + c\left[\mathbb{E}\sup_{r\in[0,t]}\norm{\tilde{\sy}^n_r}_H^2\right]^{\frac{1}{2}}\left[\mathbb{E}\int_0^S \norm{\tilde{\sy}^n_r}_H^2dr\right]^{\frac{1}{2}}\\
    &\leq c\int_0^S 1 + Cdr + c\left(\int_0^S 1dr\right)^{\frac{1}{2}} + c\left[C\right]^{\frac{1}{2}}\left[\int_0^S Cdr\right]^{\frac{1}{2}}\\
\end{align*}
using again the notation (\ref{tildesynotation}) and employing the bound (\ref{galerkinboundsatisfiedbystoppingtime2}), and applying (\ref{secondresultofuniformbounds}) with the constant $C$ coming from there. Absorbing this $C$ into our generic $c$, we have that $$\mathbb{E}\Bigg[\norm{\sy^{n}}^2_{UH,\tau^{M,t}_n \wedge S} - \norm{\sy^n_0}_{U}^2 \Bigg] \leq cS + cS^{\frac{1}{2}}.$$ This bound is of course independent of $n$, so (\ref{sufficient thing in probability condition}) follows and hence the result.

\end{proof}

\subsection{Useful Results} \label{subsection useful results}
We state some key theorems used throughout the paper. Firstly is a well-posedness result for an SPDE in a finite dimensional Hilbert Space (though driven by a Cylindrical Brownian Motion) in the standard case of Lipschitz and linear growth constraints.
\begin{proposition} \label{Skorotheorem}
Fix a finite-dimensional Hilbert Space $\mathcal{H}$. Suppose the following:
\begin{itemize}
    \item[1:] For any $T>0$, the operators $\mathscr{A}:[0,T] \times \mathcal{H} \rightarrow \mathcal{H}$ and $\mathscr{G}:[0,T] \times \mathcal{H} \rightarrow \mathscr{L}^2(\mathfrak{U};\mathcal{H})$ are measurable;\\
    \item[2:] There exists a $C_{\cdot}:[0,\infty) \rightarrow \R$ bounded on $[0,T]$ for every $T$, and constants $c_i$ such that for every $\boldsymbol{\phi}, \boldsymbol{\psi} \in \mathcal{H}$ and $t \in [0,\infty)$, \begin{align*}\norm{\mathscr{A}(t,\boldsymbol{\phi})}^2_{\mathcal{H}}  &\leq C_t\left[1 + \norm{\boldsymbol{\phi}}_{\mathcal{H}}^2\right]\\
     \norm{\mathscr{G}_i(t,\boldsymbol{\phi})}^2_{\mathcal{H}} &\leq C_tc_i\left[1 + \norm{\boldsymbol{\phi}}_{\mathcal{H}}^2\right]\\
     \sum_{i=1}^\infty c_i &< \infty\\
    \norm{\mathscr{A}(t,\boldsymbol{\phi}) - \mathscr{A}(t,\boldsymbol{\psi})}^2_{\mathcal{H}} &+\sum_{i=1}^\infty \norm{\mathscr{G}_i(t,\boldsymbol{\phi}) - \mathscr{G}_i(t,\boldsymbol{\psi})}^2_{\mathcal{H}} \leq C_t \norm{\boldsymbol{\phi}-\boldsymbol{\psi}}_{\mathcal{H}}^2
    \end{align*}
    \item[3:] $\py_0 \in L^2(\Omega;\mathcal{H})$.\\
\end{itemize}
Then there exists a process $\py:[0,\infty) \times \Omega \rightarrow \mathcal{H}$ such that for $\mathbbm{P}-a.e.$ $\omega$, $\py_\cdot(\omega) \in C\left([0,T];\mathcal{H}\right)$ for every $T>0$, $\py$ is progressively measurable in $\mathcal{H}$ and the identity \begin{equation}\label{identityingalerkinsolution}\py_t = \py_0 + \int_0^t \mathscr{A}(s,\py_s)ds + \int_0^t \mathscr{G}(s,\py_s)d\mathcal{W}_s\end{equation} holds $\mathbbm{P}-a.s.$ in $\mathcal{H}$ for every $t \geq 0$.
\end{proposition}

\begin{proof}
See \cite{goodair2022stochastic} Theorem 2.8.1.

\end{proof}

We shall also make use of the following 'Stochastic Gronwall Lemma'. 

\begin{lemma}[Stochastic Gr\"{o}nwall] \label{gronny}
Fix $t>0$ and suppose that $\boldsymbol{\phi},\boldsymbol{\psi}, \boldsymbol{\eta}$ are real-valued, non-negative stochastic processes. Assume, moreover, that there exists constants $c',\hat{c}, \tilde{c}$ (allowed to depend on $t$) such that for $\mathbbm{P}-a.e.$ $\omega$, \begin{equation} \label{boundingronny} \int_0^t\boldsymbol{\eta}_s(\omega) ds \leq c'\end{equation} and for all stopping times $0 \leq \theta_j < \theta_k \leq t$,
$$\mathbbm{E}\sup_{r \in [\theta_j,\theta_k]}\boldsymbol{\phi}_r + \mathbbm{E}\int_{\theta_j}^{\theta_k}\boldsymbol{\psi}_sds \leq \hat{c}\mathbbm{E}\left(\left[\boldsymbol{\phi}_{\theta_j} + \tilde{c} \right] + \int_{\theta_j}^{\theta_k} \boldsymbol{\eta}_s\boldsymbol{\phi}_sds\right) < \infty. $$Then there exists a constant $C$ dependent only on $c',\hat{c},\tilde{c},t$ such that $$\mathbbm{E}\sup_{r \in [0,t]}\boldsymbol{\phi}_r + \mathbbm{E}\int_{0}^{t}\boldsymbol{\psi}_sds \leq C\left[\mathbbm{E}(\boldsymbol{\phi}_{0}) + \tilde{c}\right].$$
\end{lemma}

\begin{proof}
See \cite{glatt2009strong} Lemma 5.3.
\end{proof}

The subsequent Lemma is the critical one in \cite{glatt2009strong} which gives the existence of a limiting process and stopping time from the Galerkin Scheme. 

\begin{lemma}[Convergence of Random Cauchy Sequences]\label{greenlemma}
Let $\mathcal{H}_1 \subset \mathcal{H}_2$ be Hilbert Spaces with continuous embedding, and $(\sy^n)$ be a sequence of processes such that for $\mathbb{P}-a.e.$ $\omega$, $\sy^n(\omega) \in C\left([0,T];\mathcal{H}_2\right) \cap L^2\left([0,T];\mathcal{H}_1\right)$ which is a Banach Space with norm $$\norm{\boldsymbol{\psi}}_{X(T)}:= \left(\sup_{r \in [0,T]}\norm{\boldsymbol{\psi}_{r}}^2_{\mathcal{H}_2} + \int_0^T\norm{\boldsymbol{\psi}_{r}}^2_{\mathcal{H}_1}dr\right)^{\frac{1}{2}}.$$ For some fixed $M > 1$ and $t > 0$ define the stopping times $$\tau^{M,t}_n(\omega) := t \wedge \inf\left\{s \geq 0: \norm{\sy^n(\omega)}_{X(s)}^2 \geq M + \norm{\sy^n_0(\omega)}_{\mathcal{H}_2}^2 \right\}$$ and suppose that $$\lim_{m \rightarrow \infty}\sup_{n \geq m}\mathbb{E}\norm{\sy^n-\sy^m}_{X(\tau^{M,t}_m \wedge \tau^{M,t}_n)}^2 = 0$$ and $$\lim_{S \rightarrow 0}\sup_{n \in \mathbb{N}}\mathbb{P}\left(\left\{ \norm{\sy^n}_{X(\tau^{M,t}_n \wedge S)}^2 \geq M-1+\norm{\sy^n_0}_{\mathcal{H}_2}^2 \right\}\right) = 0.$$
Then there exists a stopping time $\tau^{M,t}_{\infty}$, a subsequence $(\sy^{n_l})$ and process $\sy = \sy_{
\cdot \wedge \tau^{M,t}_{\infty}}$ such that:

\begin{itemize}
    \item $\mathbb{P}\left(\left\{ 0 < \tau^{M,t}_{\infty} \leq \tau^{M,t}_{n_l}\right)\right\} = 1$;
    \item For $\mathbb{P}-a.e.$ $\omega$, $\sy(\omega) \in C\left([0,\tau^{M,t}_{\infty}(\omega)];\mathcal{H}_2\right) \cap L^2\left([0,\tau^{M,t}_{\infty}(\omega)];\mathcal{H}_1\right)$;
    \item For $\mathbb{P}-a.e.$ $\omega$, $\sy^{n_l}(\omega) \rightarrow \sy(\omega)$ in $\left( C\left([0,\tau^{M,t}_{\infty}(\omega)];\mathcal{H}_2\right) \cap L^2\left([0,\tau^{M,t}_{\infty}(\omega)];\mathcal{H}_1\right), \norm{\cdot}_{X(\tau^{M,t}_{\infty}(\omega))} \right).$
\end{itemize}
\end{lemma}

\begin{proof}
See \cite{glatt2009strong} Lemma 5.1.
\end{proof}

Lemma \ref{extension lemma} tells us that any local strong solution can be extended up to a strictly greater stopping time. 

\begin{lemma} \label{extension lemma}
Suppose that $(\sy,\tau)$ is a local strong solution of (\ref{thespde}) in the sense of \ref{definitionofregularsolution}, \ref{definitionofHsolution}  and that $\tau$ is $\mathbbm{P}-a.s.$ finite. Then there exists a local strong solution $(\py,\sigma)$ such that $\sigma > \tau$ $\mathbbm{P}-a.s.$ and for all $t \in [0,\infty)$, $$\mathbbm{P}\left(\left\{\omega \in \Omega: \sy_{t \wedge \tau(\omega)}(\omega) =  \py_{t \wedge \tau(\omega)}  \right\} \right) = 1.$$
\end{lemma}

\begin{proof}
See \cite{glatt2009strong} Lemma 4.1. Though we have to note that it is on this instance only shown for the Stochastic Navier-Stokes Equation, the arguments are general and undergo no changes for the solutions as defined in \ref{definitionofregularsolution}, \ref{definitionofHsolution}.  
\end{proof}

\begin{proposition} \label{rockner prop}
Let $\mathcal{H}_1 \subset \mathcal{H}_2 \subset \mathcal{H}_3$ be a triplet of embedded Hilbert Spaces where $\mathcal{H}_1$ is dense in $\mathcal{H}_2$, with the property that there exists a continuous nondegenerate bilinear form $\inner{\cdot}{\cdot}_{\mathcal{H}_3 \times \mathcal{H}_1}: \mathcal{H}_3 \times \mathcal{H}_1 \rightarrow \R$ such that for $\phi \in \mathcal{H}_2$ and $\psi \in \mathcal{H}_1$, $$\inner{\phi}{\psi}_{\mathcal{H}_3 \times \mathcal{H}_1} = \inner{\phi}{\psi}_{\mathcal{H}_2}.$$ Suppose that for some $T > 0$ and stopping time $\tau$,
\begin{enumerate}
        \item $\sy_0:\Omega \rightarrow \mathcal{H}_2$ is $\mathcal{F}_0-$measurable;
        \item $\eta:\Omega \times [0,T] \rightarrow \mathcal{H}_3$ is such that for $\mathbbm{P}-a.e.$ $\omega$, $\eta(\omega) \in L^2([0,T];\mathcal{H}_3)$;
        \item $B:\Omega \times [0,T] \rightarrow \mathscr{L}^2(\mathfrak{U};\mathcal{H}_2)$ is progressively measurable and such that for $\mathbbm{P}-a.e.$ $\omega$, $B(\omega) \in L^2\left([0,T];\mathscr{L}^2(\mathfrak{U};\mathcal{H}_2)\right)$;
        \item  \label{4*} $\sy:\Omega \times [0,T] \rightarrow \mathcal{H}_1$ is such that for $\mathbbm{P}-a.e.$ $\omega$, $\sy_{\cdot}(\omega)\mathbbm{1}_{\cdot \leq \tau(\omega)} \in L^2([0,T];\mathcal{H}_1)$ and $\sy_{\cdot}\mathbbm{1}_{\cdot \leq \tau}$ is progressively measurable in $\mathcal{H}_1$;
        \item \label{item 5 again*} The identity
        \begin{equation} \label{newest identity*}
            \sy_t = \sy_0 + \int_0^{t \wedge \tau}\eta_sds + \int_0^{t \wedge \tau}B_s d\mathcal{W}_s
        \end{equation}
        holds $\mathbbm{P}-a.s.$ in $\mathcal{H}_3$ for all $t \in [0,T]$.
    \end{enumerate}
The the equality 
  \begin{align} \label{ito big dog*}\norm{\sy_t}^2_{H} = \norm{\sy_0}^2_{H} + \int_0^{t\wedge \tau} \bigg( 2\inner{\eta_s}{\sy_s}_{U \times V} + \norm{B_s}^2_{\mathscr{L}^2(\mathcal{U};H)}\bigg)ds + 2\int_0^{t \wedge \tau}\inner{B_s}{\sy_s}_{H}d\mathcal{W}_s\end{align}
  holds for any $t \in [0,T]$, $\mathbbm{P}-a.s.$ in $\R$. Moreover for $\mathbbm{P}-a.e.$ $\omega$, $\sy_{\cdot}(\omega) \in C([0,T];\mathcal{H}_2)$. 
\end{proposition}

\begin{proof}
This is a minor extension of \cite{liu2015stochastic} Theorem 4.2.5, which is stated and justified as Proposition 2.5.5. in \cite{goodair2022stochastic}. The extension is primarily from the Gelfand Triple to this setting.
\end{proof}

\textbf{Thanks:} The authors would like to thank Romeo Mensah for his contribution to the project. We would also like to thank James-Michael Leahy, Arnaud Debussche, Etienne M\'{e}min, Darryl Holm and Fang Rui Lim for valuable discussions surrounding the project.\\

\textbf{Acknowledgements:} Daniel Goodair was supported by the Engineering and Physical Sciences Research Council (EPSCR) Project 2478902. Dan Crisan and Oana Lang were partially supported by the European Research Council (ERC) under the European Union's Horizon 2020 Research and Innovation Programme (ERC, Grant Agreement No 856408).\\

\textbf{Data availability:} Data sharing is not applicable to this article as no datasets were generated or analysed during the current study.\\

\textbf{Conflict of Interest:} On behalf of all authors, the corresponding author states that there is no conflict of interest. 

\bibliographystyle{spmpsci}
\bibliography{myBibliography}

\begin{thebibliography}{10}
\providecommand{\url}[1]{{#1}}
\providecommand{\urlprefix}{URL }
\expandafter\ifx\csname urlstyle\endcsname\relax
  \providecommand{\doi}[1]{DOI~\discretionary{}{}{}#1}\else
  \providecommand{\doi}{DOI~\discretionary{}{}{}\begingroup
  \urlstyle{rm}\Url}\fi

\bibitem{alonso2020well}
Alonso-Or{\'a}n, D., Bethencourt~de Le{\'o}n, A.: On the well-posedness of
  stochastic boussinesq equations with transport noise.
\newblock Journal of Nonlinear Science \textbf{30}(1), 175--224 (2020)

\bibitem{alonso2020modelling}
Alonso-Or{\'a}n, D., Bethencourt~de Le{\'o}n, A., Holm, D.D., Takao, S.:
  Modelling the climate and weather of a 2d lagrangian-averaged
  euler--boussinesq equation with transport noise.
\newblock Journal of Statistical Physics \textbf{179}(5), 1267--1303 (2020)

\bibitem{attanasio2011renormalized}
Attanasio, S., Flandoli, F.: Renormalized solutions for stochastic transport
  equations and the regularization by bilinear multiplicative noise.
\newblock Communications in Partial Differential Equations \textbf{36}(8),
  1455--1474 (2011)

\bibitem{brzezniak2021well}
Brze{\'z}niak, Z., Slavik, J.: Well-posedness of the 3d stochastic primitive
  equations with multiplicative and transport noise.
\newblock Journal of Differential Equations \textbf{296}, 617--676 (2021)

\bibitem{cotter2020data}
Cotter, C., Crisan, D., Holm, D., Pan, W., Shevchenko, I.: Data assimilation
  for a quasi-geostrophic model with circulation-preserving stochastic
  transport noise.
\newblock Journal of Statistical Physics \textbf{179}(5), 1186--1221 (2020)

\bibitem{cotter2018modelling}
Cotter, C., Crisan, D., Holm, D.D., Pan, W., Shevchenko, I.: Modelling
  uncertainty using stochastic transport noise in a 2-layer quasi-geostrophic
  model.
\newblock arXiv preprint arXiv:1802.05711  (2018)

\bibitem{cotter2019numerically}
Cotter, C., Crisan, D., Holm, D.D., Pan, W., Shevchenko, I.: Numerically
  modeling stochastic lie transport in fluid dynamics.
\newblock Multiscale Modeling \& Simulation \textbf{17}(1), 192--232 (2019)

\bibitem{crisan2019solution}
Crisan, D., Flandoli, F., Holm, D.D.: Solution properties of a 3d stochastic
  euler fluid equation.
\newblock Journal of Nonlinear Science \textbf{29}(3), 813--870 (2019)

\bibitem{crisan2021theoretical}
Crisan, D., Holm, D.D., Luesink, E., Mensah, P.R., Pan, W.: Theoretical and
  computational analysis of the thermal quasi-geostrophic model.
\newblock arXiv preprint arXiv:2106.14850  (2021)

\bibitem{crisan2019well}
Crisan, D., Lang, O.: Well-posedness for a stochastic 2d euler equation with
  transport noise.
\newblock arXiv preprint arXiv:1907.00451  (2019)

\bibitem{crisan2020local}
Crisan, D., Lang, O.: Local well-posedness for the great lake equation with
  transport noise.
\newblock arXiv preprint arXiv:2003.03357  (2020)

\bibitem{crisan2021well}
Crisan, D., Lang, O.: Well-posedness properties for a stochastic rotating
  shallow water model.
\newblock arXiv preprint arXiv:2107.06601  (2021)

\bibitem{debussche2011local}
Debussche, A., Glatt-Holtz, N., Temam, R.: Local martingale and pathwise
  solutions for an abstract fluids model.
\newblock Phys. D \textbf{240}(14), 1123--1144 (2011)

\bibitem{doob2012measure}
Doob, J.L.: Measure theory, vol. 143.
\newblock Springer Science \& Business Media (2012)

\bibitem{dufee2022stochastic}
Duf{\'e}e, B., M{\'e}min, E., Crisan, D.: Stochastic parametrization: an
  alternative to inflation in ensemble kalman filters.
\newblock Quarterly Journal of the Royal Meteorological Society
  \textbf{148}(744), 1075--1091 (2022)

\bibitem{enciso2018biot}
Enciso, A., Garcia-Ferrero, M.A., Peralta-Salas, D.: The biot--savart operator
  of a bounded domain.
\newblock Journal de Math{\'e}matiques Pures et Appliqu{\'e}es \textbf{119},
  85--113 (2018)

\bibitem{flandoli2015open}
Flandoli, F.: An open problem in the theory of regularization by noise for
  nonlinear pdes.
\newblock In: Workshop Classic and Stochastic Geometric Mechanics, pp. 13--29.
  Springer (2015)

\bibitem{flandoli2021scaling}
Flandoli, F., Galeati, L., Luo, D.: Scaling limit of stochastic 2d euler
  equations with transport noises to the deterministic navier--stokes
  equations.
\newblock Journal of Evolution Equations \textbf{21}(1), 567--600 (2021)

\bibitem{flandoli2021high}
Flandoli, F., Luo, D.: High mode transport noise improves vorticity blow-up
  control in 3d navier--stokes equations.
\newblock Probability Theory and Related Fields \textbf{180}(1), 309--363
  (2021)

\bibitem{flandoli20212d}
Flandoli, F., Pappalettera, U.: 2d euler equations with stratonovich transport
  noise as a large-scale stochastic model reduction.
\newblock Journal of Nonlinear Science \textbf{31}(1), 1--38 (2021)

\bibitem{flandoli2022additive}
Flandoli, F., Pappalettera, U.: From additive to transport noise in 2d fluid
  dynamics.
\newblock Stochastics and Partial Differential Equations: Analysis and
  Computations pp. 1--41 (2022)

\bibitem{glatt2011cauchy}
Glatt-Holtz, N., Temam, R.: Cauchy convergence schemes for some nonlinear
  partial differential equations.
\newblock Applicable Analysis \textbf{90}(1), 85--102 (2011)

\bibitem{glatt2009strong}
Glatt-Holtz, N., Ziane, M., et~al.: Strong pathwise solutions of the stochastic
  navier-stokes system.
\newblock Advances in Differential Equations \textbf{14}(5/6), 567--600 (2009)

\bibitem{glatt2012local}
Glatt-Holtz, N.E., Vicol, V.C.: Local and global existence of smooth solutions
  for the stochastic {E}uler equations with multiplicative noise.
\newblock Ann. Probab. \textbf{42}(1), 80--145 (2014).
\newblock \doi{10.1214/12-AOP773}.
\newblock
  \urlprefix\url{https://doi-org.univaq.clas.cineca.it/10.1214/12-AOP773}

\bibitem{goodair2022existence}
Goodair, D.: Existence and uniqueness of maximal solutions to a 3d
  navier-stokes equation with stochastic lie transport.
\newblock arXiv preprint arXiv:2202.09242  (2022)

\bibitem{goodair2022stochastic}
Goodair, D.: Stochastic calculus in infinite dimensions and spdes.
\newblock arXiv preprint arXiv:2203.17206  (2022)

\bibitem{goodair2022inprep}
Goodair, D., Crisan, D.: On the navier-stokes equations with stochastic lie
  transport.
\newblock In preparation  (2022)

\bibitem{gyongy1980stochastic}
Gy{\"o}ngy, I., Krylov, N.V.: On stochastic equations with respect to
  semimartingales i.
\newblock Stochastics: An International Journal of Probability and Stochastic
  Processes \textbf{4}(1), 1--21 (1980)

\bibitem{hairer2009introduction}
Hairer, M.: An introduction to stochastic pdes.
\newblock arXiv preprint arXiv:0907.4178  (2009)

\bibitem{holm2015variational}
Holm, D.D.: Variational principles for stochastic fluid dynamics.
\newblock Proceedings of the Royal Society A: Mathematical, Physical and
  Engineering Sciences \textbf{471}(2176), 20140,963 (2015)

\bibitem{holm2019stochastic}
Holm, D.D., Luesink, E.: Stochastic wave--current interaction in thermal
  shallow water dynamics.
\newblock Journal of Nonlinear Science \textbf{31}(2), 1--56 (2021)

\bibitem{holm2020stochastic}
Holm, D.D., Luesink, E., Pan, W.: Stochastic circulation dynamics in the ocean
  mixed layer.
\newblock arXiv preprint arXiv:2006.05707  (2020)

\bibitem{kato1984nonlinear}
Kato, T., Lai, C.Y.: Nonlinear evolution equations and the euler flow.
\newblock Journal of functional analysis \textbf{56}(1), 15--28 (1984)

\bibitem{kelliher2006navier}
Kelliher, J.P.: Navier--stokes equations with navier boundary conditions for a
  bounded domain in the plane.
\newblock SIAM journal on mathematical analysis \textbf{38}(1), 210--232 (2006)

\bibitem{krylov2007stochastic}
Krylov, N.V., Rozovskii, B.L.: Stochastic evolution equations.
\newblock In: Stochastic Differential Equations: Theory And Applications: A
  Volume in Honor of Professor Boris L Rozovskii, pp. 1--69. World Scientific
  (2007)

\bibitem{lang2022pathwise}
Lang, O., Pan, W.: A pathwise parameterisation for stochastic transport.
\newblock arXiv preprint arXiv:2202.10852  (2022)

\bibitem{van2021bayesian}
van Leeuwen, P.J., Crisan, D., Lang, O., Potthast, R.: Bayesian inference for
  fluid dynamics: A case study for the stochastic rotating shallow water model.
\newblock arXiv preprint arXiv:2112.15216  (2021)

\bibitem{liu2013well}
Liu, W.: Well-posedness of stochastic partial differential equations with
  lyapunov condition.
\newblock Journal of Differential Equations \textbf{255}(3), 572--592 (2013)

\bibitem{liu2010spde}
Liu, W., R{\"o}ckner, M.: Spde in hilbert space with locally monotone
  coefficients.
\newblock Journal of Functional Analysis \textbf{259}(11), 2902--2922 (2010)

\bibitem{liu2013local}
Liu, W., R{\"o}ckner, M.: Local and global well-posedness of spde with
  generalized coercivity conditions.
\newblock Journal of differential equations \textbf{254}(2), 725--755 (2013)

\bibitem{liu2015stochastic}
Liu, W., R{\"o}ckner, M.: Stochastic partial differential equations: an
  introduction.
\newblock Springer (2015)

\bibitem{luo2021convergence}
Luo, D.: Convergence of stochastic 2d inviscid boussinesq equations with
  transport noise to a deterministic viscous system.
\newblock Nonlinearity \textbf{34}(12), 8311 (2021)

\bibitem{luo2020scaling}
Luo, D., Saal, M.: A scaling limit for the stochastic msqg equations with
  multiplicative transport noises.
\newblock Stochastics and Dynamics \textbf{20}(06), 2040,001 (2020)

\bibitem{memin2014fluid}
M{\'e}min, E.: Fluid flow dynamics under location uncertainty.
\newblock Geophysical \& Astrophysical Fluid Dynamics \textbf{108}(2), 119--146
  (2014)

\bibitem{neelima2020coercivity}
Neelima, {\v{S}}i{\v{s}}ka, D.: Coercivity condition for higher moment a priori
  estimates for nonlinear spdes and existence of a solution under local
  monotonicity.
\newblock Stochastics \textbf{92}(5), 684--715 (2020)

\bibitem{pardoux1975equations}
Pardoux, E.: Equations aux d{\'e}riv{\'e}es partielles stochastiques monotones,
  these, univ (1975)

\bibitem{robinson2016three}
Robinson, J.C., Rodrigo, J.L., Sadowski, W.: The three-dimensional
  Navier--Stokes equations: Classical theory, vol. 157.
\newblock Cambridge university press (2016)

\bibitem{rockner2022well}
R{\"o}ckner, M., Shang, S., Zhang, T.: Well-posedness of stochastic partial
  differential equations with fully local monotone coefficients.
\newblock arXiv preprint arXiv:2206.01107  (2022)

\bibitem{street2021semi}
Street, O.D., Crisan, D.: Semi-martingale driven variational principles.
\newblock Proceedings of the Royal Society A \textbf{477}(2247), 20200,957
  (2021)

\end{thebibliography}

\end{document}